\documentclass[11pt]{article}

\usepackage[centertags]{amsmath}
\usepackage[english]{babel}
\usepackage{amsfonts}
\usepackage{amssymb,amssymb,amscd}
\usepackage{amsthm}
\usepackage{a4wide}

\newtheorem{dfn}{Definition}
\newtheorem{exam}{Example}
\newtheorem{rem}{Remark}
\newtheorem{rems}[rem]{Remarks}
\newtheorem{lemma}{Lemma}
\newtheorem{proposition}{Proposition}
\newtheorem{corollary}{Corollary}
\newtheorem{theorem}{Theorem}

\newcommand{\Stab}{{\rm Stab}}
\newcommand{\val}{{\rm val}}

\newcommand{\BS}{\mathord{\mathsf {BS}}}
\def\mapright#1{\smash{\mathop{\longrightarrow}\limits^{#1}}}

\newcommand{\semi}{{\,\rule[.1pt]{.4pt}{5.3pt}\hskip-1.9pt\times}}
\newcommand{\la}{\langle}
\newcommand{\ra}{\rangle}

\newcommand{\al}{\alpha}
\newcommand{\be}{\beta}
\newcommand{\alvee}{{\al^\vee}}

\newcommand{\de}{\delta}
\newcommand{\supp}{\text{supp}\,}

\newcommand{\om}{\omega}

\newcommand{\A}{\mathbb A}

\newcommand{\Xvee}{{X^\vee}}
\newcommand{\Xveep}{{X^\vee_+}}
\newcommand{\Phivee}{{\Phi^\vee}}
\newcommand{\Phiveep}{{\Phi^\vee_+}}

\newcommand{\Rvee}{{R^\vee}}
\newcommand{\Rveep}{{R^\vee_+}}
\newcommand{\Delvee}{\Delta^\vee}
\newcommand{\Waff}{{W^{{\mathfrak a}}}}
\newcommand{\sfrak}{{{\mathfrak s}}}

\newcommand{\lam}{{\lambda}}

\newcommand{\cs}{\mathcal{S}}

\newcommand{\co}{\mathcal{O}}
\newcommand{\ck}{\mathcal{K}}
\newcommand{\cg}{\mathcal{G}}
\newcommand{\cl}{\mathcal{L}}

\newcommand{\cj}{\mathcal{J}}

\newcommand{\cb}{\mathcal{B}}

\newcommand{\cp}{\mathcal{P}}
\newcommand{\cq}{\mathcal{Q}}

\newcommand{\ch}{{\mathtt{H}}}

\newcommand{\cn}{\mathcal{N}}

\newcommand{\ct}{\mathcal{T}}

\newcommand{\chara}{{\mathop{\rm Char}\,}}
\newcommand{\Ker}{{\mathop{\rm Ker}}}
\newcommand{\df}{{\mathop{\it def}}}
\newcommand{\Lie}{{\mathop{\hbox{\rm Lie}}\,}}

\newcommand{\Mor}{{{\hbox{\rm Mor}}\,}}
\newcommand{\Aut}{{{\hbox{\rm Aut}}\,}}

\newcommand{\bF}{{\mathbb F}}
\newcommand{\br}{{\mathbb R}}
\newcommand{\bz}{{\mathbb Z}}

\newcommand{\bc}{{\mathbb C}}

\newcommand{\La}{{\mathfrak{a}}}
\newcommand{\Lg}{{\mathfrak{g}}}

\newcommand{\Lo}{{\mathfrak{o}}}

\newcount\cols
{\catcode`,=\active\catcode`|=\active
 \gdef\Young(#1){\hbox{$\vcenter
 {\mathcode`,="8000\mathcode`|="8000
  \def,{\global\advance\cols by 1 &}%
  \def|{\cr
        \multispan{\the\cols}\hrulefill\cr
        &\global\cols=2 }%
  \offinterlineskip\everycr{}\tabskip=0pt
  \dimen0=\ht\strutbox \advance\dimen0 by \dp\strutbox
  \halign
   {\vrule height \ht\strutbox depth \dp\strutbox##
    &&\hbox to \dimen0{\hss$##$\hss}\vrule\cr
    \noalign{\hrule}&\global\cols=2 #1\crcr
    \multispan{\the\cols}\hrulefill\cr%
   }
 }$}}
 \gdef\Skew(#1:#2){\hbox{$\vcenter
 {\mathcode`,="8000\mathcode`|="8000
  \dimen0=\ht\strutbox \advance\dimen0 by \dp\strutbox
  \def\boxbeg{\vbox
    \bgroup\hrule\kern-0.4pt\hbox to\dimen0\bgroup\strut\vrule\hss$}%
  \def\boxend{$\hss\egroup\hrule\egroup}%
  \def,{\boxend\boxbeg}%
  \def|##1:{\boxend\vrule\egroup\nointerlineskip\kern-0.4pt
    \moveright##1\dimen0\hbox\bgroup\boxbeg}%
  \def\\##1\\##2:{\boxend\vrule\egroup\nointerlineskip\kern-0.4pt
    \kern ##1\dimen0\moveright##2\dimen0\hbox\bgroup\boxbeg}%
  \moveright#1\dimen0\hbox\bgroup\boxbeg#2\boxend\vrule\egroup
 }$}}
}

\begin{document}
\title{One-skeleton galleries, the path model and a generalization of Macdonald's formula for Hall-Littlewood polynomials}

\author{St\'ephane Gaussent and Peter Littelmann}

\maketitle

\emph{Institut \'Elie Cartan U.M.R. 7502, Universit\'e Henri Poincar\'e Nancy 1, Bo\^ite Postale 70239, F-54506 Vand\oe uvre-l\`es-Nancy CEDEX, France, and\\ 
\indent Mathematisches Institut, Universit\"at zu K\"oln, Weyertal 86-90, D-50931 K\"oln, Germany}

\begin{abstract}
We give a direct geometric interpretation of the path model using galleries in the $1-$skeleton of the Bruhat-Tits building 
associated to a semi-simple algebraic group. This interpretation allows us to compute the coefficients of the expansion 
of the Hall-Littlewood polynomials in the monomial basis. The formula we obtain is a ``geometric compression'' of the 
one proved by Schwer, its specialization to the case ${\tt A}_n$ turns out to be equivalent to Macdonald's formula.
\end{abstract}

\tableofcontents


\section{Introduction}
We give a 
direct geometric interpretation of the path model for representations and the associated 
Weyl group combinatorics \cite{L1}. As a consequence, we get a generalization of Macdonald's formula for Hall-Littelwood polynomials in type ${\tt A}_n$ \cite{Mac}. Our formula can be seen as a geometric
compression of Schwer's formula \cite{Schwer}. 

Concerning the connection with the path model, a first step in this direction was done in \cite{GL}. 
The advantage of the new 
approach is that galleries in the one-skeleton of the apartment can directly be identified
with piecewise linear paths running along the one-skeleton. They can be concatenated and they can also be easily translated in the language of tableaux, for classical groups. The goal now is to show that
the original approach by Lakshmibai, Musili and Seshadri \cite{LMS,LS1} towards what later became the path
model has an intrinsic geometric interpretation in the geometry of the affine Grassmannian,
respectively in the geometry of the associated affine building. Another instance of this approach can be found in the work of Kapovich and Millson \cite{KM} where they use paths in the one-skeleton in their proof of the ``saturation'' theorem.

To be more precise, let $G$ be a semisimple algebraic group
defined over $\bc$, fix a Borel subgroup $B$ and a maximal torus $T$. 
Let $U^-$ be the unipotent radical of the opposite Borel subgroup. Let $\co=\bc[[t]]$ be the ring of complex 
formal power series and let $\ck=\bc(\!(t)\!)$ be the quotient field. For a dominant coweight $\lam$ and an arbitrary
coweight $\mu$ consider the following intersection in the affine Grassmannian $G(\ck)/G(\co)$:
$$
Z_{\lam,\mu}= G(\co).\lam\cap U^-(\ck).\mu.
$$
Let $\mathbb F_q$ be the finite field with $q$ elements and replace 
the field of complex numbers by the algebraic closure $K$ of $\mathbb F_q$. 
Assume that all groups are defined and split over $\mathbb F_q$. Replace
$\ck$ by $\ck_q = \mathbb F_q(\!(t)\!)$  and $\co$ by $\co_q = \mathbb F_q[[t]]$; 
the Laurent polynomials $L_{\lam,\mu}$ defined by $L_{\lam,\mu}(q)=\vert Z_{\lam,\mu}^q\vert$ show 
up as coefficients in the Hall-Littlewood polynomial:
$P_\lam=\sum_{\mu\in\Xveep}q^{-\langle\rho,\lam+\mu\rangle} L_{\lam,\mu} m_\mu.$

We replace the desingularization of the Schubert variety  $X_\lam$ in \cite{GL} by a 
Bott-Samelson type variety $\Sigma$ which is a fibred space having as factors varieties
of the form $H/R$, where $H$ is a semisimple algebraic group and 
$R$ is a maximal parabolic subgroup. In terms of the affine building, a point in this variety
is a sequence $\delta = (P_0=G(\co),Q_0,P_1,Q_1,\ldots,P_r, Q_r, P_{r+1})$ of parahoric subgroups of $G(\ck)$ 
reciprocative contained in each other, i.e. $G(\co)\supset Q_0\subset P_1\supset Q_1\subset\ldots \supset Q_r\subset P_{r+1}$. These desingularizations are smaller than the ones used in \cite{GL}, in the sense that the fibres are of smaller dimensions. In type ${\tt A}_n$, these coincide with convolution morphisms.

In terms of the faces of the building, a point in $\Sigma$ is a sequence of closed one-dimensional faces
(corresponding to the parahoric subgroups $Q_0,\ldots,Q_r$), where successive faces have (at least) a common 
zero-dimensional face (i.e. a vertex corresponding to one of the maximal parahoric subgroups $P_0,\ldots,P_{r+1}$). 
So if the sequence is contained in an apartment, then the point in $\Sigma$ can be seen as a piecewise linear path in the apartment joining the origin with a special vertex. 

We introduce the notion of a minimal one-skeleton gallery (which always lies in some apartment)
and of a positively folded combinatorial gallery in the one-skeleton. 
The points in $\Sigma$ corresponding to the points in the open orbit $G(\co).\lam\subset X_\lam$
are exactly the minimal galleries, 
we identify those two sets. By choosing a generic 
one parameter subgroup of $T$ in the anti-dominant Weyl chamber, we get a Bia\l ynicki-Birula
decomposition of $\Sigma$, the centers $\delta$ of the cells $C_\delta$ correspond to combinatorial 
one-skeleton galleries $\delta$ (i.e. the galleries lying in the standard apartment). We
show that $C_\delta\cap G(\co).\lam\not=\emptyset$ if and only if $\delta$ is positively folded.

The Bia\l ynicki-Birula decomposition of $\Sigma$ can be used to define a decomposition 
$Z_{\lam,\mu}=\bigcup_\delta Z_{\lam,\mu}\cap C_{\delta}$, the indexing set of the strata are positively folded one-skeleton galleries. To see
the {\it geometrical compression} (compare Lenart \cite{Len}, \cite{Len2}), recall the decomposition in \cite{GL}, $Z_{\lam,\mu}^q=\bigcup S_\Delta$, where the $\Delta$'s are certain galleries of alcoves of a fixed type in the appartment. Now, fix a minimal gallery of alcoves $\Delta_\lambda$ between $0$ and $\lambda$ and a minimal one-skeleton gallery $\gamma_\lambda$ contained in $\Delta_\lambda$. This allows to build a map from the galleries of alcoves in the standard apartment to the one-skeleton ones staying in this apartment. This map  sends positively folded galleries of alcoves onto positively folded one-skeleton galleries. Note that this application [even when restricted to positively folded galleries] is surjective but not injective. Further, the pieces $S_\Delta$ group together to build the pieces $Z_{\lambda,\mu}\cap C_\delta$ of our new decomposition of $Z_{\lambda,\mu}$.

For example, in the ${\tt A}_n$-case the galleries can be translated into the language of Young tableaux, 
and the positively folded galleries ending in $\mu$ correspond exactly to the semi-standard Young tableaux of shape 
$\lam$ and weight $\mu$. In this sense, the new decomposition can be viewed as the optimal geometric decomposition
for type ${\tt A}_n$. 

To investigate the intersection $Z_{\lam,\mu}\cap C_{\delta}$ we need to {\it unfold} the 
(possibly) folded gallery $\delta$.
As a consequence of the unfolding procedure we present the formula for the coefficients
of the Hall-Littlewood polynomials, the summands below counting the number of points in the intersection of $Z_{\lam,\mu}^q\cap C_{\delta}$ for $\delta$ being positively folded and ending in $\mu$:
\par\vskip 5pt\noindent
{\bf Theorem~\ref{Lpolynomialformula}.}
\par\noindent
$$
L_{\lambda,\mu} (q) = \sum_{\delta\in\Gamma^+(\gamma_\lambda, \mu)} q^{\ell(w_{D_0})} \prod_{j=1}^r \ U_i(q),
$$ where $U_i(q)$ is a polynomial of the form $\sum_{\mathbf c\in \Gamma_{\sfrak^j_{V_j}}^+(\mathbf i_j, op)} q^{t(\mathbf c)} (q-1)^{r(\mathbf c)}$ that counts the number of points over $\mathbb F_q$ in a subvariety of a generalised Grassmannian $H/R$, where $H$ and $R$ are determined by $\delta$.

\par\vskip5pt
To get a rough idea of what this formula means without getting drowned by the technical details, let us consider
the case where $G$ is of type ${\tt A}_n$. We identify the positively folded galleries with the semi-standard Young 
tableaux of shape $\lam$ having weight $\mu$.
We use the convention that the entries in the tableau are weakly increasing in the rows and
stricly increasing in the columns, so the one dimensional faces of the gallery correspond to the 
columns of the tableau. We enumerate the columns such that the right most column is the first one.
Given such a tableau $\delta$, let $E_0,\ldots,E_r$ be the columns. We want
to investigate the set of all minimal galleries in $C_\delta$ lying in $Z_{\lam,\mu}$. 
Proposition~\ref{prMinGlobal} shows that this set has a product structure 
$$
B^- w_{D_0} Q^-_{E_0}/Q^-_{E_0}\times \prod_{j=1}^{r} Min(E_{j-1},E_j)\ ,
$$
which explains the product structure for each summand in Theorem \ref{Lpolynomialformula}. To get a minimal gallery in the building 
that lifts $\delta$, i.e. is an element of $C_\delta$ and lies in $Z_{\lam,\mu}$, the possibilities 
for the first column $E_0$ form a Schubert cell leading to the term $q^{\ell(w_{D_0})}$. For $j\ge 1$, the 
possibilities for lifts of $E_j$ depend on the column $E_{j-1}$ before. 
It can be shown that Macdonald's algorithm (see \cite{Mac}) can be expressed also columnwise.
More precisely, Klostermann \cite{Kl} has shown in the framework of her thesis that the structure of the 
second sum in the formula above in Theorem~\ref{Lpolynomialformula} can be simplified in the ${\tt A}_n$-case 
so that,  in terms of Young tableaux, the resulting algorithm is exactly the same as Macdonald's algorithm.  

The positively folded one-skeleton galleries having $q^{\langle\lam+\mu,\rho\rangle}$ as a leading
term in the counting formula for $\vert Z_{\lam,\mu}^q\cap C_{\delta}\vert$, 
are called {\it LS-galleries}; this is an abbreviation for Lakshmibai-Seshadri galleries. As in \cite{GL}, for an LS-gallery $\delta$, $\overline{Z_{\lambda,\mu}\cap C_\delta}$ is an MV-cycle.

In section~\ref{LSAndYoungTableaux} we discuss the special role of the LS-galleries
and the connection with the indexing system by generalized Young tableaux
introduced by Lakshmibai, Musili and Seshadri in a series of papers, see for example
\cite{LMS, LS1, LS2}. Recall that these papers were the background for the path model theory started in \cite{L1}.
An important notion introduced in the theory of standard monomials is the 
defining chain (\cite{LMS, LS1}, see also section~\ref{seMinimal}), which was a 
breakthrough on the way for the definition of standard monomials and generalized Young tableaux.
In the context of the crystal structure of the path theory this notion again turned up to be 
an important combinatorial tool to check whether a concatenation of paths is in the Cartan component or not.
Still, the definition had the air of an ad hoc combinatorial tool. But in the context of
Bia\l ynicki-Birula cells, the folding of a minimal gallery by the action of the torus occurs
naturally: during the limit process (going to the center of the cell) the direction 
(= the sector, see section~\ref{seMinimal}) attached to a minimal gallery
is transformed into the weakly decreasing sequence of Weyl group
elements, the defining chain for the positively folded one-skeleton gallery in the center of the cell. 

The connection between the path model theory and the one-skeleton galleries
is summarized in the following corollary. For a fundamental coweight $\om$
let $\pi_{\om_{i}}:[0,1]\rightarrow X^\vee_\br$, $t\mapsto t\om$ be the path which is just the straight line 
joining $\Lo$ with $\om$ and let $\gamma_{\om}$ be the one-skeleton gallery
obtained as the sequence of edges and vertices lying on the path (see also Example~\ref{exGalsimpleCoweight} in Section \ref{suseCombiOneSk}).
\par\vskip5pt\noindent
{\bf Corollary~\ref{gallerpathcoro}.} {\it
Write a dominant coweight $\lam=\om_{i_1}+\ldots+\om_{i_r}$ as a sum of fundamental coweights,
write $\underline{\lam}$ for this ordered decomposition. Let ${\mathcal P}_{\underline \lam}$ be the associated
path model of LS-paths of shape $\underline{\lam}$ defined in \cite{L1} having as starting path the 
concatenation $\pi_{\om_{i_1}}*\ldots*\pi_{\om_{i_r}}$. For a path $\pi$ in the path 
model denote by $\gamma_\pi$ the associated gallery in the one-skeleton of $\A$
obtained as the sequence of edges and vertices lying on the path.
The one-skeleton galleries $\gamma_\pi$ obtained in this way are precisely the LS-galleries of the same 
type as $\gamma_{\om_{i_1}}*\ldots*\gamma_{\om_{i_r}}$. 
}
\par\vskip5pt
In fact, the notion of a {\it defining chain for LS-paths}
coincides in this case with the notion of a defining chain for the associated gallery.

Since the number of the LS-galleries is the coefficient of the leading term of $L_{\lam,\mu}$,
and since $P_\lam\rightarrow s_\lam$ for $q\rightarrow\infty$, we get
as an immediate consequence of Theorem~\ref{Lpolynomialformula} 
the following character formula. In combination with Corollary~\ref{gallerpathcoro}, this provides a 
geometric proof of the path character formula, first conjectured by Lakshmibai 
(see for example \cite{LS2}) and proved in \cite{L1}:
\par\vskip 5pt\noindent
{\bf Corollary~\ref{characterformula}.} {\it $\chara V(\lam) = \sum_{\delta} e^{target(\delta)}$, where the sum runs over all
LS-galleries of the same type as $\gamma_\lam$. 
}

\medskip
The article is organized as follows: In section~\ref{sePrel} we recall 
some basic facts about the affine Grassmannian and Hall-Littlewood polynomials,
in section~\ref{AppartChamberBuilding} we recall the main facts from building theory
needed later. In section~\ref{oneskeletonsection}, we introduce the main object of this article,
the one-skeleton galleries of a fixed type, and its geometric counterpart, the Bott-Samelson
variety $\Sigma$. We give a description of the Bia\l ynicki-Birula cells of $\Sigma$. For groups of type ${\tt A}_n, {\tt B}_n, {\tt C}_n$, we establish a bijection between galleries and tableaux. In section~\ref{seMinimal} we introduce the notion of a minimal one-skeleton gallery
and of a positively folded combinatorial gallery in the one-skeleton. We show that the correspondence between galleries and tableaux restricts to a bijection between positively folded galleries and semistandard tableaux.
In section~\ref{localmin} we unfold the folded galleries locally, in section~\ref{localglobalproperties} we do this 
stepwise for the full gallery and we prove: a cell $C_\delta$ contains minimal galleries if and only if 
$\delta$ is positively folded. In subsection~\ref{formulaL} we present the formula for the coefficients
of the Hall-Littlewood polynomials.
In section~\ref{LSAndYoungTableaux} we discuss the special role of the LS-galleries
and the connection with the indexing system by generalized Young tableaux
introduced by Lakshmibai, Musili and Seshadri.

\section{Preliminaries}
\label{sePrel}
Let $G$ be a connected complex semisimple algebraic group, we fix a Borel
subgroup $B\subset G$ and a maximal torus $T\subset B$.
Let $\co=\bc[[t]]$ be the ring of complex formal power series and let $\ck=\bc(\!(t)\!)$ 
be the quotient field. Denote by $v:\ck^*\rightarrow\bz$ the standard valuation 
such that $\co=\{f\in\ck\mid v(f)\ge 0\}$. As a set, the {\it affine Grassmannian} $\cg$ is the quotient
$$
\cg=G(\ck)/G(\co).
$$ 
Note that $G(\ck)$ and $\cg$ are {\it ind}--schemes and $G(\co)$ is a group 
scheme (\cite{Ku}). The $G(\co)$-orbits in $\cg$ are finite dimensional quasi-projective
varieties.

\subsection{Schubert varieties in the affine Grassmannian}\label{SchuAndGrassmann}
We recall the classification of $G(\co)$-orbits and the associated $G(\co)$-stable Schubert varieties.
Denote by $\la\cdot,\cdot\ra$ the non--degenerate pairing between 
the character group $X:=\Mor(T,\bc^*)$ of $T$ and its group $\Xvee:=\Mor(\bc^*,T)$ of
cocharacters. Let $\Phi\subset X$ be the root system of the pair 
$(G,T)$, and, corresponding to the choice of $B$, denote $\Phi^+$ the set of positive roots, 
let $\Delta=\{\al_1,\ldots,\al_n\}$ be the set of simple roots, and let $\rho$ be half the 
sum of the positive roots.

Let $\Phivee\subset \Xvee$ be the dual root system, together with a bijection 
$\Delta\rightarrow \Delvee$, $\al\mapsto\alvee$.  We denote by $\Rveep$ the 
submonoide of the coroot lattice $\Rvee$ generated by the positive coroots $\Phiveep$.
We define on $\Xvee$ a partial order by setting $\lam\succ \nu\Leftrightarrow \lam-\nu\in\Rveep$.
Let $\Xveep$ be the cone of 
dominant cocharacters:
$$\Xveep:=\{\lam\in\Xvee \mid \la\lam,\al\ra\ge 0\,\forall\,\al\in\Phi^+ \}.$$
Given $\lam\in \Xvee$, we can view in fact $\lam$ as an element of $G(\ck)$. By abuse of notation
we write also $\lam$ for the corresponding class in $\cg$. 

Let $ev:G(\co)\rightarrow G$ be the evaluation 
maps at $t=0$ and let $\cb=ev^{-1}(B)$ 
be the corresponding Iwahori subgroup. Then
$$
\cg=\bigcup_{\lam\in \Xvee}\cb.\lam=\bigcup_{\lam\in \Xveep}G(\co).\lam
$$
We denote by $X(\lam)=\overline{\cb.\lam}$ the corresponding {\it Schubert variety}.
Let $N=N_G(T)$ be the normalizer in $G$ of the fixed maximal torus $T$, we denote
by $W$ the {\it Weyl group} $N/T$ of $G$. Note that for $\lam\in \Xveep$ we have 
$$
\overline{G(\co).\lam}=X(w_0(\lam))
$$
where $w_0$ is the longest element in the Weyl group $W$. By abuse of notation we just write
$X_\lam$ for the variety $X(w_0(\lam))$ of dimension $\langle2\lam,\rho\rangle$.

\subsection{Reduction to the simply connected case}
Let now $p:G'\rightarrow G$ be an isogeny with $G'$ being simply connected. The
 natural map $p_\co:G'(\co)\rightarrow G(\co)$ is surjective
and has the same kernel as $p$. Let $X'$ and ${X'}^\vee$ be the character group respectively
group of cocharacters of $G'$ for a maximal torus $T'\subset G'$ such that $p(T')=T$,
then $p:T'\rightarrow T$ induces an inclusion ${X'}^\vee\hookrightarrow {X}^\vee$. 

The quotient  $\Xvee/{X'}^\vee$ measures the difference between $\cg$ and the affine
grassmannian $\cg'=G'(\ck)/G'(\co)$. In fact, $\cg'$ is connected, and the connected 
components of $\cg$ are indexed by  $\Xvee/{X'}^\vee$. The natural maps 
$p_\ck:G'(\ck)\rightarrow G(\ck)$ and $p_\co:G'(\co)\rightarrow G(\co)$ induce a 
$G'(\ck)$--equivariant inclusion $\cg'\hookrightarrow \cg$, which is an isomorphism onto the 
component of $\cg$ containing the class of $1$. Now $G'(\ck)$ acts via $p_\ck$ on
all of $\cg$, and each connected component is a homogeneous space for $G'(\ck)$,
isomorphic to $G'(\ck)/\cq$ for some parahoric subgroup $\cq$ of $G'(\ck)$ which is
conjugate to $G(\co)$ by an outer automorphism.

So to study $G(\co)$--orbits on $G(\ck)/G(\co)$ for $G$ semisimple, without loss of generality
we may sometimes for convenience assume that $G$ is simply connected, but we have to 
investigate more generally $G(\co)$--orbits on $G(\ck)/\cq$ for
all parahoric subgroups $\cq\subset G(\ck)$ conjugate to 
$G(\co)$ by an outer automorphism. 

\subsection{Affine Kac-Moody groups}
In the following let $G$ be a simply connected semisimple complex algebraic group.  
The {\it rotation operation} $\gamma:\bc^*\rightarrow \Aut(\ck)$, 
$\gamma(z)\big(f(t)\big)=f(zt)$ gives rise to group automorphisms
$\gamma_G:\bc^*\rightarrow \Aut(G(\ck))$, we denote $\cl(G(\ck))$ 
the semidirect product $\bc^*\semi G(\ck)$.  The rotation operation 
on $\ck$ restricts to an operation on $\co$ and hence we have a natural
subgroup $\cl(G(\co)):=\bc^*\semi G(\co)$ (for this and the following see \cite{Ku},
Chapter 13).

Let $\hat\cl(G)$ be the affine Kac-Moody group associated to
the affine Kac--Moody algebra 
$$
\hat\cl(\Lg)=\Lg\otimes \ck\oplus\bc c\oplus \bc d,
$$
where $0\rightarrow \bc c\rightarrow \Lg\otimes \ck\oplus\bc c\rightarrow
\Lg\otimes \ck\rightarrow 0$ is the universal central extension of the
{\it loop algebra} $ \Lg\otimes \ck$ and $d$ denotes the scaling element.
We have corresponding exact sequences also on the level of groups,
i.e., $\hat\cl(G)$ is a central extension of $\cl(G(\ck))$: 
$$
1\rightarrow\bc^*\rightarrow \hat\cl(G)\mapright{\pi} \cl(G(\ck))\rightarrow 1.
$$
Denote $\cp_\co\subset \hat\cl(G)$ the ``parabolic'' subgroup $\pi^{-1}(\cl(G(\co)))$, then
\begin{equation}\label{fourgrassmann}
\cg=G(\ck)/G(\co)=\cl(G(\ck))/\cl(G(\co))=\hat\cl(G)/\cp_\co .
\end{equation} 
Let $N_\ck$ be the subgroup of $G(\ck)$ 
generated by $N$ and $T(\ck)$, let $\ct\subset\hat\cl(G)$ be the corresponding standard
maximal torus (i.e. $\pi(\ct)\supset \bc^*\semi T$) and let $\cn$ 
be its normalizer in $\hat\cl(G)$. We get two incarnations of the affine Weyl group:
$$
\Waff=N_\ck/T=\cn/\ct.
$$

So to study $G(\co)$--orbits on $G(\ck)/G(\co)$ for $G$ semisimple, without loss of generality
we may assume that $G$ is simply connected and study $\cp_\co$-orbits in 
$\hat\cl(G)/\cq$, where $\cq$ is a parabolic subgroup of the affine Kac-Moody group $\hat\cl(G)$
conjugate to $\cp_\co$ by an outer automorphism.

\subsection{Hall-Littlewood polynomials}
There is a natural action of $W$ on the group algebra
$R[\Xvee]$ with coefficients in some ring $R$. For $\mu\in \Xvee$ we denote the
corresponding basis element by $x^\mu$. The algebra of symmetric polynomials
$R[\Xvee]^W$ is the algebra of invariants under this action. There are several
classical bases known for $R[\Xvee]^W$, all indexed by dominant coweights. Two
important ones are the monomial symmetric polynomials $\{m_\lam\}_{\lam\in\Xveep}$
and the Schur polynomials $\{s_\lam\}_{\lam\in\Xveep}$.
The monomial polynomials are just the orbit sums
$m_\lam=\sum_{\mu\in W\lam} x^\mu$. The Schur polynomial $s_\lam$
is the character of the irreducible representation $V(\lam)$ of the Lie algebra
${\mathfrak g}^\vee$ of the Langlands' dual group $G^\vee$ of $G$.

Specializing the ring of coefficients $R$ to the ring $\cl := \bz[q, q^{-1}]$ of Laurent
polynomials we have another basis for $\cl[\Xvee]^W$, the Hall-Littlewood
polynomials $\{P_\lam\}_{\lam\in\Xveep}$. They are defined by
$$
P_\lam=\frac{1}{W_\lam(q^{-1})}\sum_{w\in W} w\big(x^\lam\prod_{\alpha\in\Phi^+}\frac{1-q^{-1}x^{-\alpha^\vee}}{1-x^{-\alpha^\vee}}\big)
$$
where $W_\lam\subset W$ is the stabilizer of $\lam$ and $W_\lam(t) = \sum_{w\in W_\lam} t^{\ell(w)}$.
The Hall- Littlewood polynomials interpolate between the monomial symmetric polynomials
and the Schur polynomials because $P_\lambda(1) = m_\lam$ and $P_\lam\rightarrow s_\lam$ for
$q\rightarrow\infty$.

We define Laurent polynomials $L_{\lam,\mu}$ for $\lam,\mu\in\Xveep$ by
$$
P_\lam=\sum_{\mu\in\Xveep}q^{-\langle\rho,\lam+\mu\rangle} L_{\lam,\mu} m_\mu.
$$
Since $P_\lam\rightarrow s_\lam$ for $q\rightarrow\infty$, we know that 
$q^{-\langle\rho,\lam+\mu\rangle} L_{\lam,\mu}\in\bz[q^{-1}]$.

The Hall-Littlewood polynomials are connected with the geometry of the affine Grassmannian.
Let $B^-\subset G$ be the opposite Borel subgroup and denote by $U^-$ its unipotent radical.
We are interested in the structure of the irreducible components of the intersection of the following orbits in $\cg$:
\begin{equation}
\label{intersectionMV}
Z_{\lam,\mu}:= G(\co).\lam\cap U^-(\ck).\mu \subset \cg,\quad\lam\in\Xveep,\mu\in\Xvee.
\end{equation}
For a prime power $q$ let $\bF_q$ be the {finite} field with $q$ elements, set $\ck_q:=\bF_q(\!(t)\!)$
and $\co_q:=\bF_q[[t]]$, and let $Z_{\lam,\mu}^q$ be defined as above, only $\ck$ and $\co$ being
replaced by $\ck_q$ and $\co_q$.  The Laurent polynomials $L_{\lam,\mu}$ have the following
geometric interpretation coming from the Satake isomorphism. 
\vskip 5pt\noindent 
{\bf Fact.} $\vert Z_{\lam,\mu}^q\vert = L_{\lam,\mu}(q)$.
\section{Apartments, chambers and buildings}\label{AppartChamberBuilding}
Instead of studying directly the intersection $Z_{\lam,\mu}$ in (\ref{intersectionMV}), we replace the Schubert
variety $X_\lam$ by a desingularization given by an appropriately chosen Bott-Samelson variety or {\it variety of galleries}.
In this context the $U^-(\ck)$-orbits are replaced by Bia\l ynicki-Birula cells associated to a generic anti-dominant coweight.
To describe the choice of the desingularization and get hold of the combinatorial tools to calculate
$\vert Z_{\lam,\mu}^q\vert $, we need to recall some notation from the theory of buildings.
As references we suggest  \cite{B}, \cite{BT}, \cite{R}  and/or \cite{T74}.
\subsection{Apartment}
The apartment associated to the root and coroot datum is the 
real vector space $\A = X^\vee \otimes_{\mathbb Z}\mathbb R$ together
with the hyperplane arrangement defined by the set $\{(\alpha,n)\mid  \alpha\in\Phi,n\in\bz\}$ of affine roots. 
In terms of affine Kac-Moody algebras, a couple $(\alpha,n)$ corresponds to the real affine root $\alpha+n\delta$, 
where $\delta$ denotes the smallest positive imaginary root. For an affine root $(\alpha,n)$ we write $s_{\alpha,n}:x\mapsto x-\bigl(\langle\alpha,
x\rangle+n\bigr)\,\alpha^\vee$ for the affine reflection and $\ch_{\alpha,n}=\{x\in\A
\mid\langle\alpha,x\rangle +n = 0\}$ for the corresponding affine hyperplane of fixed points,  and we write
$$
\ch^+_{\alpha,n} =\{x\in\A \mid\langle\alpha,x\rangle +n\geqslant 0\}\ ;
$$
for the corresponding closed half-space. Similarly we define the negative half space $\ch^-_{\alpha,n}$.

\subsection{Chambers, alcoves, faces and sectors}
\begin{dfn} \rm
The irreducible components of $\A-\bigcup_{\alpha\in \Phi^+}\ch_{\alpha,0}$ are called {\it open
(spherical) chambers}, the closure is called a {\it closed chamber}  or {\it Weyl chamber},
or just {\it chamber}.
The irreducible components of $\A-\bigcup_{(\alpha,n)\in \Phi^+\times\bz}\ch_{\alpha,n}$ are called
{\it open alcoves}, the closure is called a {\it closed alcove} or just an {\it alcove}.
\end{dfn}
The Weyl group $W$ and the affine Weyl group $\Waff$ can be realized in this context as follows: 
$W$ is the finite subgroup of $GL(\A)$ generated by the reflections
$s_{\alpha,0}$, $\alpha\in \Phi$, the affine Weyl group $\Waff$ is the group of affine transformations of $\A$
generated by the affine reflections $s_{\alpha,n}$, $(\alpha,n)\in \Phi\times \bz$. 
The dominant Weyl chamber
$$
C^+:=\{x\in\A\mid \forall\alpha\in\Phi^+: \langle \alpha,x\rangle \geqslant 0\}=\bigcap_{\alpha\in\Phi^+}\ch_{\alpha,0}^+
$$ is a fundamental domain for the action of $W$ on $\A$ and the {\it fundamental alcove} 
$$
\Delta_f=\{x\in \A\mid \forall\alpha\in\Phi^+: 0\le\langle \alpha,x\rangle\le 1\}=
\bigcap_{\alpha\in\Phi^+}\ch_{\alpha,0}^+\cap\bigcap_{\alpha\in\Phi,n>0}\ch_{\alpha,n}^+
$$ 
is a fundamental domain for the action of $\Waff$ on $\A$.
\begin{dfn}\rm
By a {\it face} $F$ we mean a subset of $\A$ obtained as the intersection $\bigcap_{(\beta,m)} \ch_{\beta,n}^\bullet$,
where for each pair $(\beta,n)$, $\beta\in \Phi^+,n\in \bz$, one choses $\ch_{\beta,n}^\bullet$ to be either the hyperplane, the positive or the negative
halfspace. By the corresponding
 {\it open face} $F^o$ we mean the subset of $F$ obtained when replacing
 the closed affine halfspaces in the definition of $F$ by the corresponding
 open affine halfspaces. 
\end{dfn} 
  
We call the affine span $\langle F^o\rangle_{\rm aff}=\langle F\rangle_{\rm aff}$ the 
{\it support} of the (open) face, the {\it dimension} of the face is the dimension of its support. 
A {\it wall of an alcove} is the support of a codimension one face. In general, instead of
the term hyperplane we use often the term {\it wall}, which is more common in the language of buildings.

For any subset $\Omega$ and any face $F$ contained in an apartment
$A$ of a building, we say that a
wall $\ch$ {\it separates $\Omega$ and $F$} if  $\Omega$ is
contained in a closed half space defined by $\ch$ and $F^o$ is a
subset of the opposite open half space.

We call a face of dimension one in $\A$ an {\it edge} and a face of dimension zero a {\it vertex}.
For a vertex $\nu$ let $\Phi_\nu\subset \Phi$ be the subrootsystem consisting of all roots
$\alpha$ such that $\nu\in\ch_{(\alpha,n)}$ for some integer $n$.
A vertex $\nu$ is called a {\it special vertex} if $\Phi_\nu=\Phi$. The special vertices are 
precisely the coweights for $G$ of adjoint type.

By a {\it sector} $\sfrak$ with vertex $\nu\in\A$ we mean a closed chamber translated by $\nu$, i.e.,
there exists a closed chamber ${C}$ such that
\begin{equation}
\label{sfrak1}
\sfrak:=\{\lam\in \A\mid \lam=\nu + z\text{\ for some\ }z\in {C}\}.
\end{equation}
By abuse of notation we write $-\sfrak$ for the sector
\begin{equation}
\label{sfrak2}
-\sfrak=\nu-C=\{\mu\in A\mid \mu=\nu-x\text{\ for some\ }x\in C\}.
\end{equation}
For a sector $\sfrak$ with vertex $\nu$ and an element $\mu\in \A$ let $\sfrak(\mu)$ be the sector obtained
from $\sfrak$ by translating the sector by $\mu-\nu$: If $\sfrak$ is as in (\ref{sfrak1}), then
\begin{equation}
\label{sfrak3}
\begin{array}{rcl}
\sfrak(\mu)&=&\{\lam\in \A\mid \lam=(\mu-\nu)+ z\text{\ for some\ }z\in {\sfrak}\}. \\
&=&\{\lam\in \A\mid \lam=\mu + z\text{\ for some\ }z\in {C}\} \\
\end{array}
\end{equation}
If $\mu\in \sfrak$, then obviously $\sfrak(\mu)\subset \sfrak\subset \sfrak(-\mu)$.
 \subsection{Faces, parahoric and parabolic subgroups}
 The faces in $\A$ are in bijection with parabolic subgroups of the affine
 Kac-Moody group $\hat\cl(G)$
 containing $\ct$ and parahoric subgroups in $G(\ck)$ containing $T$. 
 
 To a root vector $X_\alpha\in\Lie G$, one associates the one-parameter subgroup $U_\alpha  = \{x_\alpha(f) = exp(X_\alpha\otimes f )\mid f\in\ck\}$ of $G(\ck)$ (resp. of $\hat\cl(G)$). If $f = at^n$ for some $a\in\mathbb C$ and $n\in\mathbb Z$, then, for a fixed $n$, the set $U_{\alpha+n\delta} = \{x_\alpha(at^n)\mid a\in\mathbb C\}$ is a one-parameter subgroup associated to the real affine root $\alpha+n\delta$.
 
 \begin{dfn}\rm 
 Given a face $F$, let $\hat P_F$ be the {\it unique parabolic subgroup}
 of $\hat\cl(G)$ containing $\ct$ and all root subgroups $U_{\alpha+n\delta}$ such that $F\subset \ch^+_{\alpha,n}$. 
 
 Given a face $F$, let $U_F$ be the subgroup {of $G(\ck)$} generated by 
 all elements of the form
 $x_\alpha( f )$, where $f\in\ck^*$
 is such that $v(f)\ge n$ and $F\subset \ch^+_{\alpha,n}$.
 Let  $P_F$ be the {\it unique parahoric subgroup}
 of $G(\ck)$ containing $T$ and $U_F$.
  \end{dfn}

For example, if $F$ is a face of the fundamental alcove, then $\cb\subset P_F$. Indeed,
the fundamental alcove itself corresponds to the Iwahori subgroup $\cb\subset G(\ck)$ respectively 
the fixed Borel subgroup $\hat B$ of $\hat\cl(G)$. The origin corresponds to the parahoric subgroup $G(\co)\subset G(\ck)$
respectively the parabolic subgroup $\cp_\co\subset \hat\cl(G)$.

\subsection{The affine building}
Let $\co=\bc[[t]]$ be the ring of complex formal power series and let $\ck=\bc(\!(t)\!)$ 
be the quotient field. Let $N=N_G(T)$ be the normalizer in $G$ of the fixed maximal torus 
$T\subset G$, then  the Weyl group $W$ of $G$ is isomorphic to $N/T$. 
For a real number $r$ let $U_{\beta,r}\subset U_\beta(\ck)$ be the unipotent subgroup
$$
U_{\beta,r}=\{1\}\cup \Big\{x_\beta ( f )\mid f\in\ck^*, v(f)\ge r \Big\}.
$$
For a non-empty subset $\Omega\subset \A$ let 
$\ell_\beta(\Omega)=-\inf_{x\in \Omega} \langle\beta,x\rangle$. We attach to 
$\Omega$ a subgroup of $G(\ck)$ by setting
\begin{equation}\label{uomega}
U_\Omega:=\langle U_{\beta,\ell_\beta(\Omega)}\mid\beta\in\Phi\rangle.
\end{equation}
Let $N(\ck)$ be the subgroup of $G(\ck)$ 
generated by $N$ and $T(\ck)$. To define the affine building $\cj^{\mathfrak a}$,
let $\sim$ be the relation on $G(\ck)\times\A$ defined by:
$$
(g,x)\sim (h,y)\quad\hbox{\rm if }\exists\,n\in {N(\ck)}\ \hbox{\rm such that\ }
nx=y\ \hbox{\rm and\ }g^{-1}hn\in U_x,
$$ where $U_x = U_{\{x\}}$.
\begin{dfn}\rm
The {\it affine building} $\cj^{\mathfrak a}:=G(\ck)\times\A/\sim$ associated to $G$ is the
quotient of $G(\ck)\times\A$ by ``$\sim$''.
The building $\cj^{\mathfrak a}$ comes naturally equipped with a $G(\ck)$--action
$g\cdot(h,y):=(gh,y)$ for $g\in G(\ck)$ and $(h,y)\in \cj^{\mathfrak a}$.
\end{dfn}
The map $\A\rightarrow \cj^{\mathfrak a}$, $x\mapsto(1,x)$
is injective and $N(\ck)-$equivariant, we will identify in the following 
$\A$ with its image in $\cj^{\mathfrak a}$. More generally, a subset $A$
of $\cj^{\mathfrak a}$ is called an {\it apartment} if it is of the form
$g\A$ for some $g\in G(\ck)$. We extend in the same way the notion of a face $F$,
a sector $\sfrak$, a chamber $C$ and the notion of a parahoric subgroup $P_F$ associated to a face. Moreover, the action of $G(\ck)$ is such that the subgroup $U_{\alpha + n\delta}$ fixes the halfspace $\ch^+_{\alpha, n}$; indeed, $x_\alpha(at^n)$ belongs to $U_x$, whence, $(x_\alpha(at^n), x)\sim (1,x)$. 

We denote by $r_{-\infty} : \cj^{\mathfrak a} \to \A$ the retraction centered at $-\infty$. 
It is a chamber complex map and the fibers of $r_{-\infty}$ are the $U^-(\ck)-$orbits in 
$ \cj^{\mathfrak a}$ (see \cite{GL} Definition 8 and Proposition~1, or \cite{BT} Sections 6,7).

\subsection{Residue building}

Let $V$ be a vertex in $\cj^{\mathfrak a}$. Let $\cj^{\mathfrak a}_V$ be the set of all faces $F$ in $\cj^{\mathfrak a}$ such that 
$F\supset V$. Following Bruhat and Tits in Remark 4.6.35 of \cite{BT2}, one endows $\cj^{\mathfrak a}_V$ 
with the complex simplicial structure given by the relation $F\subset F'$, for two faces containing $V$. 
Further, let $\ch_V$ be the connected reductive subgroup of $G$ with root system $\Phi_V$. Then, 
Theorem 4.6.33 of {\it loc. cit.} shows that the structure of a spherical building on the set of all parabolic 
subgroups of $\ch_V$ is isomorphic to the one on $\cj^{\mathfrak a}_V$. 

This isomorphism restricts to any apartment and implies that if $A$ is an apartment in 
$\cj^{\mathfrak a}$, then the set $A_V$
of all faces $F\supset V$ contained in $A$ is an apartment in $\cj^{\mathfrak a}_V$. The simplicial structure 
on $A_V$ is the one associated to the Coxeter complex given by the spherical group $W^v_V$. The latter 
is the subgroup of $W$ generated by the reflections along $\Ker (\alpha)$, for all $\alpha\in\Phi_V$.
\vskip 5pt\noindent
{\bf Notation.}
The set $\cj^{\mathfrak a}_V$, endowed with this structure, is called the {\it residue building of $\cj^{\mathfrak a}$ at $V$}. 
The group $\ch_V$ acts transitively on the set of pairs $(C_V\subset A_V)$ of a chamber in an apartment in $\cj^{\mathfrak a}_V$.
\vskip 5pt
For any face $F$ of $\cj^{\mathfrak a}$ containing $V$, we denote the associated face in $\cj^{\mathfrak a}_V$ by $F_V$. 
Given a sector $\sfrak = V+ C$ in $\A$ with vertex $V$, one associates the chamber $\sfrak_V$ of $\A_V$ 
in the following way: let $\Delta\supset V$ be the unique alcove in $\A$ such that
$\Delta^o\cap\sfrak^o\not=\emptyset$, then $\sfrak_V:=\Delta_V$. By abuse of notation, 
$-\sfrak_V$ will denote the chamber associated to $(V-C)$. Let $C_V^\pm$ denote the 
positive (resp. negative) chamber in $\A_V$ associated to $V + C^\pm$. The stabilizers of 
$C_V^\pm$ in $\ch_V$ are opposite Borel subgroups, denoted by $B_V^\pm$.

Let now $V\subset F$ be a one-dimensional face containing a vertex in $\cj^{\mathfrak a}$. Let
$P_V\supset P_F$ be the parahoric subgroups associated to $V$ and $F$, then
$P_V/P_F$ is isomorphic to a Grassmannian $\ch_V/Q_F$ where $Q_F\supset B_V$ is the 
maximal parabolic subgroup in $\ch_V$ associated to the simple root $\al_F$ defined by the type of $F_V$. 
\section{One-skeleton galleries}\label{oneskeletonsection}
Roughly speaking, a {\it one-skeleton gallery} is a sequence of edges in $\cj^{\mathfrak a}$,
two subsequent ones having a common vertex. A {\it combinatorial one-skeleton gallery} is essentially a gallery that stays in the apartment $\A$. We will see that the set of one-skeleton galleries
of fixed type inherits in a natural way the structure of a Bott-Samelson variety $\Sigma$ and provides
the desired desingularization of the Schubert variety $X_\lam$ (see Proposition~\ref{pr:BS}).
The combinatorial one-skeleton galleries correspond precisely to the centers of 
Bia\l ynicki-Birula cells (Section~\ref{suseCells}) for the smooth variety $\Sigma$.
\subsection{Combinatorial one-skeleton galleries}\label{suseCombiOneSk}
\begin{dfn}\rm 
We call a sequence 
$\gamma = (V_0\subset E_0\supset V_1 \subset E_1 \supset \cdots \supset
V_r\subset E_r\supset V_{r+1})$ of faces in $\A$  
a {\it combinatorial one-skeleton} gallery if
\begin{itemize}
\item the faces $V_i$, $i=0,\ldots,r+1$, are vertices in $\A$; 
\item  the vertex $V_0$ (the {\it source} of the gallery) and the vertex $V_{r+1}$
(the {\it target} of the gallery) are special vertices;
\item the faces $E_i$, $i=0,\ldots,r$, are edges in $\A$.
\end{itemize}
\end{dfn}
If $\gamma' =(V'_0\subset E'_0\supset \cdots\subset E'_t\supset
V'_{t+1})$ is another one-skeleton gallery such that $V'_0 = V_{r+1}$, then one
can concatenate the two galleries to get a new one:
$$
\gamma * \gamma' = (V_0\subset E_0\supset  \cdots \supset V_r
\subset E_r\supset V_{r+1} =
V'_0\subset E'_0\supset \cdots\subset E'_t\supset V'_{t+1}).
$$
By abuse of notation we often write $\gamma * \gamma'$ even if $V'_0 \not= V_{r+1}$. In this case, we mean the concatenation of  $\gamma$ with the displaced gallery $\gamma'+(V_{r+1}-V'_0)$. This construction
makes sense since, by assumption, $V'_0$ and $V_{r+1}$ are special vertices.
\begin{exam}\label{exGalsimpleCoweight}\rm
Suppose $G$ is simple, of adjoint type and $\omega$ is a fundamental coweight. Let $\br_{\ge 0}\omega\subset\A$ 
be the extremal ray of the dominant Weyl chamber $C^+$ spanned by $\omega$. Set $V_0=\Lo$ 
and let $E_0$ be the unique face of dimension one in the intersection of $\br_{\ge 0}\omega$ with the fundamental alcove.
If the second vertex $V_1$ of $E_0$ is different from $\omega$, 
then let subsequently $E_i$ be the unique dimension one face in $\br_{\ge 0}\omega$ (different from $E_{i-1}$) having
$V_i$ as a common vertex with $E_{i-1}$. We obtain a one-skeleton gallery 
$$
\gamma_\omega=(V_0=\Lo\subset E_0\supset V_1 \subset \cdots \supset
V_r\subset E_r\supset \omega=V_{r+1})
$$ {\it joining} $\Lo$ with $\omega$.
We refer to these kind of galleries as {\it fundamental galleries},
the faces $E_j$ of such a gallery are called {\it fundamental faces} (although, they might not be contained in the fundamental alcove).
\end{exam}
\begin{exam} \label{exGalCoweightarbifreesequence}\rm 
Let $\lam$ be an arbitrary dominant coweight. We call a one-skeleton gallery $\gamma=
(V_0\subset E_0\supset  \cdots \subset E_r\supset V_{r+1})$
a {\it dominant combinatorial gallery joining $\Lo$ and $\lam$ along the coweight lattice} if
$\gamma=\gamma_{\om_{i_1}}*\gamma_{\om_{i_2}}*\cdots*\gamma_{\om_{i_r}}$ is a 
concatenation of fundamental galleries such that $\sum_{j=1}^s \om_{i_j}=\lam$.
\end{exam} 
\begin{exam} \label{exGalCoweightarbi}\rm
If we have fixed an enumeration $\omega_1,\ldots,\omega_n$ of the fundamental coweights 
and $\lam=\sum a_i\omega_i$, then we write $\gamma_{\underline \lam}$ for the gallery
$\gamma_{a_1\om_{1}}*\cdots*\gamma_{a_n\om_{n}}$ joining $\Lo$ and $\lam$.
\end{exam} 
\begin{exam} \label{exGalCoweightarbicompletelyfree}\rm
Let $\lam$ be again an arbitrary dominant coweight. We call a one-skeleton gallery $\gamma=
(V_0\subset E_0\supset  \cdots \subset E_r\supset V_{r+1})$
a {\it dominant combinatorial gallery joining $\Lo$ and $\lam$} if the source is $\Lo$,
the target is $\lam$ and all the faces $E_j$ are displaced fundamental faces.
\end{exam}
\begin{dfn}\label{defType}
\rm
Let $\cs^\La$ be the set of affine roots $(\alpha,n)$ such that $\Delta_f\cap\ch_{(\alpha,n)}$
is a face of codimension one. Given a face $F$ of the fundamental alcove $\Delta_f$, we call 
$\cs^\La(F):=\{(\alpha,n)\in \cs^\La\mid F\subset \ch_{(\alpha,n)}\}$ the {\it type} of $F$.
Given an arbitrary face $F\subset \A$, there exists a unique face $F^f$ of the fundamental alcove
which is $\Waff$-conjugate to $F$. We set $\cs^\La(F):=\cs^\La(F^f)$ and call this the type of $F$.\end{dfn}

\begin{dfn}\label{gallerytype}\rm
Given a combinatorial one-skeleton gallery 
$\gamma=(V_0\subset E_0\supset V_1 \subset  \cdots \supset V_r\subset E_r\supset V_{r+1})$,
we call the sequence
$$
t_\gamma:=(\cs^\La(V_0)\supset \cs^\La(E_0)\subset \cs^\La(V_1)\supset\ldots\subset \cs^\La(V_r)
\supset \cs^\La(E_r)\subset \cs^\La(V_{r+1}))
$$
the gallery of types or the type of $\gamma$. We denote by $\Gamma(t_\gamma, V_0)$ the set of {\it all combinatorial galleries starting in $V_0$ and
having $t_\gamma$ as type}.
\end{dfn}
\vskip 5pt\noindent
{\bf Notation.}
Because a face $F$ is always contained in an apartment $A = g\A$, the notion of a {\it one-skeleton gallery}, 
of the {\it type of a face} and the {\it type of a gallery} extends to the whole building $\cj^{\mathfrak a}$. 
\vskip 5pt
Let $W_{V_i}\subset\Waff$ be the Weyl group of $P_{V_i}$, i.e., $W_{V_i}$ is the stabilizer of the vertex $V_i$,
and let $W_{E_i}\subset\Waff$ be the Weyl group of $P_{E_i}$, i.e., $W_{E_i}$ is the stabilizer of the edge $E_i$.
\begin{lemma}\label{weylgroupclasssequence}
Let $\gamma=(V_0\subset E_0\supset V_1 \subset  \cdots \supset V_r\subset E_r\supset V_{r+1})$ be 
a combinatorial one-skeleton gallery.
The set $\Gamma(t_\gamma, V_0)$ can be identified with sequences of Weyl group 
classes in $\prod_{i=0}^r W_{V_i}/W_{E_i}$
via the map
$$
\begin{array}{rcl}
(w_0,\ldots,w_r)\mapsto &(V_0\subset w_0(E_0)\supset w_0(V_1) \subset  w_0w_1(E_1)\supset \cdots\hfill\\
& \hskip 80pt \cdots \subset w_0\cdots w_r(E_r)\supset w_0\cdots w_r(V_{r+1})).
\end{array}
$$ 
In particular, the set $\Gamma(t_\gamma, V_0)$ is finite.
\end{lemma}
\vskip 3pt\noindent
{\it Proof.} Let $\gamma'=(V_0\subset E_0'\supset V_1'\subset E_1'\supset V_2'\supset\ldots)$ be a one-skeleton 
gallery in $\Gamma(t_\gamma, V_0)$.
Since the type of $E_0$ and $E_0'$ are the same, the two have to be conjugate by the finite 
reflection group $W_{V_0}$ generated
by all affine reflections $s_{\alpha,n}$ such that $V_0\subset\ch_{\alpha,n}$. Proceeding by induction, we see that the map
defined above is a bijection. Therefore
\begin{equation}
\label{numberofgalleries}
\vert \Gamma(t_\gamma, V_0)\vert= \sum_{i=1}^r \vert W_{V_i}/W_{E_i}\vert,
\end{equation}in particular, the set $\Gamma(t_\gamma, V_0)$ is finite.
\qed

\subsection{Young tableaux and one-skeleton galleries for classical groups of type ${\tt A}_n,{\tt B}_n,{\tt C}_n$}\label{Youngcombi}


Throughout this section
we use  the Bourbaki enumeration of the weights and coweights. 
Given a partition $p=(p_1,\ldots,p_n)$, the associated {\it Young diagram} of 
shape $p$ consists of left justified rows of boxes with 
$p_1$ boxes in the first row, $p_2$ boxes in the second row, etc.
We enumerate the rows from top to bottom ($R_1,\ldots$) and the columns from the right to the left
($C_1,\ldots$).
\par\noindent
{\bf Young tableaux of type ${\tt A_n}$:}
For a dominant coweight $\lambda = \sum_{i=1}^n a_i\omega_i$ set $p_i=a_i+\ldots+a_n$,
we call $p_{\underline{\lambda}}=(p_1,\ldots,p_n)$ the associated partition.
By a {\it Young tableau ${\mathcal T}$ of shape $p_{\underline{\lambda}}$ and type ${\tt A_n}$}
we mean a filling of the boxes of the Young diagram of shape $p_{\underline{\lambda}}$
with positive integers such that the entries are smaller or equal to $n+1$ and
the entries are strictly increasing in the columns (top to bottom). The tableau is
called {\it semistandard} if in addition the entries are weakly increasing in the rows (left to right).
\par
We use the linearly ordered alphabet ${\frak N}=\{1<2<\ldots<n<\bar n<\ldots<\bar2<\bar 1\}$
with the convention $\bar{\bar i}=i$.
\par\noindent
{\bf Young tableaux of type ${\tt B_n}$:}
For a dominant coweight $\lambda = \sum_{i=1}^n a_i\omega_i$ set $p_i=2a_i+\ldots+2a_{n-1}+a_n$,
we call  $p_{\underline{\lambda}}=(p_1, \ldots,p_n)$ the associated partition.
By a {\it Young tableau of shape $p_{\underline{\lambda}}$ and type ${\tt B_n}$}
we mean a filling of the boxes of the Young diagram of shape $p_{\underline{\lambda}}$ with elements of 
${\frak N}$ such that the entries are strictly increasing in the columns (top to bottom), 
and $i$ and $\bar i$ are never entries in the same column.
Further, for each pair of columns $(C_{2j-1},C_{2j})$, $j=1,\ldots,a_1+\ldots+a_{n-1}$,
either $C_{2j-1}=C_{2j}$ or the column $C_{2j}$ is obtained from $C_{2j-1}$ by exchanging
some of the entries $k$, $1\le k\le \bar1$, in $C_{2j-1}$ by $\bar k$. The tableau is
called {\it semistandard} if in addition the entries are weakly increasing in the rows (left to right).
\par\noindent
{\bf Young tableaux of type ${\tt C_n}$:}
For a dominant coweight $\lambda = \sum_{i=1}^n a_i\omega_i$ set $p_1=a_1+\sum_{j=2}^n 2a_j$
and for $i\ge 2$ set $p_i=2a_i+\ldots+2a_n$,
we call $p_{\underline{\lambda}}=(p_1, \ldots,p_n)$ the associated partition.
By a {\it Young tableau of shape $p_{\underline{\lambda}}$ and type ${\tt C_n}$}
we mean a filling of the boxes of the Young diagram of shape $p_{\underline{\lambda}}$ 
with elements of ${\frak N}$, strictly increasing in the columns (top to bottom), but $i$ and $\bar i$ are never entries in the same column.
Further, for each pair of columns $(C_{a_1+2j-1},C_{a_1+2j})$, $j=1,\ldots,a_2+\ldots+a_n$,
either $C_{a_1+2j-1}=C_{a_1+2j}$ or the column $C_{a_1+2j}$ is obtained from $C_{a_1+2j-1}$ by exchanging for an even number of times
an entry $k$, $1\le k\le \bar1$, in $C_{a_1+2j-1}$ by $\bar k$. 
The tableau is called {\it semistandard} if in addition the entries are weakly increasing in the rows (left to right).
\begin{exam}
\label{exTableaux}
The following tableaux are semistandard Young tableaux of shape $p_{\underline{\lambda}}$ for 
$\lambda=\omega_1+\omega_2+\omega_3$:
$$
\text{type ${\tt A}_3$:}\ \Young(1,1,3|2,3|3),\quad
\text{type ${\tt B}_3$:}\ \Young(1,1,2,2,\bar 2|2,2,\bar 1|\bar 3),\quad
\text{type ${\tt C}_3$:}\ \Young(1,1,2,3,\bar 3|2,\bar 3,\bar 3,\bar 2|3,\bar 2).
$$
\end{exam}
\par\noindent
{\bf The one-skeleton gallery $\gamma_{\omega_i}$:}
\par\noindent
$(1)$ If $\omega_i$ is a {\it minuscule fundamental coweight}
(i.e. $i$ is arbitrary for type ${\tt A}_n$, $i=1$ for type ${\tt C}_n$, $i=n$ for type ${\tt B}_n$),
then $\gamma_{\om_i}=(\Lo\subset E\supset \omega_i)$,
where $E$ is the closed face $\{t\omega_i\mid t\in[0,1]\}$. The galleries of the same type as $\gamma_{\om_i}$ are the galleries
$\gamma_{\sigma(\om_i)}=(\Lo\subset \sigma(E)\supset \sigma(\omega_i))$, where $\sigma(E)= \{t\sigma(\omega_i)\mid t\in[0,1]\}$ and 
$\sigma\in W/W_{\omega_i}$. It follows that the gallery is completely determined by the weight $\sigma(\omega_i)$. 
\par\noindent
$(2)$ If $\omega_i$ is a not a minuscule fundamental weight, then $\langle \omega_i,\beta^\vee\rangle\le 2$
for all positive roots (because we consider only groups of classical type), and there exists at least one root such 
that $\langle \omega_i,\beta^\vee\rangle=2$. Hence 
$\gamma_{\om_i}=(\Lo\subset E_1\supset V\subset E_2\supset \omega_i)$, where $V=\frac{1}{2}\om_i$,
$E_1$ is the closed face $\{t\omega_i\mid t\in[0,\frac{1}{2}]\}$ and $E_2$ is the displaced face 
$E_1+\frac{1}{2}\om_i=\{t\omega_i\mid t\in[\frac{1}{2},1]\}$. The galleries of the same type having $E_1$ as a first
one-dimensional face are of the form $\gamma=(\Lo\subset E_1\supset V\subset \sigma(E_1)+\frac{1}{2}\om_i\supset \frac{\om_i+ \sigma(\om_i)}{2})$,
where $\sigma$ is an element in the subgroup $W^v_V$ of $W$ generated by the simple reflection $s_{\alpha_j}$, $j\not=i$, and the reflection
$s_\beta$, where $\beta$ is the dominant short root. An arbitrary one-skeleton gallery of the same type as $\gamma_{\om_i}$
is of the form $\gamma=(\Lo\subset \tau(E_1)\supset \frac{\tau(\om_i)}{2}\subset \tau\sigma(E_1)+\frac{1}{2}\tau(\om_i)\supset  
\frac{\tau(\om_i)+ \tau\sigma(\om_i)}{2})$,
where $\tau\in W$ and $\sigma$ is a above. So the gallery is completely determined by the pair of weights 
$\tau(\om_i)$ and $\tau\sigma(\om_i)$,
\par\noindent
{\bf Weights, one column tableaux and two column tableaux:}
\par\noindent
We encode the Weyl group conjugates of a fundamental coweight in a tableau consisting of one
column. Then $\omega_i=\epsilon_1+\ldots+\epsilon_i$
for $i=1,\ldots,n$, except for type ${\tt B}_n$, in this case $\omega_n=\frac{1}{2}(\epsilon_1+\ldots+\epsilon_n)$.
To have a uniform notation, for $1\le i\le n$ we write $\epsilon_{\bar i}$ for $-\epsilon_{i}$. In type ${\tt A}_n$
we have $W\cdot \omega_i=\{\epsilon_{j_1}+\ldots+\epsilon_{j_i}\mid 1\le i_1<\ldots<j_i\le n+1\}$.
By writing the indices $j_1<\ldots<j_i$ as entries in a Young tableaux of shape $p_{\underline{\om_i}}$ we get a bijection
between the elements in the orbit and 
the Young tableaux of shape $p_{\underline{\om_i}}$ and type ${\tt A}_n$.

In type ${\tt C}_n$ and ${\tt B}_n$ we have
$
W\cdot \omega_i=\{\epsilon_{j_1}+\ldots+\epsilon_{j_i}\mid 
1\le j_1<\ldots<j_i\le \bar 1,\ \forall\, k,\ell: j_k\not=\bar{j_\ell}\},
$
except for $\omega_n$ in type ${\tt B}_n$, in this case $W\cdot \omega_n
=\{\frac{1}{2}(\epsilon_{j_1}+\ldots+\epsilon_{j_n})\mid 
1\le j_1<\ldots<j_n\le \bar 1, \forall\, k,\ell:j_k\not=\bar{j_\ell}\}$.
So by writing the indices as entries in a Young tableaux consisting of one column with $i$ boxes,
this provides a bijection between the orbit $W\cdot\omega_i$ and 
the one column Young tableaux satisfying the column conditions in the definition of Young tableaux of type ${\tt B}_n$ and ${\tt C}_n$.
In particular:
\begin{lemma}
If $\omega_i$ is minuscule, then this correspondence gives a bijection between the galleries of the same type as 
$\gamma_{\omega_i}$ and the set of Young tableaux of shape $p_{\underline{\omega_i}}$ and type ${\tt A}_n$ respectively 
${\tt B}_n$ or ${\tt C}_n$.
\end{lemma}
Suppose  $\omega_i$ is not minuscule and 
\begin{equation}\label{nonmini}
\gamma=(\Lo\subset \tau(E_1)\supset 
V=\frac{\tau(\om_i)}{2}\subset \tau\sigma(E_1)+\frac{1}{2}\tau(\om_i)\supset  \frac{\tau(\om_i)+ \tau\sigma(\om_i)}{2})
\end{equation}
is a gallery of the same type as $\gamma_{\omega_i}$. If $\tau=id$, then $W^v_V$ is a subgroup of $W$ of type ${\tt B_i}\times {\tt B_{n-i}}$
for type ${\tt B}_n$ and of type ${\tt D_i}\times {\tt C}_{n-i}$ for type ${\tt C}_n$, we write $W^v_V=(W_V^1)^v\times (W_V^2)^v$ for this product decomposition.
We have seen above that the possible choices for $\sigma(E_1)$ is in bijection with the orbit $W^v_V\cdot \omega_i$. 
The second part in the product decomposition of $W^v_V$ is in the stabilizer of $\omega_i$, so the possible choices for $\sigma(E_1)$ are in bijection 
with the orbit $(W_V^1)^v\cdot\omega_i$. The weights occurring in this orbit are twice the weights of the $Spin_{2i+1}$ respectively $Spin_{2i}$ representation,
i.e. the weights are obtained from $\omega_i=\epsilon_1+\ldots+\epsilon_i$ just by a change of the signs in the ${\tt B}_n$-case and by
an even number of sign changes in the ${\tt C}_n$-case. Now if $\gamma$ is an arbitrary gallery  of the same type as 
$\gamma_{\omega_i}$ as in $(\ref{nonmini})$, by linearity the weight $\tau\sigma(\om_i)$ is obtained from the weight 
$\tau(\om_i)$ by a change of some signs
in the ${\tt B}_n$-case respectively by an even number of sign changes in the ${\tt C}_n$-case. So attach to $\gamma$ the (two column) 
tableau of shape $p_{\underline{\omega_i}}$ having as first column the one corresponding to $\tau(\omega_i)$ and as second the one corresponding
to $\tau\sigma(\om_i)$. It follows: 
\begin{lemma} If $\omega_i$ is not minuscule,
then this correspondence describes a bijection between the galleries of the same type as $\gamma_{\omega_i}$ and the set of
Young tableaux of shape $p_{\underline{\omega_i}}$ and type ${\tt B}_n$ respectively ${\tt C}_n$.
\end{lemma}
\par\noindent
{\bf One-skeleton galleries and tableaux:}
\par\noindent
For a dominant coweight $\lambda = \sum a_i\omega_i$ let $\gamma_{\underline{\lambda}}$ 
be the concatenation of the galleries $\gamma_{\omega_j}$ associated to the fundamental 
coweights:
$$
\begin{array}{rcl}
\gamma_{\underline{\lambda}} & = & 
\underbrace{\gamma_{\omega_1} *\cdots *\gamma_{\omega_1}}_{a_1 \hbox{\footnotesize --times }} * 
\cdots *\underbrace{\gamma_{\omega_n} *\cdots *\gamma_{\omega_n}}_{a_n \hbox{\footnotesize --times }}\\
\end{array}
$$ 
and let $\gamma=\gamma_1*\cdots*\gamma_r$, $r=\sum a_i$, be a gallery of the same type as $\gamma_{\underline{\lambda}}$.
One can associate to $\gamma$ in a natural way a tableau of shape $p_{\underline{\lam}}$: fix $j$ minimal such that $a_j\not=0$
and let ${\mathcal T}_1$ be the one- respectively two column tableau of shape $p_{\underline{\omega_j}}$ associated to $\gamma_1$.
Suppose we have already defined ${\mathcal T}_k$, $1\le k\le r$. If $r=k$, then set ${\mathcal T}_\gamma={\mathcal T}_k$. 
If $k<r$, let $\ell$ be such that $\gamma_{k+1}$ is of shape
$\omega_\ell$ and let ${\mathcal T}_{k+1}$ be the tableau obtained from ${\mathcal T}_k$ by adding to the left the one column 
(if $\omega_\ell$ is minuscule) respectively the two column tableau (if $\omega_\ell$ is not minuscule) corresponding to 
$\gamma_{k+1}$. The construction above implies:
\begin{proposition}\label{gallerytableau}
The correspondence $\gamma\leftrightarrow {\mathcal T}_\gamma$ describes a bijection between the set of galleries of the same type as 
$\gamma_{\underline{\lambda}}$ and the set of
Young tableaux of shape $p_{\underline{\lambda}}$ and type ${\tt A}_n$ respectively ${\tt B}_n$ or ${\tt C}_n$.
\end{proposition}

\subsection{Varieties of galleries and Bott-Samelson varieties}

Fix a combinatorial one-skeleton gallery 
$\gamma=(V_0 = \Lo\subset E_0\supset V_1 \subset  \cdots \supset V_r\subset E_r\supset V_{r+1})$, we can associate
to the gallery a sequence of parahoric subgroups:
$$
G(\co)\supset P_{E^f_0}  \subset P_{V^f_1} \supset\ldots \subset P_{V^f_r}\supset P_{E^f_r} \subset P_{V^f_{r+1}}.
$$
We use now this correspondence to identify one-skeleton galleries with points in (generalized) 
Bott-Samelson varieties.
\begin{dfn}\rm
The {\it variety $\Sigma(t_\gamma)$ of galleries of type $t_\gamma$ starting in $V_0 = \Lo$} is the closed subvariety of
$$
G(\ck)/G(\co)\times G(\ck)/P_{E^f_0}\times \ldots \times G(\ck)/P_{E^f_r}\times G(\ck)/P_{V^f_{r+1}}
$$ 
given by all sequences of parahoric subgroups of shape 
$$
G(\co)\supset Q_0\subset R_1\supset Q_1\subset\cdots\supset  Q_r\subset R_{r+1},
$$
where $R_i$ is conjugate to $P_{V^f_i}$ for $i=1,\ldots,r+1$ and $Q_i$ is conjugate to $P_{E^f_i}$ for $i=0,\ldots,r$. 
\end{dfn} 
The action of the group $G(\ck)$ on $\cj^{\mathfrak a}$ naturally extends to an action of $G(\ck)$ on the set of galleries. 
The action of $G(\ck)$ is type preserving, the variety of galleries of fixed type $\Sigma(t_\gamma)$ starting in $V_0$
is stable under the action of $G(\co)$.
Because of the bijection of parahoric subgroups with faces of $\cj^{\mathfrak a}$, the set of all points of the 
variety $\Sigma(t_\gamma)$ is in bijection with the one-skeleton galleries in $\cj^{\mathfrak a}$
$$
g = (V_0 = \Lo\subset E'_0\supset V'_1\subset E'_1\supset \cdots\subset E'_r\supset
V'_{r+1})
$$ 
having type $t_\gamma$. The combinatorial galleries correspond to sequences of subgroups 
conjugated to the $P_{E^f_i}$'s and $P_{V^f_i}$'s by elements in $\Waff$, 
these are precisely the $T$-fixed points in $\Sigma(t_\gamma)$.
Given a sequence of parahoric subgroups
$$
G(\co)\supset P_{E^f_0}\subset P_{V^f_1}\supset P_{E^f_0}\subset\cdots\supset  P_{E^f_r}\subset P_{V^f_{r+1}},
$$ one defines the fibred product 
$$
G(\co)\times_{P_{E^f_0}} P_{V^f_1} \times_{P_{E^f_1}} \ldots  
\times_{P_{E^f_{r-1}}}  P_{V^f_r}/{P_{E^f_r}}
$$ as the quotient of $P_{V^f_0}\times P_{V^f_1}\times\cdots 
\times P_{V^f_r}$ by $P_{E^f_0}\times P_{E^f_1}\times\cdots \times P_{E^f_r}$ given by the action : 
$$
(p_0,p_1,..., p_r)\cdot (q_0,q_1,...,q_r)  = (q_0p_0,p_0^{-1}q_1p_1,..., p_{r-1}^{-1}q_rp_r).
$$ This fibred product is a smooth projective complex variety. Its points are denoted by $[g_0,\ldots,g_r]$. The following proposition is proved in \cite{CC} in the case of varieties of galleries in the spherical building associated to a semi-simple group. The proof extends naturally to our setting.

\begin{proposition}
As a variety, $\Sigma(t_\gamma)$ is isomorphic to the fibred product via the map 										
$$
\begin{array}{rcl}
[g_0,\ldots,g_r]\mapsto & 
(P_{V_0}\supset g_0{P_{E_0}}g_0^{-1}\subset g_0P_{V_1}g_0^{-1}\supset g_0g_1P_{E_1}g_1^{-1}g_0^{-1}
\subset\cdots\hskip 30pt \\
&\hskip 50pt\hfill\cdots \subset g_0\cdots g_rP_{V_{r+1}}g_r^{-1}\cdots g_0^{-1}).
\end{array}
$$
\end{proposition}

Given a dominant coweight $\lam$, let $\gamma_\lam=(\Lo\subset E_0\supset\ldots\subset E_r\supset \lam)$ 
be a corresponding one-skeleton gallery as in Example \ref{exGalCoweightarbicompletelyfree}. In this case, the variety of galleries of type $t_{\gamma_\lam}$ starting in $\Lo$ is called the Bott-Samelson variety associated to the gallery $\gamma_\lam$ and is denoted by:
$$
\Sigma(\gamma_\lam):=G(\co)\times_{P_{E^f_0}} P_{V^f_1} \times_{P_{E^f_1}} \ldots 
P_{V^f_{r-1}}\times_{P_{E^f_{r-1}}} P_{V^f_r}/{P_{E^f_r}}.
$$ The set of all combinatorial galleries in the Bott-Samelson variety is denoted by $\Gamma(\gamma_\lam)$. For instance, the gallery $\gamma_\lam$ corresponds to $[1,w_1,...,w_r]$, where $w_j$ is the minimal representative of the longest element in $W_{V^f_j}/W_{E^f_j}$.

\begin{exam}\rm
If $\omega$ is a minuscule fundamental coweight, then $\gamma_\omega=(\Lo\subset E_0\supset \omega)$ and 
$\Sigma(\gamma_\om)=G(\co)/P_{E_0}$ is just a homogeneous space, isomorphic 
to the orbit $G(\co).\om=\overline{G(\co).\om}=X_\om$.
\end{exam}
In general, the connection between $X_\lam$ and $\Sigma (\gamma_\lambda)$ is given by the following proposition: 

\begin{proposition}
\label{pr:BS}
As in Definition \ref{defType}, denote by $\lam^f$ also the point in $\cg=G(\ck)/G(\co)$ corresponding to the 
vertex of the fundamental alcove of the same type as $\lambda$.
The canonical product map
$$
\begin{array}{rcl}
\pi:\ \Sigma(\gamma_\lam):=G(\co)\times_{P_{E_0}} P_{V_1} \times_{P_{E^f_1}} \ldots P_{V^f_r}/{P_{E^f_r}}&\rightarrow& \cg \\
\left [g_0,g_1,\ldots,g_r\right ] &\mapsto & g_0g_1\cdots g_r {\lambda^f}
\end{array}
$$
has as image the Schubert variety $X_\lam$. The induced map $\pi: \Sigma(\gamma_\lam)\rightarrow X_\lam$
defines a desingularization of the variety $X_\lam$.
\end{proposition}
\noindent 
The proof of this proposition is similar to the proof in the classical case and is based on the fact that the gallery $\gamma_\lam$ is minimal. The notion of minimality in our context is defined and discussed in Section \ref{seMinimal}.
\subsection{Cells}
\label{suseCells}
Let $\eta:\bc^*\rightarrow T$ be a generic anti-dominant coweight. Then the set of $\eta$-fixed-points 
in $\Sigma(\gamma_\lam)$
is finite and is in bijection with the set of all combinatorial galleries of the same type as $\gamma_\lambda$. For such a
fixed $\gamma$ denote by $C_{\gamma}$ the corresponding Bia\l ynicki-Birula cell, i.e. the set of points such 
that $\lim_{t\rightarrow 0}\eta(t).x=\gamma$. 

For a face $F$ in $\cj^{\mathfrak a}$, $\lim_{t\rightarrow 0}\eta(t).F = r_{-\infty}(F)$, and for a face $F$ in $\A$, 
$r_{-\infty}^{-1}(F) = U^-(\ck).F$. Therefore, we want to determine as precisely as possible the group 
$\Stab_-(F) = \Stab_{U^-(\ck)}(F)$ and the set $\Stab_-(V)/\Stab_-(F)$ when $F$ and $V$ are faces of the Coxeter
complex such that $V\subset F$. 

Bruhat and Tits (see (7.1.1) in \cite{BT}) associate to a face $F$ of the Coxeter complex 
the function $f_F:\alpha\mapsto\inf_{k\in\mathbb Z}\{\alpha(F) + k\geq 0\}$. 
If $\alpha\in\Phi$, then $f_F(\alpha)$ is the smallest integer $n$ such that $F$ 
lies in the closed half-space $H^+_{\alpha,n}$. The function $f_F$ is convex
and positively homogeneous of degree $1$; in particular,
$f_F(i\alpha+j\beta)\leqslant if_F(\alpha)+jf_F(\beta)$ for all roots
$\alpha,\beta\in\Phi$ and all positive integers $i,j$. 

When $F$ and
$V$ are two faces of $\A$ such that
$V\subset F$, then we denote by $\Phi_-^{\mathfrak a}(V,F)$ the set of
all affine roots $\beta\in\Phi_-\times\mathbb Z$ such that
$V\subset \ch_\beta$ and $F\not\subset H^+_\beta$; in other words,
$(\alpha,n)\in\Phi_-^{\mathfrak a}(V,F)$ if and only if $\alpha\in\Phi_-$,
$n=f_{V}(\alpha)$ and $n+1=f_F(\alpha)$. We denote by
$\Stab_-(V,F)$ the subgroup of $U^-(\ck)$ generated by the
elements of the form $x_\beta(a)$ with $\beta\in\Phi_-^{\mathfrak a}(V,F)$
and $a\in\mathbb C$. We plot an example to help the understanding 
of all these definitions. In the following picture, $\alpha$ is a positive root.

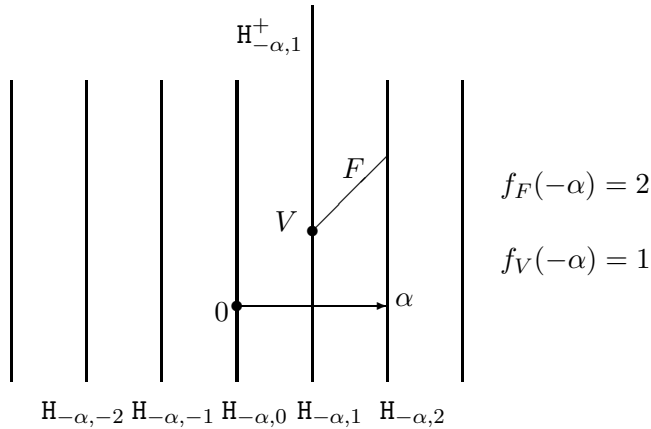
\begin{center}
 \setlength{\unitlength}{1cm}
\vskip-60pt
\begin{picture}(6,8)
\put(0.4,0.5){$\ch_{-\alpha,-2}$}
\put(1.6,0.5){$\ch_{-\alpha,-1}$}
\put(2.8,0.5){$\ch_{-\alpha,0}$}
\put(3.8,0.5){$\ch_{-\alpha,1}$}
\put(4.9,0.5){$\ch_{-\alpha,2}$}
\put(3,5.5){$\ch^+_{-\alpha,1}$}

\put(6.5,2.5){$f_V(-\alpha) = 1$}
\put(6.5,3.5){$f_F(-\alpha) = 2$}

\put(3,2){\circle*{0.15}}
\put(2.7,1.8){$0$}
\put(3,2){\vector(1,0){2}}
\put(5.1,2){$\alpha$}

\put(4,3){\circle*{0.15}}
\put(3.5,3){$V$}
\put(4,3){\line(1,1){1}}
\put(4.4,3.7){$F$}

\thicklines
\put(0,1){\line(0,1){4}}
\put(1,1){\line(0,1){4}}
\put(2,1){\line(0,1){4}}
\put(3,1){\line(0,1){4}}
\put(4,1){\line(0,1){5}}
\put(5,1){\line(0,1){4}}
\put(6,1){\line(0,1){4}}
\end{picture}
\end{center}
\vskip-10pt
The following proposition is proved in \cite{BaGa}, Proposition 19.
\begin{proposition}
\label{pr:DescStab-}
\begin{enumerate}
\item\label{it:PrDS-a}
The stabilizer $\Stab_-(F)$ of a face $F$ of the Coxeter
complex is generated by the elements $x_\alpha(p)$, where
$\alpha\in\Phi_-$ and $p\in\ck$ satisfy $\val(p)\geqslant
f_F(\alpha)$.
\item\label{it:PrDS-b}
Let $F$ and $V$ be two faces of the Coxeter complex such that
$V\subset F$. Then $\Stab_-(V,F)$ is a set of
representatives for the right cosets of $\Stab_-(F)$ in $\Stab_-(V)$.
For any total order on the set $\Phi_-^{\mathfrak a}(V,F)$, the map
$$
(a_\beta)_{\beta\in\Phi_-^{\mathfrak a}(V,F)}\mapsto\prod_{\beta\in\Phi_-^{\mathfrak a}(V,F)}x_\beta(a_\beta)
$$ is a bijection from $\mathbb C^{\Phi_-^{\mathfrak a}(V,F)}$ onto $\Stab_-(V,F)$, where $\mathbb C^{\Phi_-^{\mathfrak a}(V,F)}$ is the set of all mappings from $\Phi_-^{\mathfrak a}(V,F)$ to $\mathbb C$. 
\end{enumerate}
\end{proposition}

\medskip
In $\cj^{\mathfrak a}_V$, $F_V$ corresponds to a spherical face of dimension one
given by an element $\overline {w_F}\in W_V/W_F$ such that $F =
\overline {w_F} \phi^-_F$, where $\phi^-_F$ is the face having the same type as $F_V$
contained in $C_V^-$. Let $D = proj_F(C_V^-)$ be the closest chamber to $C_V^-$
containing $F_V$, then $\overline{w_F} = \overline{w(C_V^-, D)}$.

\begin{proposition}
The walls $\ch_\beta$, $\beta\in\Phi^{\mathfrak a}_-(V,F)$, viewed as walls in
$\A_V$, are the walls crossed by any minimal gallery of chambers between
$C_V^-$ and $D$.
\end{proposition}

\noindent{\it Proof. } By definition, $\Phi^{\mathfrak a}_-(V,F) = \{\beta\in\Phi_-\times\mathbb Z \mid V\subset \ch_\beta,\ F\not\subset H^+_\beta\}$. So, for any $\beta$ in this set, the wall $\ch_\beta$ separates $C_V^-$ from $F_V$. Moreover, $F\not\subset H^+_\beta$ implies that it separates also $C_V^-$ from $D$. Hence, $\ch_\beta$ is crossed by any minimal gallery of chambers between $C^-_V$ and $D$.
\qed
\bigskip

Therefore, $\Stab_-(V,F)$ can be identified with $U_V^-(\overline {w_F})$, where
the latter is defined as $B_V^-\overline {w_F} P^-_F/P^-_F = U_V^-(\overline {w_F})\overline {w_F} P^-_F/P^-_F.$
Let 
$$
\delta = [\de_0,\de_1,...,\de_r] = (0 = V_0\subset E_0\supset V_1 \subset 
\cdots \supset V_r\subset E_r\supset V_{r+1}) \in\Gamma(\gamma_\lambda) 
$$ 
and set
$$
\Stab_-(\delta)=\Stab_-(V_0,E_0)\times
\Stab_-(V_1,E_1)\times\cdots\times
\Stab_-(V_r,E_r).
$$
\begin{proposition}\label{pr:Cell}
The map
$$
\begin{array}{l}
f:(v_0,v_1,\ldots,v_r)\quad\mapsto \\
\quad\quad\quad\bigl[v_0\;\overline{\delta_0}\,,\;
\overline{\delta_0}^{-1}\,v_1\;\overline{\delta_0\delta_1}\,,\;
\overline{\delta_0\delta_1}^{-1}\,v_2\;\overline{\delta_0\delta_1
\delta_2}\,,\ldots,\;\overline{\delta_0\cdots\delta_{r-1}}^{-1}\,
v_r\;\overline{\delta_0\cdots\delta_r}\,\bigr]
\end{array}$$
from $\Stab_-(\delta)$ to $\Sigma(\gamma_\lam)$ is injective
and its image is $C_\delta$ (here $\overline{x}$ means that we take a coset representative
of $x$ in $G(\ck)$). Therefore, $C_\delta$ is isomorphic to
$\mathbb C^{\Phi_-^a(V_0,E_0)}\times\cdots\times \mathbb C^{\Phi_-^a(V_r,E_r)}$.
\end{proposition}

\noindent{\it Proof. } The proof is similar to the one of Proposition 22 in \cite{BaGa}, 
we give it for the comfort of the reader. Set $\widetilde{\Stab_-(\delta)} = $ 
$$ \Stab_-(V_0)\underset{\Stab_-
(E_0)}\times\Stab_-(V_1)\underset{\Stab_-(E_1)}
\times\cdots\underset{\Stab_-(E_{r-1})}\times\Stab_-
(V_r)/\Stab_-(E_r).$$
Using the inclusions
\begin{alignat*}2
\Stab_-(E_j)&\subseteq\overline{\delta_0\cdots\delta_j}\;P_{E_j^f}\;
\overline{\delta_0\cdots\delta_j}^{-1}&&\text{(for
$0\leqslant j\leqslant r$),}\\[4pt]
\Stab_-(V_0)&\subseteq G(\co)\overline{\delta_0}^{-1},&&\\[4pt]
\Stab_-(V_j)&\subseteq\overline{\delta_0\cdots\delta_{j-1}}\;P_{V_j^f}\;\overline{\delta_0\cdots\delta_j}^{-1}&\qquad&\text{(for
$1\leqslant j\leqslant r$),}
\end{alignat*}
standard arguments imply that the map
$$
\begin{array}{l}
f:[v_0,v_1,\ldots,v_r]\quad\mapsto \\
\quad\quad\quad\bigl[v_0\;\overline{\delta_0}\,,\;
\overline{\delta_0}^{-1}\,v_1\;\overline{\delta_0\delta_1}\,,\;
\overline{\delta_0\delta_1}^{-1}\,v_2\;\overline{\delta_0\delta_1
\delta_2}\,,\ldots,\;\overline{\delta_0\cdots\delta_{r-1}}^{-1}\,
v_r\;\overline{\delta_0\cdots\delta_r}\,\bigr]
\end{array}
$$ from $\widetilde{\Stab_-(\delta)}$ to $\hat\Sigma(\gamma_\lambda)$ is
well-defined.

The proof of Proposition~6 in~\cite{GL} says that
an element $d=[g_0,g_1,\ldots,g_r]$ in the Bott-Samelson variety
belongs to the cell $C_\delta$ if and only if there exists
$u_0,u_1,\ldots,u_r\in U^-(\ck)$ such that
$$
g_0g_1\cdots g_j E_j^f\,=u_j E_j\quad\text{and}
\quad u_{j-1} V_j=u_j V_j
$$
for each $j$. Setting $v_0=u_0$ and $v_j=u_{j-1}^{-1}u_j$ for
$1\leqslant j\leqslant r$, the conditions above can be rewritten
$$g_0g_1\cdots g_j P_{E_j^f} = v_0v_1\cdots v_j\;\overline{\delta_0\delta_1
\cdots\delta_j}\;P_{E_j^f} \quad\text{and}\quad v_j\in\Stab_-(V_j),$$
which shows that $f([v_0,v_1,\ldots,v_r])=d$. Therefore the image
of $f$ contains the cell $C_\delta$. The reverse inclusion can be
established similarly.

The map $f$ is injective. Indeed suppose that two elements
$v=[v_0,v_1,\ldots,v_r]$ and $v'=[v'_0,v'_1,\ldots,v'_r]$ in
$\widetilde{\Stab_-(\delta)}$ have the same image. Then
$$v_0v_1\cdots v_j\;\overline{\delta_0\delta_1\cdots\delta_j}\; P_{E_j^f} =
v'_0v'_1\cdots v'_j\;\overline{\delta_0\delta_1\cdots\delta_j}\; P_{E_j^f}$$
for each $j\in\{0,\ldots,r\}$. This means geometrically that
$$v_0v_1\cdots v_j\,\overline{\delta_0\delta_1\cdots\delta_j}\,
E_j^f\,=v'_0v'_1\cdots v'_j\;\overline{\delta_0\delta_1
\cdots\delta_j}\,E_j^f;$$
in other words, $v_0v_1\cdots v_j$ and $v'_0v'_1\cdots v'_j$ are equal
in $U^-(\ck)/\Stab_-(E_j)$. Since this holds for each $j$,
the two elements $v$ and $v'$ are equal in $\widetilde{\Stab_-(\delta)}$.

We conclude that $f$ induces a bijection from $\widetilde{\Stab_-(\delta)}$
onto $C_\delta$. It then remains to observe that the map
$(v_0,v_1,\ldots,v_r)\mapsto[v_0,v_1,\ldots,v_r]$ from
$\Stab_-(\delta)$ to $\widetilde{\Stab_-(\delta)}$ is bijective. This
follows from Proposition~\ref{pr:DescStab-} part \ref{it:PrDS-b}: indeed
for each $[a_0,a_1,\ldots,a_r]\in\widetilde{\Stab_-(\delta)}$, the
element $(v_0,v_1,\ldots,v_r)\in\Stab_-(\delta)$ such that
$[v_0,v_1,\ldots,v_r]=[a_0,a_1,\ldots,a_r]$ is uniquely determined by
the condition that for all $j\in\{0,1,\ldots,r\}$,
$$v_j\in\bigl((v_0\cdots v_{j-1})^{-1}(a_0\cdots a_j)\Stab_-(E_j)
\bigr)\cap\Stab_-(V_j,E_j).$$
\qed
\section{Minimal one-skeleton galleries}
\label{seMinimal}
To study the intersection $Z_{\lam,\mu}:= G(\co).\lam\cap U^-(\ck).\mu$ (see (\ref{intersectionMV}))
using the language of galleries, we need to characterize which galleries in $\Sigma(\gamma_\lam)$
map onto the dense orbit $G(\co).\lam$ in $X_{\lam}$ (see Corollary~\ref{cor:summary}). This will be done by introducing
the notion of minimal galleries. These galleries replace the 
minimal galleries of alcoves used in \cite{GL}. 

\subsection{Minimality relative to an equivalence class of sectors}

A sector $\sfrak$ in the affine building is a sector in some apartment.
Two sectors are called equivalent if the intersection of the two is
again a sector. 
Recall that for two given sectors $\sfrak_1,\sfrak_2$, there exists an apartment $A$ and 
subsectors $\sfrak'_1\subset \sfrak_1$, $\sfrak'_2\subset \sfrak_2$
such that $\sfrak'_1, \sfrak'_2\subset A$. 
The set of equivalence classes of sectors is in bijection with the 
set of Weyl chambers in $\A$. Given a sector $\sfrak$, we denote such an equivalence class by
$\underline\sfrak$.
\begin{dfn}
\label{dfnMini}
A one-skeleton gallery 
$$
\gamma = (V_0\subset E_0\supset V_1 \subset E_1 \supset \cdots \supset
V_r\subset E_r\supset V_{r+1})
$$ 
is called minimal if there exists an equivalence class of sectors $\underline\sfrak_\gamma$
and representatives $\sfrak_0,\ldots,\sfrak_r\in\underline\sfrak_\gamma$ 
such that for all  $i=0,\ldots,r$: $V_i$ is the vertex for the sector $\sfrak_i$ and $V_{i}\subset E_{i}\subset \sfrak_i$. The class $\underline\sfrak_\gamma$ is not necessarily uniquely determined by $\gamma$.
\end{dfn}
The sequence $\underline\sfrak(\gamma)=(\sfrak_0,\ldots,\sfrak_r)$ is called a {\it chain of sectors associated
to $\gamma$.}
\begin{exam}\rm
The galleries described in Examples~\ref{exGalsimpleCoweight},\ref{exGalCoweightarbifreesequence},
\ref{exGalCoweightarbi} and \ref{exGalCoweightarbicompletelyfree} are minimal galleries such that
$\underline\sfrak_\gamma= \underline C^+$.
\end{exam}
\begin{rem}\label{locminbutnotglobmin}\rm

1) With a little extra effort, one can see that this definition is an ``instance'' of Definition 5.24 of \cite{CC}, where Contou-Carr\`ere defines generalized minimal galleries in a Coxeter complex.

2) Thinking in geometric terms one might be inclined to demand that ``minimality"
should be a local property, i.e. to be verified at each vertex of the gallery.
This is not sufficient, see below. Propositions~\ref{minsklel1}
and \ref{orbit} show that the more rigid definition above is the right definition
for our purpose. 
\end{rem}
\begin{exam}\label{locminbutnotglobminexam}\rm
Consider the apartment of type ${\tt A}_2$, we use the notation $\gamma_{\om_1},\gamma_{\om_2}$
as in Example~\ref{exGalCoweightarbifreesequence} for the fundamental weights. For an element $w$
of the Weyl group set $\gamma_{w(\om_i)}:=w(\gamma_{\om_i})$, $i=1,2$.

The galleries $\gamma_1=\gamma_{s_1(\omega_1)}*\gamma_{s_1s_2(\omega_2)}$
and $\gamma_2=\gamma_{s_1s_2(\omega_2)}*\gamma_{s_2s_1(\omega_1)}$
are minimal with $\underline\sfrak_{\gamma_1}= \underline{s_1s_2(C^+)}$
and $\underline\sfrak_{\gamma_2}=\underline{s_1s_2s_1(C^+)}$. 
But the gallery $\gamma:=\gamma_{s_1(\omega_1)}*\gamma_{s_1s_2(\omega_2)}*
\gamma_{s_2s_1(\omega_1)}$ is not minimal in the sense above.
\end{exam}
The natural action of $G(\ck)$ on $\cj^{\mathfrak a}$ induces a natural
action on one-skeleton galleries: Let $\gamma$ be a one-skeleton gallery
and ${\mathbf g}\in G(\ck)$, then we set
$$
{\mathbf g}. \gamma =({\mathbf g}.V_0\subset {\mathbf g}.E_0\supset {\mathbf g}.V_1 \subset 
{\mathbf g}.E_1 \supset \cdots \supset {\mathbf g}.V_r\subset {\mathbf g}.E_r
\supset {\mathbf g}.V_{r+1})
$$
It follows immediately that the property of being minimal is preserved by the action.
Let $\Lo$ be the origin in $\A$.
\begin{proposition}\label{minsklel1}
Let $\gamma$ be a minimal one-skeleton gallery in the building $\cj^{\mathfrak a}$ 
starting in $V_0 = \Lo$ and let $\underline\sfrak(\gamma)=(\sfrak_0,\ldots, \sfrak_r)$ be an associated chain of sectors.
\begin{itemize}
\item[a)] $\gamma$ is contained in $\sfrak_0 $. 
\item[b)] For all $i=0,\ldots,r+1$: $(V_i\subset E_i\supset V_{i+1}\subset\cdots \supset V_{r+1})\subset \sfrak_0(V_i)$. 
In particular, one may choose as associated chain of sectors $\sfrak(\gamma)=(\sfrak_0,\sfrak_0(V_1),
\ldots, \sfrak_0(V_r))$.
\item[c)] There exists a unique gallery $\gamma'$ in the orbit $G(\co). \gamma $ such that $\gamma'$ is contained in  
the dominant Weyl chamber $C^+$ in $\A$ and the chain of sectors associated
to $\gamma$ can be chosen to be all in the class of $C^+$.
\end{itemize}
\end{proposition}
\noindent 
{\it Proof.} 
The sectors $\sfrak_{i},\sfrak_{i+1}$, $0\le i\le r-1$, are in the same equivalence class, so there exists a
subsector $\sfrak_{i}'$ contained in both sectors. A sector is the closure of the convex hull
of its vertex and any subsector, and hence $\sfrak_i$ is the closure of the convex hull of
$V_i$ and $\sfrak_{i}'$, and $\sfrak_{i+1}$ is the closure of the convex hull of $V_{i+1}$ and $\sfrak_{i}'$.
Since $\sfrak_i\supset E_i\supset V_{i+1}$ it follows that $\sfrak_{i+1}$ is a subsector of
$\sfrak_{i}$, in fact, $\sfrak_{i+1}=\sfrak_{i}(V_{i+1})$. By induction we conclude:
\begin{equation}
\label{sectorinclusion}
\sfrak_0\subset \sfrak_0(V_1)=\sfrak_1
\subset \sfrak_0(V_2)=\sfrak_1(V_2)=\sfrak_2\subset \ldots
\subset \sfrak_0(V_r)= \ldots= \sfrak_r.
\end{equation}
Now $E_j\subset \sfrak_j$ for all $j=0,\ldots,r$, so
$(V_i\subset E_i\supset V_{i+1}\subset\cdots \supset V_{r+1})\subset \sfrak_0(V_i)$,
which finishes the proof of $a)$ and $b)$. 

Since $G(\co)$ acts transitively on the set of sectors having $\Lo$ as vertex, there exists 
${\mathbf g}\in G(\co)$ such that ${\mathbf g}.\sfrak_0=C^+$. It follows: $\gamma'={\mathbf g}.\gamma$
is completely contained in $C^+$. It remains to prove the uniqueness.

Suppose now ${\mathbf g'}.\gamma =(V'_0\subset E'_0\ldots)$ and 
${\mathbf g''}. \gamma =(V''_0\subset E''_0\ldots)$ are contained in the dominant Weyl chamber
and hence in $\A$. The action of $G(\co)$ preserves types, 
so both galleries have the same gallery of types. Obviously we have $V'_0=V''_0=\Lo$ and 
$E'_0=E''_0$ since both are faces of the same type of the fundamental alcove. It follows: 
$V'_1=V''_1$. Since ${\mathbf g'}.\sfrak_1={\mathbf g'}.(\sfrak_0(V_1))=C^+(V_1')=C^+(V_1'')
={\mathbf g''}.(\sfrak_0(V_1))={\mathbf g''}.\sfrak_1$, $E_1'\supset V_1'$ and $E_1''\supset V_1'$ are faces
of the same type of the same sector $C^+(V_1')$, so necessarily $E_1'=E_1''$. 
Repeating the argument shows $\gamma=\gamma'$.
\qed
\begin{rem}\label{headsector}\rm
In part $b)$ above one can replace $\sfrak_0$ by $-\sfrak_0$ (see \eqref{sfrak2}
for the notation), but one has to replace the ``tail" of the gallery by the ``head":
For all $i=0,\ldots,r+1$: 
$(V_0\subset E_0\supset V_{1}\subset\cdots \supset V_{i})\subset -\sfrak_0(V_i)$.  
\end{rem}
\subsection{Orbits}
The following proposition gives us a precise dictionary between the language
of minimal one-skeleton galleries and orbits of $G(\co)$ in the affine Grassmannian 
$G(\ck)/G(\co)$. 
\begin{proposition}\label{orbit}
Let $\gamma $ be a minimal one-skeleton gallery in $\cj^{\mathfrak a}$ starting in $\Lo$ and ending
in $\lam=V_{r+1}$ in the dominant Weyl chamber $C^+$ in $\A$.
The target $\lam=V_{r+1}$ is a special point and hence is a coweight, by abuse of notation
we also write $\lam$ for the corresponding point in $G(\ck)/G(\co)$. The following natural map
between the $G(\co)$-orbit of the gallery $\gamma$ and the $G(\co)$-orbit of $\lam$ in $\cg$
is bijective:
$$
G(\co). \gamma\longrightarrow G(\co).\lambda\subset G(\ck)/G(\co),\quad
{\mathbf g}. \gamma\mapsto {\mathbf g}.\lam.
$$
\end{proposition}
\noindent
{\it Proof.} 
The map $\pi$ defined in Proposition \ref{pr:BS} is $G(\co)-$equivariant and, as a desingularization of $X(\lam)$, it must be an isomorphism over an open subset of $X(\lam) = \overline{G(\co).\lam}$. So it restricts to a bijection $G(\co). \gamma \simeq G(\co). \lam$.
\qed
\vskip5pt\noindent
Summarizing we have proved:
\begin{corollary}\label{cor:summary}
\begin{enumerate}
  \item Let $\gamma$ be a minimal one-skeleton gallery starting in $\Lo$, then the $G(\co)$-orbit
  of $\gamma$ contains a unique element completely contained in the dominant Weyl chamber.
  \item Let $\gamma,\gamma'$ be two minimal one-skeleton galleries starting in $\Lo$.The two
  galleries are conjugate under the action of $G(\co)$ if and only if they have the same galleries of types.
  \item Let $\gamma=(V_0=\Lo\subset E_0\supset V_1 \subset E_1 \supset \cdots \supset
V_r\subset E_r\supset V_{r+1})$ be a minimal one-skeleton gallery contained in the dominant Weyl chamber
and let $\lam=V_{r+1}\in X^+$ be the target. The projection $G(\co).\gamma\mapsto G(\co).\lam\subset G(\ck)/G(\co)$
is a bijection.
\end{enumerate}
\end{corollary}

\subsection{Positively folded one-skeleton galleries}
Consider a vertex $V$ of a gallery together with the two edges $E$ and $F$. To simplify the notation, 
we call such a sequence $(V_0\subset E\supset V\subset F\supset V_1)$ of vertices 
and edges a {\it two steps gallery}. Note that none of the vertices needs to be a special vertex,
and we often omit $V_0$ and $V_1$. A {\it two steps gallery} is called {\it minimal}
if there exists a sector $\mathfrak s$ with vertex $V_0$ such that 
$E\subset  \mathfrak s$ and $F\subset {\mathfrak s}(V)$. An equivalent condition
is the following: there exists a sector $\mathfrak s'$ with vertex $V$ such that $E\subset \mathfrak s'$ and
$F\subset -\mathfrak s'$.
\vskip 5pt\noindent
\begin{dfn}\label{posfoldefn} \rm
We say that a two steps gallery $(E\supset V\subset F')\subset\A$ is obtained from $(E\supset V\subset F)\subset\A$ 
by a {\it positive folding} if there exists an affine root $(\beta,n)$ such that 
$$
V\in \ch_{\beta,n}, \quad F'=s_{\beta,n}(F)\quad\text{and}\quad \ch_{\beta,n}\text{\ separates\ } F \text{\ and\ }C^-(V) \text{\ from\ } F'.
$$
A two steps gallery $(E\supset V\subset F)$ in $\A$ is called {\it positively folded} if either the gallery is a minimal,
or if there exist faces $F_0,\ldots, F_s$ containing $V$ such that:
\begin{itemize}
\item $(E\supset V\subset F_0)$ is minimal and $ F_s=  F$,
\item $\forall j=1,\ldots,s$: $(E\supset V\subset F_j)$ is obtained from $(E\supset V\subset F_{j-1})$
by a positive folding.
\end{itemize} 
\end{dfn}
In the residue building at a vertex $V$ we say that $(E_V,F_V)$ is a {\it minimal pair} if there exists two 
opposite sectors $\sfrak$ and $-\sfrak$ with vertex $V$ such that $E \subset \sfrak$ and 
$F\subset -\sfrak$. We use this notion to get the following equivalent definition for a positively folded 
two-step gallery, which uses more the language of the residue building:
\begin{dfn}\label{posfoldefn2}\rm
The two-step gallery $(E\supset V\subset F)$ in $\A$ is called {\it positively folded} if there exist 
\begin{itemize}
\item faces $F_{0,V},\ldots,F_{s,V}$ such that $(E_V, F_{0,V})$ is a minimal pair, and $F_{s,V}=F_{V}$,
\item for all $j=1,\ldots,s$ there exists an affine root $(\beta_j,n_j)$ such that $\beta_j\in \Phi_V$,
$V\in \ch_{\beta_j,n_j}$, $s_{\beta_j,n_j}(F_{j-1,V})=F_{j,V}$ and $\ch_{\beta_j,n_j}$ separates $C^-_V$ and $F_{j-1,V}$ from $F_{j,V}$. 
\end{itemize}
\end{dfn}
\begin{rems}\rm
1) Note that two faces $E_V$ and $F'_V$ could be opposite in $\cj^{\mathfrak a}_V$ (i.e. there exists two 
opposite chambers $D$ and $-D$ such that $E_V\subset D$ and $F'_V\subset -D$) without being a minimal pair. 
This can be seen in a root system of type ${\tt B_2}$.

2) Note that neither the face $F_0$ nor the sequence of reflections are unique in Definition \ref{posfoldefn}
and Definition \ref{posfoldefn2}. 
Below is an example for a root system of type ${\tt B_2}$. The dot in the middle is the vertex $V$. 
The fact that $(E\supset V\subset F)$ is positively folded can be seen using one of the two faces $F_0$ 
and some reflections with respect to the drawn walls.

\bigskip
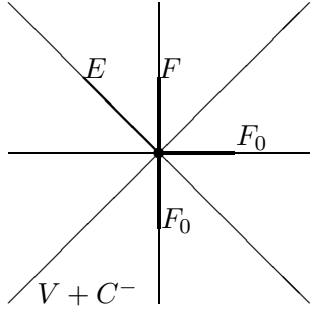
\begin{center}
 \setlength{\unitlength}{1cm}
\begin{picture}(4,4)
\thinlines
\put(0,0){\line(1,1){4}}
\put(0,2){\line(1,0){4}}
\put(0,4){\line(1,-1){4}}
\put(2,0){\line(0,1){4}}

\put(2,2){\circle*{0.15}}
\put(1,3){$E$}
\put(2,3){$F$}
\put(1.9,1){ $F_0$}
\put(2.9,2.1){ $F_0$}
\put(0,0){$\quad V+C^-$}

\thicklines
\put(2,2){\line(-1,1){1}}
\put(2,2){\line(0,1){1}}
\put(2,2){\line(1,0){1}}
\put(2,2){\line(0,-1){1}}
\end{picture}
\end{center}
\end{rems}

Since the equivalence classes of sectors are in 
bijection with the Weyl chambers, we can endow the set of equivalence classes with the Bruhat order:
$\underline\sfrak\ge \underline\sfrak'$ iff $\underline\sfrak=\underline{\tau(C^+)}$, $\underline\sfrak=\underline{\kappa(C^+)}$
and $\tau\ge\kappa$.
Minimal galleries $\gamma$ are characterized by the property that one can find an associated chain of sectors
$\underline\sfrak(\gamma)=(\sfrak_0,\ldots\sfrak_r)$ such that we have for the classes:
$\underline\sfrak_0= \ldots=\underline\sfrak_r$.

We are going to weaken this condition for
the {\it combinatorial positively folded one-skeleton galleries}:
\begin{dfn}\rm
\label{dfnPosFold1}
For a dominant coweight $\lam$, let $\gamma_\lam$ be a minimal one-skeleton gallery contained in 
$C^+$, starting in $\Lo$ and ending in $\lam$. A combinatorial one-skeleton gallery of type $t_{\gamma_\lam}$:
$$
\gamma=(V_0=\Lo\subset E_0\supset\ldots\subset E_r\supset V_{r+1})\subset\A
$$ 
is called {\it globally positively folded} or just {\it positively folded} if 
\begin{itemize}
\item[{\it i)}] the gallery is locally positively folded, i.e. the two-step galleries $(E_{i-1}\supset V_i\subset E_i)$ are
positively folded for all $i=1,\ldots,r$;
\item[{\it ii)}] there exists a chain of sectors $\underline\sfrak(\gamma)=(\sfrak_0,\ldots\sfrak_r)$
such that for all $i=0,\ldots,r$: $V_i$ is the vertex of $\sfrak_i$ and $E_i\subset \sfrak_i$, and
$\underline\sfrak_0\ge \ldots\ge \underline\sfrak_r$.
\end{itemize}
The sequence of sectors respectively the sequence of Weyl group elements $\df(\gamma)=(\tau_0,\ldots,\tau_r)$
(where $\underline{\tau_i(C^+)}= \underline\sfrak_i$) is called a defining chain for $\gamma$.
\end{dfn}
\begin{rem}\label{locminbutnotglobmexamposfol}\rm
A defining chain for a gallery is not necessarily unique. If $\tau_0=\ldots=\tau_r$,
then the gallery is obviously minimal. Note that the gallery $\gamma$ in Example~\ref{locminbutnotglobminexam}
is locally minimal and hence locally positively folded, i.e. the two-step galleries $(E_0\supset V_1\subset E_1)$ and 
$(E_1\supset V_2\subset E_2)$ are positively folded, but the gallery is not globally positively folded.
\end{rem}
\subsection{Local and global properties in special cases}
\label{sec:increaseconcat}
By Remark~\ref{locminbutnotglobmexamposfol} and Example~\ref{locminbutnotglobminexam} we see that
minimality and being positively folded are in general not local properties. In this section we will 
show now that there are many interesting cases where actually {\it local minimality} implies {\it global minimality}
and {\it locally positively folded} implies {\it globally positively folded}.

Fix a dominant coweight $\lam$ and let $\gamma_{\underline{\lambda}}$ 
be a concatenation of the galleries $\gamma_\omega$ associated to the fundamental 
coweights as in Example \ref{exGalCoweightarbi}. 
More precisely, recall that we fixed a total order on the set of fundamental coweights: $\omega_1,...,\omega_n$, and 
if $\lambda = \sum a_i\omega_i$, the associated minimal gallery $\gamma_{\underline{\lambda}}$ is
the concatenation of the correspondingly displayed galleries, see Example \ref{exGalCoweightarbi}:
$$
\begin{array}{rcl}
\gamma_{\underline{\lambda}} & = & 
\underbrace{\gamma_{\omega_1} *\cdots *\gamma_{\omega_1}}_{a_1 \hbox{\footnotesize times }} * 
\cdots *\underbrace{\gamma_{\omega_n} *\cdots *\gamma_{\omega_n}}_{a_n \hbox{\footnotesize times }}\\
 & & \\
 & = & (0 = V^c_0\subset E^c_0\supset V^c_1 \subset E^c_1 \supset
\cdots \supset V^c_r\subset E^c_r\supset V^c_{r+1} = \lambda) \ ,
\end{array}
$$ where the $V^c_j$'s and the $E^c_j$'s are vertices and faces of the dominant Weyl chamber. 

Let $\supp\lam$ be the set of nodes $N_i$ of the Dynkin diagram such that $a_i\not=0$.
We make a special assumption on the enumeration of the nodes: 
\begin{itemize}
  \item[$(*)$] If $N_i\in \supp\lam$, then none of the nodes $\{N_j\mid j<i\}$ is
  connected in the Dynkin diagram with one of the nodes $\{N_j\mid j>i\}$.
  
\end{itemize}
If the Dynkin diagram has no branches, i.e. the root system of $G$ is of type ${\tt A,B,C,F_4}$ or ${\tt G_2}$, 
then the Bourbaki enumeration of the nodes satisfies the property $(*)$ for all dominant coweights. If $G$ is of type ${\tt D}$ or ${\tt E}$
and $\supp \lam$ is contained in a subdiagram of type ${\tt A}$, then
it is easy to see that one can find an enumeration satisfying the condition $(*)$. 

\begin{proposition}
\label{prLocMiniEquivcomb}
Suppose the enumeration of the nodes of the Dynkin diagram of $G$ satisfies the condition 
$(*)$ for  $\supp\lam$. Let $\gamma\subset \A$ be a combinatorial one-skeleton gallery of 
the same type as $\gamma_{\underline\lam}$.
If $\gamma$ is locally positively folded, then $\gamma$ is globally positively folded.
\end{proposition}
\begin{rem}\rm
A locally positively folded combinatorial one-skeleton can be viewed as the gallery version 
of a weakly standard Young tableau defined by Lakshmibai, Musili and Seshadri in \cite{LMS}, \S 12. 
The proof below is an adaption of their proof that in special cases (like the ones above)
weakly standard Young tableaux are standard Young tableaux.
\end{rem}
\proof
Let $\gamma$ be a combinatorial one-skeleton gallery of the same type as 
$\gamma_{\underline\lam}$, say 
$$
\gamma=(V_0=\Lo\subset E_0\supset\ldots\subset E_r\supset V_{r+1})\subset\A.
$$
The gallery $\gamma$ is a concatenation $\gamma=\gamma_1*\ldots*\gamma_N$ of combinatorial one-skeleton 
galleries, each being of the same type as $\gamma_{\omega}$ for some fundamental weight $\omega$ 
corresponding to one of the nodes in the support of $\lam$. By abuse of notation we say that
an edge $E_i$ is {\it of weight type $\om_{E_i}$} if $E_i$ occurs in the concatenation within a one-skeleton gallery
of the same type as $\gamma_{\omega_{E_i}}$, and we say that $E_i$ is of {\it weight class} $\kappa_i\in W/W_{\om_{E_i}}$
if the ray $\br_{\ge 0}\kappa_i(\om_{E_i})$ coincides with the ray
starting in $V_i$ and passing through $V_{i+1}$, up to a displacement by $V_i$.
The gallery $\gamma$ is hence completely described by the sequence of Weyl group classes
$(\kappa_0,\ldots,\kappa_r)$. Further, given a sector $\sfrak$ with vertex $V_i$, then $E_i\subset \sfrak$
only if $\underline\sfrak=\tau(C^+)$ for an element $\tau\in W$ such that $\tau\equiv \kappa_i \bmod W_{\om_{E_i}}$.

It follows that to give a sequence of sectors $(\sfrak_0,\ldots,\sfrak_r)$ such that $\sfrak_i$ has vertex $V_i$ and
$E_i\subset \sfrak$ is equivalent to give a sequence of Weyl group elements $(\tau_0,\ldots,\tau_r)$ such that
$\tau_i\equiv \kappa_i \bmod W_{\om_{E_i}}$ for $i=0,\ldots,r$. The gallery is globally positively folded if and only
if one can choose the Weyl group elements such that in addition $\tau_0\ge \ldots\ge\tau_r$.

As a first step, we show that the local minimality implies for all $i=0,\ldots,r-1$ the 
existence of pairs $(\sigma_i,\eta_{i+1})\in W\times W$ such that $\sigma_i\ge \eta_{i+1}$,
$\sigma_i \equiv \kappa_i \bmod W_{\om_{E_i}}$ and 
$\eta_{i+1} \equiv \kappa_{i+1} \bmod W_{\om_{E_{i+1}}}$.
For the positively folded two-step gallery $(E_i\supset V_{i+1}\subset E_{i+1})$
let  $(E_i\supset V_{i+1}\subset F_0)$ be a corresponding minimal gallery with sector $\mathfrak t_0$,
i.e., $\mathfrak t_0$ has vertex $V_{i}$, $E_i\subset \mathfrak t_0$ and 
$F_0\subset \mathfrak t_0'= \mathfrak t_0(V_{i+1})$. 
If $F_0=E_{i+1}$, then set $\mathfrak t_1= \mathfrak t_0'$. 
If $F_0\not=E_{i+1}$, then let
$(\beta,n)$ be the affine root such that $F_1=s_{\beta,n}(F_0)$ is obtained by a positive folding.
Since $\ch_{\beta,n}$ separates $F_0$ and $C^-(V)$ from $F_1$, it separates also $\mathfrak t'_0$ and $C^-(V)$ from $\mathfrak t'_1=s_{\beta,n}(\mathfrak t'_0)$, so 
$\underline{\mathfrak t}_0\ge \underline{\mathfrak t}'_1$.
By repeating the argument if $E_{i+1}\not=F_1$, we obtain successively the sector $\mathfrak t_1$ with vertex $V_{i+1}$
such that $E_{i+1}\subset \mathfrak t_1$ and $\underline{\mathfrak t}_0\ge \underline{\mathfrak t}_1$. 
Let $\sigma_i,\eta_{i+1}\in W$ be such that $\underline{\sigma_i(C^+)}= \underline{\mathfrak t}_0$ and
$\underline{\eta_{i+1}(C^+)}= \underline{\mathfrak t}_1$, so $\sigma_i\ge \eta_{i+1}$ and $\sigma_{i} \equiv \kappa_i \bmod W_{\om_{E_i}}$,
$\eta_{i+1} \equiv \kappa_{i+1} \bmod W_{\om_{E_{i+1}}}$.

We start now to define the sequence of Weyl group elements $\tau_0,\ldots,\tau_r$
by choosing for $\tau_0\in W$ the maximal representative of the class $\kappa_0$.
Suppose we have already defined $\tau_0,\ldots,\tau_i\in W$ such that $\tau_0\ge \ldots\ge \tau_i$
and $\tau_j\equiv \kappa_j \bmod W_{\om_{E_j}}$ for $j=0,\ldots,i$. 
Let $k_0$ be such that the node $N_{k_0}$ corresponds to the fundamental weight $\om_{E_{i}}$,
let $I$ be the set of nodes $I=\{N_\ell\mid \ell < k_0\}$ and set $J=\{N_\ell\mid \ell> k_0\}$. Denote by $W_I$, $W_J$
and $W_{I\cup J}$ the subgroups of $W$ generated by the $s_{\alpha}$ associated to the simple roots
corresponding to the nodes in $I$, $J$ and $I\cup J$ respectively. The condition $(*)$ implies
that the elements in $W_I$ commute with the elements in $W_J$.
By abuse of notation we write $\bar\tau_j $ not only for the class of $\tau_j$ in $W/W_{\om_{E_i}}=W/W_{I\cup J}$,
but also for the minimal representative of this class in $W$. 
So we can write $\tau_j=\bar \tau_j x_j y_j$, where $x_j\in W_I$ and $y_j\in W_J$. Recall that $x_jy_j=y_jx_j$
by condition $(*)$. 

Since $\bar \tau_j x_j$ is a minimal representative in $W$ of the class $(\tau_j\bmod W_J)$, 
the inequalities $\tau_0\ge\ldots\ge\tau_i$ imply the inequalities $\bar \tau_0x_0\ge\ldots\ge \bar \tau_ix_i.$ 
Let now $y$ be the maximal element in $W_J$. Since $y$ and the $y_0,\ldots,y_i$ fix the fundamental weight
$\omega_\ell$ for $\ell <k_0$, we can assume without loss of generality $y_j=y$ for all $j=0,\ldots,i$,
because if one replaces the $y_j$ by $y$, then one still has the desired properties for all $j=0,\ldots,i$:
$$
\tau_0= \bar\tau_0x_0y\ge \tau_1= \bar\tau_1x_1y\ge\ldots\ge \tau_i= \bar\tau_ix_iy, \quad 
\text{and}\quad\tau_j\equiv \kappa_j\mod W/W_{\om_{E_j}}.
$$
 To extend the sequence and define $\tau_{i+1}$, we consider now the pair $\sigma_i\ge\eta_{i+1}$ defined 
at the beginning. Recall
that $\sigma_{i} \equiv \kappa_i=\bar\tau_i \bmod W_{\om_{E_i}}$,
$\eta_{i+1} \equiv \kappa_{i+1} \bmod W_{\om_{E_{i+1}}}$. We can write $\sigma_i=\bar\tau_ip_iq_i$ and 
$\eta_{i+1}=\bar\eta_{i+1} r_{i+1} t_{i+1}$, where $p_i,r_{i+1}\in W_I$, $q_i, t_{i+1}\in W_J$ and $\bar\eta_{i+1}$ 
denotes the class of $\eta_{i+1}$ in $W/W_{I\cup J}$ as well as the minimal representative of the class in $W$.

Set $\tau_{i+1}=\bar\eta_{i+1}t_{i+1},\quad\text{then}\quad \tau_{i+1}\equiv \kappa_{i+1} \bmod  W_{\om_{E_{i+1}}}$
because $r_{i+1}$ fixes $\om_{E_{i+1}}$. Further, $\tau_i=\bar\tau_i x_iy\ge \tau_{i+1}=\bar\eta_{i+1}t_{i+1}$ because
$$
(\sigma_i\mod W_{\om_{E_{i}}})=\kappa_i = \bar\tau_i \ge (\eta_{i+1}\mod W_{\om_{E_{i}}})=\bar \eta_{i+1}
$$
and, by construction, $y\ge t_{i+1}$. Proceeding by induction gives the desired defining chain.
\qed
\subsection{Semistandard Young tableaux and positively folded one-skeleton galleries}
To characterize the tableaux corresponding to positively
folded galleries, recall that the Bourbaki enumeration of the fundamental
coweights satisfies the condition $(*)$ in section~\ref{sec:increaseconcat} for the groups of type ${\tt A_n}$, ${\tt B_n}$ and ${\tt C_n}$.
\begin{proposition}\label{semisyoungtableau}
The bijection in Proposition~\ref{gallerytableau}
induces a bijection between the positively folded galleries and the semistandard tableaux.
\end{proposition}
\begin{proof}
By Proposition~\ref{prLocMiniEquivcomb}, a locally positively folded gallery is automatically globally positively folded.
Consider two consecutive faces of dimension one in the gallery: $(V_{i-1}\subset E_{i-1}\supset V_i\subset E_i)$.
Then either $E_{i-1}=V_{i-1}+\sigma(\{t\omega_j\mid t\in[0,1]\})$ or $E_{i-1}=V_{i-1}+\sigma(\{t\omega_j\mid t\in[0,\frac{1}{2}]\})$
for some $j$ and some $\sigma\in W/W_{\omega_j}$, and $E_{i}=V_{i}+\tau(\{t\omega_k\mid t\in[0,1]\})$ or $E_{i}=
V_{i}+\tau(\{t\omega_k\mid t\in[0,\frac{1}{2}]\})$
for $k=j$ or $k=j+1$ and some $\tau\in W/W_{\omega_k}$ (for a more precise description of the possible $\tau$ in the second case, see Equation (\ref{nonmini})). Denote by $C_{i-1},C_{i}$ the columns in the Young tableaux
corresponding to the weights $\sigma(\omega_j)$ and $\tau(\omega_k)$. It remains to show that the condition
{\it positively folded at $V_i$} is equivalent to the condition that the entries in the rows of the tableaux consisting of the
two columns $C_{i-1}$ (on the right side) and $C_{i}$ are weakly increasing.

It is easy to verify that the condition on the rows is equivalent to
$\sigma\ge \tau$ in $W/W_{\omega_j}$ if $\omega_j=\omega_k$ (see \cite{BilLakshmi}, Chapter 3), respectively there exists lifts
$\tilde\sigma\in W$ of $\sigma$ and $\tilde\tau\in W$ of $\tau $ such that $\tilde\sigma\ge \tilde\tau$ (see \cite{LMS}).
Suppose first $V_i$ is a special point. The condition $\tilde\sigma\ge \tilde\tau$ implies
the condition positively folded: one starts with the sector $V_i+\tilde\sigma(C^+)$ which contains
a conjugate of $E_i$ forming a minimal pair with $E_{i-1}$. If $\tilde\sigma> \tilde\tau$, then
one can find a sequence of reflections such that $\tilde\sigma>s_{\beta_1}\tilde\sigma>\ldots>\tilde\tau$
and the length decreases in each step by one. It follows that the corresponding folds at $V_i$
are all positive. The reverse direction is proved in the same way: start with a minimal pair $E_{i-1}\supset V_i\subset E_{i,0}$
for $E_{i-1}\supset V_i\subset E_i$, let $\tilde\sigma(C^+)$ be the chamber such that $E_{i-1}\subset V_{i-1}+\tilde\sigma(C^+)$
and $E_{i,0}\subset V_{i}+\tilde\sigma(C^+)$, applying the positive folds to the sector $V_{i}+\tilde\sigma(C^+)$
yields a sector $V_{i}+\tilde\tau(C^+)$ containing $E_i$, and, since the folds are positive, one has
$\tilde\sigma\ge \tilde\tau$. Since the sectors contain the faces we have $\tilde\sigma(\omega_j)=\sigma(\omega_j)$ 
and $\tilde \tau(\omega_k) =\tau(\omega_k)$, i.e. these are lifts for $\sigma$ and $\tau$. If $V_i$ is not a special point, then $W_{V_i}^v = (W_{V_i}^1)^v\times (W_{V_i}^2)^v$ is of type $\mathtt B_j\times\mathtt B_{n-j}$ or $\mathtt D_j\times\mathtt C_j$, where, in both cases, the second factor acts trivially. We have a bijection between the possible entries of $C_i$ and the orbit $(W_{V_i}^1)^v\cdot (\epsilon_1+\epsilon_2+\cdots + \epsilon_j)$, in the following way: Let $k_1,...,k_s, \bar \ell_1,...,\bar \ell_{j-s}$, $1\le k_p,\ell_q\le n$, 
be the entries of $C_{i-1}$, we order the set of integers $\{k_1,...,k_s, \ell_1,...,\ell_{j-s}\}$ in ascending order and we identify 
this linearly ordered set with $\{1,2,..., j\}$. With respect to this bijection, the columns $C_{i-1}, C_i$ correspond to 
$\sigma(\epsilon_1+\epsilon_2+\cdots + \epsilon_j)$ and $\tau(\epsilon_1+\epsilon_2+\cdots + \epsilon_j)$ for some
$\sigma,\tau\in (W_{V_i}1)^v/(W_{V_i}1)^v_{\epsilon_1+\cdots + \epsilon_j}$. 
Now again (see \cite{BilLakshmi}), the Bruhat order on the orbit and the row condition on pairs of colums coincide. So the same arguments, as above, show that positively folded and the row condition are equivalent.
\end{proof}


\section{Local minimality}\label{localmin}
The language of building theory allows us to translate the 
study of the intersection $Z_{\lam,\mu}:= G(\co).\lam\cap U^-(\ck).\mu$
into a problem of studying intersections of subsets of a Bott-Samelson variety $\Sigma(\gamma_\lam)$: 
$$
Z_{\lam,\mu}= G(\co).\lam\cap U^-(\ck).\mu=
\bigcup_{\substack{\delta\in \Gamma(t_{\gamma_\lam}, \Lo)\\ {\rm target}(\delta)=\mu}}
\{\text{minimal galleries}\}\cap C_\delta.
$$
Here $C_\delta$ denotes the Bia\l ynicki-Birula cell associated to the combinatorial gallery $\delta$, which,
in terms of building theory, is the same as the fiber over $\delta$ of the retraction $r_{-\infty}$.

To describe more precisely the intersection of the set of minimal galleries with such a cell,
we need to ``unfold'' $\delta$, i.e. we need to construct minimal galleries that retract onto $\delta$. 
As a first step we will, in this section, describe how to unfold two steps galleries. An important tool
will be the galleries of residue chambers.

\subsection{Positively folded galleries of chambers}
\label{susePositiveChambers}
Let $E$ and $F$ be one dimensional faces in $\cj^{\mathfrak a}$ containing a vertex $V$, let also $\sfrak$ be a 
sector with vertex $V$ containing $E$. Let $w_{\sfrak_V} = w(C_V^-,\sfrak_V) $ be the element in $W^v_V$ that sends 
$C^-_V$ to $\sfrak_V$. Among the residue chambers containing $F_V$ denote by $D$ the one closest to $C_V^-$.
Fix a reduced decomposition of $w_F = w(C_V^-,D) = s_{i_1}\cdots s_{i_r}$ in $W^v_V$ and let 
$\mathbf i = (i_1,..., i_r)$ be the type of the decomposition. 
We denote by $\al_{i_j}$ the simple 
root in $\Phi_V$ corresponding to $s_{i_j}$. For any root $\alpha\in \Phi_V$, 
$x_\alpha(\cdot)$ denotes the one-parameter 
additive subgroup of $H_V$ associated to $\alpha$, let $U_\alpha$ 
denote its image in $H_V$.

We consider now galleries of residue chambers $\mathbf c = (C_V^-,C_1,...,C_r)$ in the 
apartment $\A_V$ starting at $C_V^-$ and of type $\mathbf i$. The set of these galleries 
is in bijection with the set $\Gamma (\mathbf i) = \{1,s_{i_1}\}\times\cdots\times \{1,s_{i_r}\}$ 
via the map $(c_1,...,c_r)\mapsto (C_V^-, c_1 C_V^-,...,c_1\cdots c_r C_V^-)$. 
Let $\be_j = c_1\cdots c_j (\al_{i_j})$, then $\be_j$ is the root corresponding to the common 
wall $H_j = H_{\be_j}$ of $C_{j-1} =c_1\cdots c_{j-1} C_V^- $ and $C_j = c_1\cdots c_j C_V^-$. 
In the following, we shall identify a sequence $(c_1,...,c_r)$ and the corresponding gallery.

\begin{dfn}
A gallery $\mathbf c = (c_1,...,c_r)\in\Gamma(\mathbf i)$ is said to be positively folded with 
respect to $\sfrak_V$ if $c_j = 1$ implies $w_{\sfrak_V}^{-1} \be_j < 0$. We denote the set 
of positively folded galleries by $\Gamma_{\sfrak_V}^+(\mathbf i)$.
\end{dfn}

If $\sfrak_V = C_V^+$, a gallery $\mathbf c = (c_1,...,c_r)$ is positively folded with respect 
to $C_V^+$ if, and only if, the associated subexpression $(id, c_1,c_1c_2,...,c_1\cdots c_r)$ 
is distinguished, see Deodhar \cite{Deo}, Definition 2.3.

\begin{proposition}
A gallery $\mathbf c = (C_V^-,C_1,...,C_r)\in\Gamma(\mathbf i)$ is positively folded with respect 
to $\sfrak_V$ if, and only if, $C_j = C_{j-1}$ implies that the wall $H_j = H_{\be_j}$ separates 
$\sfrak_V$ from $C_j = C_{j-1}$.
\end{proposition}

\noindent{\it Proof. } We have the following equivalences: 

\noindent ($H_j$ separates $\sfrak_V$ from $C_j = C_{j-1}$) $\Longleftrightarrow$ ($w_{\sfrak_V}^{-1}H_j$ 
separates $C_V^-$ from $w_{\sfrak_V}^{-1}C_j = w_{\sfrak_V}^{-1}C_{j-1}$)  $\Longleftrightarrow$ 
($w_{\sfrak_V}^{-1} \be_j $ is a negative root).
\qed

The set of all galleries of chambers starting at $C_V^-$ of type $\mathbf i$ in the building 
$\cj^{\mathfrak a}_V$ has a structure of a smooth projective algebraic variety, which we denote by $\BS(\mathbf i)$. 
(In fact, it is a Bott-Samelson variety.)
To a gallery of chambers $\mathbf c = (c_1,...,c_r) = (C_V^-,C_1,...,C_r)$ in $\Gamma(\mathbf i)$, 
one can associate an open subset $\mathcal O_{\sfrak_V} (\mathbf c)$ and a cell 
$\mathcal C_{\sfrak_V}(\mathbf c)$ in the variety $\BS(\mathbf i)$. They are defined in the 
following way: for any $j\in\{1,...,r\}$, and any $a_j\in\mathbb C$, set $ o_j =x_{c_j(\al_{i_j})}(a_j) c_j $, 
then $\mathcal O_{\sfrak_V}(\mathbf c) = \{ (C_V^-=C'_0,C'_1,...,C'_r)\mid \forall j: \ C'_j = o_1\cdots o_j C_V^-\}$; 
further, set
$$
g_j = 
\left\{
\begin{array}{ll}
c_j & \hbox{ if } w_{\sfrak_V}^{-1} \be_j > 0 \\
x_{c_j(\al_{i_j})}(a_j) c_j & \hbox{ if } w_{\sfrak_V}^{-1} \be_j < 0.
\end{array}\right.
$$ 
then $\mathcal C_{\sfrak_V}(\mathbf c) = \{ (C_V^-=C'_0,C'_1,...,C'_r)\mid \forall j: \ C'_j = g_1\cdots g_j C^- \}$.  
The minimal galleries in $\mathcal C_{\sfrak_V}(\mathbf c)$
are those such that for any $j$: $C'_{j-1}\ne C'_j$, i.e. $c_j\not=1$ if $w_{\sfrak_V}^{-1} \be_j > 0$,
and $a_j\ne 0$ if $c_j=1$ and $w_{\sfrak_V}^{-1} \be_j < 0$.
We denote the set of minimal galleries by $\mathcal C^m_{\sfrak_V}(\mathbf c)$.
\begin{lemma}
\label{reMinCell}
The set $\mathcal C^m_{\sfrak_V}(\mathbf c)$ is empty if the gallery $\mathbf c$
is not positively folded with respect to $\sfrak_V$. If $\mathbf c$ is positively folded 
with respect to $\sfrak_V$, then $\mathcal C^m_{\sfrak_V}(\mathbf c)$ is isomorphic to:
$$
\mathcal C^m_{\sfrak_V}(\mathbf c) \simeq \mathbb C^{t(\mathbf c)}\times (\mathbb C^*)^{r(\mathbf c)}
$$ where 
$$
t(\mathbf c) = \sharp\{j\mid c_j = s_{i_j} \hbox{ and } w_{\sfrak_V}^{-1} \be_j < 0\},\ 
r(\mathbf c) = \sharp\{j\mid c_j = 1 \hbox{ and } w_{\sfrak_V}^{-1} \be_j < 0\}.
$$
\end{lemma}
\begin{proposition}
The cell $\mathcal C_{\sfrak_V}(\mathbf c)$ identifies with $r_{\sfrak_V}^{-1}(\mathbf c)$, 
where $r_{\sfrak_V} : \cj^{\mathfrak a}_V \to \A_V$ is the retraction centered at $\sfrak_V$.
\end{proposition}

\noindent{\it Proof. } For any chamber $C'$, the retraction can be defined as 
$r_{\sfrak_V}(C') = \lim_{s\to 0} s^{\theta} C'$, where $\theta$ is a regular coweight 
contained in $\sfrak_V$. To simplify, we take $\theta = w_{\sfrak_V}(-\rho^\vee)$. 
Further, the retraction applies componentwise to the galleries, whence 
$r_{\sfrak_V}(\mathbf g) = (C_V^-, r_{\sfrak_V}(C'_1),..., r_{\sfrak_V}(C'_r))$. 
For any $j$, $r_{\sfrak_V}(C'_j) = \lim_{s\to 0} s^{w_{\sfrak_V}(-\rho^\vee)} g_1\cdots g_j C_V^- 
= \lim_{s\to 0} g'_1\cdots g'_j C_V^-$, where
$$
g'_j = 
\left\{
\begin{array}{ll}
c_j & \hbox{ if } w_{\sfrak_V}^{-1} \be_j > 0 \\
x_{c_j(\al_{i_j})}(s^{\langle c_j(\al_{i_j}), \ c_{j-1}\cdots c_1 w_{\sfrak_V}(-\rho^\vee)\rangle}a_j) c_j & 
\hbox{ if } w_{\sfrak_V}^{-1} \be_j < 0.
\end{array}\right.
$$ But $\langle c_j(\al_{i_j}), c_{j-1}\cdots c_1 w_{\sfrak_V}(-\rho^\vee)\rangle 
= \langle w_{\sfrak_V}^{-1}\be_j, -\rho^\vee\rangle$. Therefore, 
$\mathcal C_{\sfrak_V}(\mathbf c)\subset r_{C_{\sfrak_V}}^{-1}(\mathbf c)$.
One sees in the same way that $\mathcal C_{\sfrak_V}(\mathbf c)\supset r_{C_{\sfrak_V}}^{-1}(\mathbf c)$.
\qed
\bigskip
\begin{rem}\rm
The cells define a Bia\l ynicki-Birula decomposition of the variety of all galleries of chambers $\BS(\mathbf i)$.
In fact, $\BS(\mathbf i) = \coprod_{\mathbf c\in\Gamma(\mathbf i)} \mathcal C_{\sfrak_V}(\mathbf c)$. 
\end{rem}

\subsection{Two steps minimal one-skeleton galleries}

\bigskip
\begin{theorem}
\label{thmLocMin}
Let $(E\supset V\subset F)$ be a two steps one-skeleton gallery in $\A$. There exists a minimal gallery 
$(E\supset V\subset E')$ in $\cj^{\mathfrak a}$ such that $E'$ has the same type as $F$ and $r_{-\infty}(E') = F$ if, and only if, 
$(E\supset V\subset F)$ is positively folded.
\end{theorem}

We divide the proof of Theorem \ref{thmLocMin} into four lemmas. 
Choose a chamber $D$ containing $F_V$ and
let $w$ be the element that sends $C_V^-$ to $D$.
 
\begin{lemma}
\label{lemminimalpositiv}
Suppose there exists a minimal one-skeleton gallery $(E\supset V\subset E')$ such that $r_{-\infty}(E') = F$.
Let $\sfrak$ be a sector in $\A$ with vertex $V$ containing $E$ such that $E'\subset -\sfrak$, in any apartment 
containing $\sfrak$ and $E'$. Then
one can find a minimal gallery of residue chambers $\mathbf m'$ of type $\mathbf i = (i_1,...,i_r)$
between $C^-_V$ and $E'_V$ such that 
\begin{itemize}
\item[{\it i)}] $w = s_{i_1}\cdots s_{i_r}$ is a reduced decomposition,
\item[{\it ii)}] $\mathbf c = r_{\sfrak_V} (\mathbf m')\subset\A_V$ is a positively folded gallery of residue chambers with 
respect to $\sfrak_V$, 
\item[{\it iii)}] $(E_V, F_V')$ is a minimal pair, where $F_V' = r_{\sfrak_V}(E'_V)$ and $F'_V$ is of the same type as $F_V$.
\end{itemize}
\end{lemma}

\noindent{\bf Proof. } The fact that $E'$ has the same type as $F$ is a consequence 
of $r_{-\infty}(E') = F$. Transferred to the setting of the residue building, 
the retraction $r_{-\infty}$ identifies with the retraction centered at $C^-_V$ of 
$\cj_V^{\mathfrak a}$ onto $\A_V$, so $r_{C_V^-}(E'_V) = F_V$. Since this retraction preserves the distances from $C^-_V$, any minimal 
gallery $\mathbf m'=(C_V^-, C'_1,...,C'_r)$ of residue chambers in $\cj^{\mathfrak a}_V$
from $C^-_V$ to $E'_V$ (in any apartment containing those two) retracts onto a minimal gallery 
from $C^-_V$ to $F_V$, say of type $\mathbf i = (i_1,...,i_r)$. Further, one can choose $C'_r$ such that $r_{C_V^-}(C'_r) = D$. Since the gallery is minimal,
it follows that $\mathbf i = (i_1,...,i_r)$ corresponds to a reduced decomposition 
of the element $w$.

Consider the variety of galleries $\BS(\mathbf i)$. 
The gallery $\mathbf m'$ belongs to the cell $\mathcal C_{\sfrak_V}(\mathbf c)$, 
where $\mathbf c =r_{\sfrak_V} (\mathbf m') = (C^-_V,C_1,...,C_r)$. 
Let us suppose that $C_j = C_{j-1}$. Let us moreover assume (without loss of generality) 
that the gallery $\mathbf m'$ is already retracted 
until the index $j-1$, meaning that $\mathbf m' = (C^-_V,C_1,...,C_{j-1}, D_j,...,D_r)$, where $(C_{j-1},D_j,...,D_r)$ 
is a minimal gallery retracting onto $(C_{j-1},C_j,...,C_r)$. Suppose that $H_j$ does not separate $\sfrak_V$ from 
$C_j = C_{j-1}$. The chambers $C_{j-1}$ and $D_j$ have to be distinct by the assumption on the minimality, 
so $C_j=r_{\sfrak_V}(D_j)$ and ${\sfrak_V}$ can not be on the same side of $H_j$, contradicting
the assumption $C_j = C_{j-1}$ are not separated from $\sfrak_V$ by $H_j$. 
It follows that the gallery of chambers $\mathbf c$ is positively folded, i.e.,
$\mathbf c\in \Gamma_{\sfrak_V}^+(\mathbf i)$.

Let $r_{\infty,\sfrak}$ be the retraction from $\infty$, but now with respect to the sector $\sfrak$.
On the level of the residue building, the retraction $r_{\infty,\sfrak}$ identifies with the retraction 
centered at $\sfrak_V$ of  $\cj_V^{\mathfrak a}$ onto $\A_V$. So if we set $r_{\infty,\sfrak}(E') = F'$,
then $F_V'=r_{\sfrak_V}(E'_V)$ and we get a minimal pair $(E_V,F'_V)$ in $\cj^{\mathfrak a}_V$.

The face $E'_V$ is contained in the opposite of $\sfrak_V$ in any apartment containing $E_V$ and
$E'_V$, and $r_{\sfrak_V}$ preserves the distance from $\sfrak_V$. It follows that $F_V'=r_{\sfrak_V}(E'_V)$ 
is contained in $-\sfrak_V$, and hence we get a minimal pair $(E_V,F'_V)$ in $\cj^{\mathfrak a}_V$. 
Since the type of $E'_V$ and $F_V$ are the same and the type of $E'_V$ and $F'_V$
are the same, this finishes the proof of the lemma. 
\qed
\begin{lemma}
\label{lemminimalpositiv2}
If there exists a minimal one-skeleton gallery $(E\supset V\subset E')$ such that $r_{-\infty}(E') = F$,
then the one-skeleton gallery $(E\supset V\subset F)$ is positively folded.
\end{lemma}

\noindent{\bf Proof. } 
Let $\mathbf c = r_{\sfrak_V} (\mathbf m')\subset\A_V$ be the positively folded gallery of residue chambers with 
respect to $\sfrak_V$ described in Lemma~\ref{lemminimalpositiv}.
By construction, unfolding $\mathbf c$ gives a minimal gallery from $C_V^-$ to $F_V$. We will see
that this unfolding procedure shows that $(E\supset V\subset F)$ is positively folded.

The procedure works as follows:
Let $\{j_1<\cdots < j_k\}\subset \{1,...,r\}$ be the indices where $\mathbf c$ is folded. Then 
we unfold the gallery of chambers starting with the fold at the wall $H_{j_1}$, the resulting gallery
will then still have a fold at $s_{H_{j_1}}({H_{j_2}})$, we unfold the gallery at this wall etc.
The face $F'_V$ will be reflected each time and we get
$$
\begin{array}{rcl}
F_V & = & s_{H_{j_1}}\cdots  s_{H_{j_{k-1}}}s_{H_{j_k}} (s_{H_{j_1}}\cdots  s_{H_{j_{k-1}}})^{-1} 
\cdots s_{H_{j_1}}s_{H_{j_2}}s_{H_{j_1}} s_{H_{j_1}} F'_V\\
 & = & \tau_{k}\cdots \tau_{1} F'_V\ ,
\end{array}
$$ 
where $\tau_{l} = s_{H_{j_1}}\cdots  s_{H_{j_{l-1}}}s_{H_{j_l}} (s_{H_{j_1}}\cdots  s_{H_{j_{l-1}}})^{-1}$. 
To see that $(E\supset V\subset F)$ is positively folded, it remains
to prove that each time the face is reflected away from $C^-_V$. 

First recall that $\mathbf c$ is positively folded, so for each folding step we have the chambers
$C_{j_k}=C_{j_k-1}$ and $\sfrak_V$ lie within different half-spaces with respect to the wall $H_{j_k}$.
Further, since $F_V'\subset -\sfrak_V$, the chambers $C_{j_k}=C_{j_k-1}$ and the face $F'_V$ lie
within the same half-space. We use the suggestive notation
$$
F'_V, C_{j_k} = C_{j_k-1}\quad\mid_{H_{j_k}}\quad \sfrak_V\supset E_V\ ,
$$
for this situation.

The gallery of chambers $\mathbf c$ starts at $C_V^-$ and is folded for the first time at
the hyperplane $H_{j_1}$. It follows that the chambers $C_V^-$ and $C_{j_1} = C_{j_1-1}$, and hence
also $F'_V $, are within the same half-space with respect to $H_{j_1}$:
$$
C_V^-,C_{j_k} = C_{j_k-1}, F'_V \quad\mid_{H_{j_k}}\quad \sfrak_V\supset E_V.
$$
Thus, after the first unfolding, we have: 
$$
C^-_V, C_1, ..., C_{j_1-1}\quad \mid_{H_{j_1}}\quad \tau_{1}(C_{j_1}), \tau_{1}(F'_V)\ ,
$$ 
meaning that the chambers $C^-_V, C_1, ..., C_{j_1-1}$ are separated from 
$\tau_{1}(C_{j_1})$ and from $\tau_{1}(F'_V)$ by the wall $H_{j_1}$ (note that 
the face $F'_V = \tau_{1}(F'_V)$ may be contained in the wall $H_{j_1}$). In particular,
either $F'_V$ is fixed by the reflection or is reflected away from $C^-_V$.
The gallery 
$$
\mathbf c^1 = (C^-_V,C_1,...,C_{j_1-1}, \tau_{1}(C_{j_1}),..., \tau_{1}(C_{j_2-1}), \tau_{1}(C_{j_2}),..., \tau_{1}(C_{r}))
$$ 
is now minimal up to the index $j_2-1$. Moreover, we know that  
$$
F'_V, C_{j_2-1} = C_{j_2}\quad\mid_{H_{j_2}}\quad \sfrak_V\supset E_V\ ,
$$ 
applying $\tau_{1}$, we get 
$$
\tau_{1}(F'_V), \tau_{1}C_{j_2-1} = \tau_{1}C_{j_2} \quad\mid_{\tau_{1}H_{j_2}}\quad \tau_{1}E_V\ .
$$ 
The gallery of chambers $\mathbf c^1$ is folded for the first time at the hyperplane $\tau_{1}H_{i_2}$,
so $C^-_V$ and $\tau_{1}C_{j_2-1} = \tau_{1}C_{j_2}$, and hence also $\tau_{1}(F'_V)$
are on the same side of $\tau_{1}H_{i_2}$. Therefore, when we unfold with respect 
to $\tau_{1}H_{i_2}$, this wall separates $C^-_V$ and $\tau_1 (F'_V)$ from $\tau_2\tau_1 (F'_V)$. 
This procedure can be iterated to show that at each step the image of $F'_V$ is folded away from $C^-_V$,
which proves that $(E\supset V\subset F)$ is positively folded. 
\qed
\vskip 5pt
\noindent{\bf Proof of Theorem~\ref{thmLocMin}: ``{$\mathbf{\Rightarrow}$}''}.
Lemmas \ref{lemminimalpositiv} and \ref{lemminimalpositiv2} show the 
existence of a minimal one-skeleton gallery 
$(E\supset V\subset E')$ such that $r_{-\infty}(E') = F$ implies that the 
one-skeleton gallery $(E\supset V\subset F)$ 
is positively folded.
\qed
\vskip 5pt

Let $(E\supset V\subset F)$ be a positively folded one-skeleton gallery. Let $\sfrak\subset \A$ be a sector with vertex $V$ containing $E$. Choose a chamber $D$ containing $F_V$ and let
$w_D$ be the element that sends $C^-_V$ to $D$. 
Let $\mathbf i = (i_1,...,i_r)$ be the type of a reduced 
decomposition of $w_D = s_{i_1}\cdots s_{i_r}$ in $W^v_V$. 

\begin{lemma}
\label{leReciproc1}
For $w\le w_D$ let $F'_V$ be the a face of $w(C^-_V)$ of the same type as $F_V$. Then there exists a gallery
of chambers $\mathbf c = (C_V^-,C_1,...,C_r)$ of type $\mathbf i$, positively folded with respect to $-w(C^-_V)$,
such that $F_V'\subset C_r$.
\end{lemma}

\noindent{\bf Proof. } 
Let $\mathbf m$ be a minimal gallery of type $\mathbf i = (i_1,...,i_r)$ between $C_V^-$ and $D\supset F_V$.
By the subword property, there exists a folded gallery $\mathbf d = (C_V^-,D_1,...,D_r)$ of 
type $\mathbf i$ in $\A_V$ such that $D_r\supset F'_V$.

Suppose the gallery is not positively folded with respect to $-w(C^-_V)$.
Let $j$ be the smallest index such that $-w(C^-_V)$ and $D_j = D_{j+1}$ are on the 
same side of the wall $H_{i_j}$ of type $i_j$. 

The last chamber $D_r$ contains $F'_V\subset w(C^-_V)$, and $w(C^-_V)$ lies within the other 
half-space defined by $H_{i_j}$. It follows that the gallery $\mathbf d$ has to meet
$H_{i_j}$ for some index larger than $j$. Let $j_{max}= \max_{k}\{H_{i_k}=H_{i_j}, k>j\}$ or 
set $j_{max} = r$ if $H_{i_j}\supset F'_V$.

Consider the new gallery of type $\mathbf i$, $\mathbf d' = (C_V^-,D'_1,...,D'_r)$ defined by : 
$$
D'_k = \left\{
\begin{array}{ll}
D_k & \hbox{ if } k\leqslant j\\
s_{H_{i_j}}(D_k) & \hbox{ if } j+1\leqslant k\leqslant j_{max}\\
D_k & \hbox{ if } k> j_{max}\ .\\
\end{array}
\right.
$$ 
This gallery still has the property that the last chamber contains $F'_V$: $D'_r\supset F'_V$, and the gallery
is now positively folded with respect to $-w(C^-_V)$ till the index $i_j$. 
By repeating the procedure if necessary, one obtains a gallery 
$\mathbf c = (C_V^-,C_1,...,C_r)\in\Gamma_{-w(C^-_V)}^+(\mathbf i)$ such that  $F'_V \subset C_r$. 
\qed

\begin{lemma}
\label{leReciproc2}
Let $F'_V$ be the face of $-\sfrak_V$ of the same type as $F_V$. Then there exists a face $E'_V$ of the same type as $F_V$ such that $(E_V,E'_V)$ is a minimal pair in $\cj_V^{\mathfrak a}$ and $r_{C^-_V}(E'_V) = F_V$.
\end{lemma}

\noindent{\bf Proof. } Because $(E\supset V\subset F)$ is positively folded and $\sfrak \supset E$, the chamber $-\sfrak_V$ is closer to $C^-_V$ than $D$. Therefore $w = w(C^-_V,-\sfrak_V) \leq w_D$. So we can apply Lemma \ref{leReciproc1} to get a gallery of chambers $\mathbf c=(C_V^-,C_1,...,C_r)$ of type $\bf i$ such that $\mathbf c$ is positively folded with respect to $\sfrak_V$ and $F_V\subset C_r$.

According to the preceding section (see Lemma~\ref{reMinCell} and before), 
there exist a minimal gallery $\mathbf m = (C^-_V,C'_1,...,C'_r)$ in the cell $\mathcal C_{\sfrak_V}(\mathbf c)$, and
the chambers $C_j'$ can be described as $C'_j = g_1\cdots g_j C_V^-$ where 
$g_j = c_j$ or $x_{c_j(\al_{i_j})}(a_j) c_j $, and 
$c_j\not=1$ if $w_{\sfrak_V}^{-1} \be_j > 0$,
and $a_j\ne 0$ if $c_j=1$ and $w_{\sfrak_V}^{-1} \be_j < 0$.

Let $E'_V$ be the face of the same type as $F'_V$ contained in $C'_r$. First, we note that the minimality of the gallery $\mathbf m = (C^-_V,C'_1,...,C'_r)$ and the fact that $r_{C^-_V}(\mathbf m) = \mathbf c$ ensures that $r_{C^-_V}(E'_V) = F_V$. Second, we are going to prove that $\sfrak_V$ and $E'_V$ are contained in the apartment $g\A_V$, 
with $g = g_1\cdots g_r$, and, in this apartment, $E'_V$ is contained in the chamber opposite $\sfrak_V$. 

The proof is by an inductive procedure. We show that, for all $j\in\{1,...,r\}$, $\sfrak_V$ and $g_1\cdots g_j C_V^-$ 
are in the apartment $g_1\cdots g_j \A_V$. We write in the following just $H_j$ for
the common wall $ H_{\beta_j}$ of $C_{j-1}$ and $C_j$ of type $i_j$.

By assumption, $\bf c$ is a
positively folded gallery with respect to $\sfrak_V$, so there are three possible relative position for 
$\sfrak_V$, $C_V^-$ and $C_1$ with respect to $H_1$:

1) $\sfrak_V$ and $C_V^-$ are on the same side of $H_1$ and $C_1$ not, then 
$C'_1 = g_1C_V^- = x_{-\al_{i_1}}(a_1)s_{i_1} C_V^- = x_{-\al_{i_1}}(a_1) C_1$. 
But $x_{-\al_{i_1}}(a_1)$ pointwise stabilizes the halfspace bounded by $H_1$ 
containing $C_V^-$, hence $x_{-\al_{i_1}}(a_1)(\sfrak_V) = \sfrak_V$ and $C'_1$ 
are in the apartment $g_1\A_V$;

2) $\sfrak_V$ and $C_V^- = C_1$ are separated by $H_1$, then 
$C'_1 = g_1C_V^- = x_{\al_{i_1}}(a_1) C_V^-$ but $x_{\al_{i_1}}(a_1)$ pointwise 
stabilizes the halfspace bounded by $H_1$ not containing $C_V^-$, hence 
$\sfrak_V$ and $C'_1$ are in the apartment $g_1\A_V$;

3) $\sfrak_V$ and $C_1$ are on the same side of $H_1$ and $C_V^-$ not, 
then $w_{\sfrak_V}$ has a reduced decomposition that starts with $s_{i_1}$, 
$w_{\sfrak_V} = s_{i_1}u$, so $w_{\sfrak_V}^{-1}(-\al_{i_1}) > 0$, whence 
$g_1 = c_1 = s_{i_1}$ and $\sfrak_V$ and $C'_1 = s_{i_1}C_V^-$ are in the apartment $g_1\A_V$.

\noindent 
By induction we assume now that the chambers $\sfrak_V$ and $g_1\cdots g_{j-1} C_V^-$ are in the 
apartment $A_{j-1} = g_1\cdots g_{j-1} \A_V$. Again, we have three possible relative positions
for $\sfrak_V, C_{j-1}$ and $C_{j}$:

1) $\sfrak_V$ and $C_{j-1}$ are on the same side of $H_j$ and $C_j$ not, then 
$\sfrak_V$ and $C'_{j-1}$ are on the same side of $g_1\cdots g_{j-1} H_j$ in $A_{j-1}$, and 
$$
\begin{array}{rcl}
C'_j & =  &g_1\cdots g_{j-1} x_{-\al_{i_j}}(a_j)s_{i_j} C_V^- \\
 & = & g_1\cdots g_{j-1} x_{-\al_{i_j}}(a_j)s_{i_j} (g_1\cdots g_{j-1})^{-1} C'_{j-1}\\
 & = &  g_1\cdots g_{j-1} x_{-\al_{i_j}}(a_j) (g_1\cdots g_{j-1})^{-1} g_1\cdots g_{j-1} s_{i_j} (g_1\cdots g_{j-1})^{-1} C'_{j-1},
\end{array}
$$  where $g_1\cdots g_{j-1} s_{i_j} (g_1\cdots g_{j-1})^{-1} C'_{j-1}$ is the chamber adjacent to 
$C'_j$ along $g_1\cdots g_{j-1} H_j$ in $A_{j-1}$. Moreover, $g_1\cdots g_{j-1} x_{-\al_{i_j}}(a_j)  
(g_1\cdots g_{j-1})^{-1}$ pointwise stabilizes the halfspace bounded by $ g_1\cdots g_{j-1} H_j$ 
containing $C'_{j-1}$ and $\sfrak_V$. So $\sfrak_V$ and $C'_j$ are in the apartment $g_1\cdots g_j \A_V$.

2) $C_{j-1} = C_j$ and $\sfrak_V$ are separated by $H_j$, then
$C'_{j-1}$ and $\sfrak_V$ are separated by $g_1\cdots g_{j-1} H_j$ in $A_{j-1}$, and $\sfrak_V$ and the 
chamber 
$$
g_1\cdots g_{j-1} s_{i_j} (g_1\cdots g_{j-1})^{-1} C'_{j-1}
$$ 
are on the same side of this wall. Moreover, for $a_j\ne 0$
$$
C'_j =  g_1\cdots g_{j-1} x_{\al_{i_j}}(a_j) C_V^- = g_1\cdots g_{j-1} 
x_{\al_{i_j}}(a_j)(g_1\cdots g_{j-1})^{-1} C'_{j-1} 
$$ 
is a chamber adjacent to $C'_{j-1}$ along $g_1\cdots g_{j-1} H_j = g_1\cdots g_{j-1} x_{\al_{i_j}}(a_j) H_j$ 
in $g_1\cdots g_j \A_V$. 
The root-subgroup  $g_1\cdots g_{j-1} x_{\al_{i_j}}(a_j)(g_1\cdots g_{j-1})^{-1}$ pointwise 
stabilizes the halfspace bounded by $ g_1\cdots g_{j-1} H_j$ and containing the chamber
$ g_1\cdots g_{j-1} s_{i_j} (g_1\cdots g_{j-1})^{-1} C'_{j-1}$. So $\sfrak_V$ and $C'_j$ 
are in the apartment $g_1\cdots g_j \A_V$.
 
3) $\sfrak_V$ and $C_j$ are on the same side of $H_j$ and $C_{j-1}$ not, then 
$w_{\sfrak_V}^{-1} \be_j > 0 $ and so $C'_j =g_1\cdots g_{j-1} s_{i_j} C_V^-$. Whence  
$\sfrak_V$ and $C'_j$ are in the apartment $g_1\cdots g_j \A_V$.

Therefore $\sfrak_V$ and $E'_V$ are contained in the apartment 
$g\A_V = g_1\cdots g_r\A_V$, and in this apartment $E'_V$ is the 
image of the face $F'_V$ contained in $-\sfrak_V$. More precisely, 
$E'_V  = g \phi_V^- = b w_{-\sfrak_V} \phi_V^- = b F'_V$, where $\phi_V^-$ is 
the face having the type of $F'_V$ contained in $C_V^- $, $F'_V = w_{-\sfrak_V} \phi_V^-$ 
and $b\in B_{\sfrak_V}  = Stab_{H_V}(\sfrak_V)$. This element is obtained as follows: 
\begin{equation}
\label{eqExplicitMiniGal}
\begin{array}{rcl}
g & = & g_1\cdots g_r \\
 & = & x_{c_1(\al_{i_1})}(a_1) c_1\cdots x_{c_r(\al_{i_r})}(a_r) c_r  \\
 & = & x_{\beta_1}(\pm a_1) \cdots x_{\beta_r}(\pm a_r) c_1\cdots c_r\\
 & = & x_{\beta_1}(\pm a_1) \cdots x_{\beta_r}(\pm a_r) w_{-\sfrak_V} \ .
\end{array}
\end{equation}
As $E_V$ and $F'_V$ are in opposite chambers in $\A_V$, so are $E_V$ and $E'_V$ in $g\A_V$.
Let $E'\subset \cj^{\mathfrak a}$ be the one dimensional face such that $V\subset E'$ and $E'_V$
is the associated face in the residue building $\cj_V^{\mathfrak a}$. 
Let $r_{\infty,\sfrak}$ be the retraction from $\infty$, but now with respect to the sector $\sfrak$.
On the level of the residue building, the retraction $r_{\infty,\sfrak}$ identifies with the retraction 
$r_{\sfrak_V}$ centered at $\sfrak_V$ of  $\cj_V^{\mathfrak a}$ onto $\A_V$. Since $E_V'$ retracts with respect 
to $r_{\sfrak_V}$ onto $F'_V$ in $\A_V$, $E'$ retracts with respect 
to $r_{\infty, \sfrak}$ onto $F'$ in $\A$. The retraction is distance preserving with respect
to $\sfrak$,  so the fact that $E$ and $F'$ are in opposite sectors implies that the same holds
for $E$ and $E'$.
In other words, $(E_V, E'_V)$ is a minimal pair. 
\qed

\medskip
\noindent{\bf Proof of Theorem~\ref{thmLocMin}: ``{$\mathbf{\Leftarrow}$}''}.
Since $(E\supset V\subset F)$ is positively folded, there exists  a sector $\sfrak\supset E$ with vertex $V$ and a face $F'\supset V$ of the same type as $F$ such that $F'\subset -\sfrak$. Therefore, we can apply Lemmata \ref{leReciproc1} and \ref{leReciproc2} to get a minimal pair $(E_V,E'_V)$, 
with $r_{C^-_V}(E'_V) = F_V$, in other words a minimal gallery $(E\supset V\subset E')$, with $r_{-\infty}(E') = F$. The fact that $E'$ has the same type as $F$ is a consequence of $r_{-\infty}(E') = F$.
\qed

\bigskip
\begin{dfn}
\label{dfn:Min}
Given a two-step gallery $(E\supset V\subset F)$ in $\cj^{\mathfrak a}$, 
denote by $Min(E,F)$ the set of all faces $E'\supset V$ such that $r_{-\infty}(E') = F$ and $(E\supset V\subset E')$ is 
minimal. This set can be identified with the set of all faces $E'_V$ such that $r_{C_V^-} (E'_V) = F_V$ 
and $(E'_V,E_V)$ is a minimal pair. 
\end{dfn}
We assume now that $(E\supset V\subset F)$ is positively folded, 
we want to give this set an algebraic structure as an open subset of a union of cells in a Bott-Samelson variety.

We use the same notation as in section~\ref{susePositiveChambers}.
Let $\sfrak$ be a sector containing $E$ and let $w_{\sfrak_V} = 
w(C_V^-,\sfrak_V) $ be the element in $W^v_V$ that sends 
$C^-_V$ to $\sfrak_V$. Let $D$ be the chamber containing $F_V$ the 
closest to $C_V^-$. Since $(E\supset V\subset F)$ is positively folded, 
$w_{-\sfrak_V} = w(C_V^-, -\sfrak_V) \leq w_D = w(C_V^-, D)$. Fix a 
reduced decomposition of $w_D = s_{i_1}\cdots s_{i_r}$ in $W^v_V$ and denote its type by 
$\mathbf i = (i_1,..., i_r)$. 

We denote by $\Gamma_{\sfrak_V}^+(\mathbf i, op)$ the set of all galleries 
$\mathbf c = (C_V^-,C_1,...,C_r)$ of residue chambers of type $\mathbf i$ 
which are positively folded with respect to $\sfrak_V$ and have the property that the 
face $F'_V$ of the same type as $F_V$ contained in $C_r$ forms a minimal pair with 
$E_V$ in $\A_V$. 

\begin{proposition}
\label{prMinLocal}
The set $Min(E,F)$ is in bijection with the disjoint union 
$\coprod_{\mathbf c\in\Gamma_{\sfrak_V}^+(\mathbf i,op)} \mathcal C^m_{\sfrak_V}(\mathbf c)$, 
where $ \mathcal C^m_{\sfrak_V}(\mathbf c)$ is the set of all minimal galleries in the cell $ \mathcal C_{\sfrak_V}(\mathbf c)
\subset \BS(\mathbf i)$. 
\end{proposition}

\noindent{\bf Proof. }
First recall that $Min(E,F)$ identifies with the set of all faces $E'_V$ such that $(E_V,E'_V)$ is 
a minimal pair and $r_{C_V^-}(E'_V) = F_V$. Next, the proof of Lemma \ref{leReciproc2} asserts that to a 
minimal gallery $\mathbf m'\in \mathcal C_{\sfrak_V}(\mathbf c)$ corresponds such a unique 
face $E'_V = g(\mathbf m')$. It is the face of the same type as $F_V$ contained in the last 
chamber of $\mathbf m'$. Lemma \ref{lemminimalpositiv} shows that this mapping $g$ is surjective. 
Suppose now that $\mathbf m',\mathbf n'\in \mathcal C_{\sfrak_V}(\mathbf c)$ are two minimal 
galleries such that  $E'_V = g (\mathbf m') = g(\mathbf n')$. Since $D$ is the closest chamber to 
$C_V^-$ containing $F_V$ the last chambers of $\mathbf m'$ and of $\mathbf n'$ have to be the 
same. Since they have the same type and the same origin, $\mathbf m' = \mathbf n'$.
\qed


\section{From local properties to global properties}\label{localglobalproperties}
In Theorem~\ref{thmLocMin} we have shown that one can obtain a minimal
two steps gallery by ``unfolding'' a combinatorial two steps gallery
$(E\supset V\subset F)$ only if the latter is positively folded. This provides 
a procedure to unfold a locally positively folded combinatorial one-skeleton gallery 
inductively to get a locally minimal one-skeleton gallery. The first aim
of this section is to show that if one starts with a globally positively folded
gallery, then this unfolding algorithm produces automatically globally minimal one-skeleton galleries. Next we derive a formula for the polynomials $L_{\lambda,\mu}$.

\subsection{From positively folded two-steps galleries to minimal galleries}

\begin{proposition}\label{coro:non-empty}
Let $\delta = [\delta_0,\delta_1,...,\delta_r] = (V_0\subset E_0\supset\cdots\subset 
E_r\supset V_{r+1}) \in\Gamma(\gamma_\lam)$. The intersection $\{\text{minimal galleries}\}\cap C_\delta$ 
is non-empty if, and only if, $\delta$ is positively folded.
\end{proposition}

\noindent{\bf Proof. } Let $\gamma = (V_0\subset E'_0\supset V'_1\subset\cdots\supset V'_r\subset E'_r\supset 
V'_{r+1})$ be a minimal one-skeleton gallery in the cell $C_\delta$. Since $\gamma$ starts at $V_0=\Lo$, 
we may replace $\gamma$ by $u\gamma$ for some $u\in U^-(\co)$ if necessary and
assume that $E'_0 = E_0$ and $V'_1 = V_1$ are in $\A$. Let $\underline\sfrak'(\gamma) = 
(\sfrak_0,\sfrak'_1,...,\sfrak'_r)$ be the sequence of representatives of the same equivalence class of sectors
such that $V'_i$ is the vertex of $\sfrak'_i$ and $E'_i\subset \sfrak'_i$. The sequence
starts with a sector tipped at $0$ whose image by $r_{-\infty}$ is the chamber 
$\tau_0(C^+)$, for some $\tau_0\in W$. We know that $(E_0\supset V_1\subset E'_1)$ is minimal and such that 
$r_{-\infty}(E'_1) = E_1$, hence Lemma \ref{lemminimalpositiv2} shows that $(E_0\supset V_1\subset E_1)$ is 
positively folded. This means that there exists a face $V_1\subset E''_1\subset\A$ such that $(E_0\supset V_1\subset E''_1)$ is 
minimal and $(E_0\supset V_1\subset E_1)$ is obtained from $(E_0\supset V_1\subset E''_1)$ by a positive folding,
see Lemma~\ref{lemminimalpositiv} and Lemma~\ref{lemminimalpositiv2} and the proofs. 
Now in the proof one may choose for the minimal gallery of residue chambers as last chamber the residue chamber
associated to $\sfrak_0(V_1)=\sfrak_1$.
But $E''_1$ is contained in $\tau_0(C^+)(V_1)$ (see Lemma~\ref{lemminimalpositiv} and its proof) 
and the sector $\sfrak_1$ retracts onto a sector of $\A$ tipped at $V_1$ 
and containing $E_1$. Therefore $r_{-\infty}(\sfrak_1) = \tau_1(C^+)(V_1)$, with $\tau_0 \geq \tau_1$.

We want to repeat this argument to prove the claim in an inductive procedure. To do so,
recall from Proposition \ref{pr:Cell}, that $\gamma$ corresponds to a sequence 
$$
(v_0,v_1,...,v_r)\in \Stab_-(V_0,E_0)\times \Stab_-(V_1,E_1) \times\cdots\times\Stab_-(V_r,E_r)
$$ and that $E'_j = v_0v_1\cdots v_j \delta_0\delta_1\cdots\delta_j E_j^f = v_0v_1\cdots v_j E_j$. The gallery 
$\gamma$ can be retracted step by step, that means that we consider the sequence:
$$
\begin{array}{rcl}
\gamma & = & (V_0\subset E'_0\supset V'_1\subset\cdots\supset V'_r\subset E'_r\supset V'_{r+1}),\\
\gamma^0 & = & (V_0\subset E_0\supset V_1\subset v_1E_1\cdots \subset (v_1\cdots v_r) E_r\supset (v_1\cdots v_r) V_{r+1}),\\
 &\vdots & \\
\gamma^{j-1} & = & (V_0\subset E_0\supset V_1\subset\cdots \subset E_{j-1}\supset V_j\subset v_jE_j\supset\cdots \\
&&\hskip 110pt\cdots \subset (v_j\cdots v_r) E_r\supset (v_j\cdots v_r) V_{r+1})\\
 &\vdots & \\
\gamma^r = \delta & = & (V_0\subset E_0\supset V_1\subset\cdots\supset V_r\subset E_r\supset V_{r+1}).
\end{array}
$$ Now, at each step, $(E_{j-1}\supset V_j\subset v_jE_j)$ is minimal because it is obtained from a minimal 
two-step gallery by applying elements of $G(\ck)$. So, we can repeat the previous arguments to show that 
$\delta$ is globally positively folded.

Reciprocally, we show that if $\delta$ is 
positively folded then one can, inductively, built a minimal gallery that retracts onto it. 
Indeed, we start applying Theorem~\ref{thmLocMin} at the vertex $V_r$. So we get a gallery 
$$
\delta^r = (V_0\subset E_0\supset V_1\subset\cdots \subset E_{r-1}\supset V_r\subset 
E'_r\supset V'_{r+1}),
$$ where $(E_{r-1}\supset V_r\subset E'_r)$ is minimal in an apartment $A_r$, 
$r_{-\infty}(E'_r) = E_r$ and a sequence of sectors $(\sfrak_0,\sfrak_1,...,\sfrak_r)$ 
such that $\underline\sfrak_0\geq\sfrak_1\geq\cdots\geq\underline\sfrak_{r-1} = \underline\sfrak_r$. 
We apply the theorem again at the vertex $V_{r-1}$. So we get a gallery 
$$
\delta^{r-1} = (V_0\subset E_0\supset V_1\subset\cdots \subset E_{r-2}\supset V_{r-1}\subset 
E'_{r-1}\supset V'_{r}),
$$ where $(E_{r-2}\supset V_{r-1}\subset E'_{r-1})$ is minimal in an 
apartment $A_{r-1}$, $r_{-\infty}(E'_{r-1}) = E_{r-1}$ and a sequence of sectors $(\sfrak_0,\sfrak_1,...,\sfrak_{r-1})$ 
such that $\underline\sfrak_0\geq\sfrak_1\geq\cdots\geq\underline\sfrak_{r-2} = 
\underline\sfrak_{r-1}$. Now, since $(E_{r-1}\supset V_r\subset 
E'_r)$ is positively folded in $\A$, there exists a face $F'_r$ of the same type as 
$E_r$ such that $(E_{r-1}\supset V_r\subset F'_r)$ is minimal in $\A$. Since 
$E'_{r-1} = u_{r-1} E_{r-1}$, we can take $A_{r-1} = u_{r-1} \A$ and the image of 
$(E_{r-1}\supset V_r\subset F'_r)$ in $A_{r-1}$ is still minimal. So we complete the 
gallery $\delta^{r-1}$ with it to get a one-skeleton gallery which is minimal after the 
index $r-1$ and contained in the sector $\sfrak_{r-2} = \sfrak_{r-1} = \sfrak_r$ of $A_{r-1}$. 
Iterating this procedure, we get a minimal one-skeleton gallery that retracts onto $\delta$.
\qed

\subsection{A formula for $L_{\lam,\mu}$}\label{formulaL}
For a dominant coweight $\lam$ let $\gamma_\lam$ be a
dominant combinatorial gallery joining $\Lo$ and $\lam$ (see 
Example~\ref{exGalCoweightarbicompletelyfree}). 
The investigation of the intersection 
$Z_{\lam,\mu}$ can be transferred to the Bott-Samelson variety 
$\Sigma(\gamma_\lambda)$:
$$
Z_{\lam,\mu}= G(\co).\lam\cap U^-(\ck).\mu=
\bigcup_{\substack{\delta\in \Gamma(t_{\gamma_\lam}, \Lo)\\ {\rm target}(\delta)=\mu}}
\{\text{minimal galleries}\}\cap C_\delta.
$$
Proposition~\ref{coro:non-empty} states that the intersection $\{\text{minimal galleries}\}\cap C_\delta$
is non-empty if and only if $\delta$ is positively folded. We want to describe the intersection more
precisely.

Recall from Definition \ref{dfn:Min} that for a two steps gallery $(E\supset V\subset F)$ the set $Min(E,F)$ identifies with the 
set of all faces $E'_V$ such that $(E_V,E'_V)$ is a minimal pair and $r_{C_V^-}(E'_V) = F_V$. 

Let
$\delta = (\Lo=V_0\subset E_0\supset \cdots\supset V_r\subset E_r\supset \mu)$ be a 
positively folded combinatorial gallery and
let $B^-\subset G$ be the opposite Borel subgroup. Denote by $D_0$ the chamber in $\A$
which contains $E_0$ and is the closest to $C^-$, and let $w_{D_0}\in W$ be the element such that
$w_{D_0}(C^-)=D_0$.
\begin{proposition}
\label{prMinGlobal}
The set of all minimal one-skeleton galleries in the cell $C_\delta$ identifies with the product
$$
B^- w_{D_0} Q^-_{E_0}/Q^-_{E_0}\times \prod_{j=1}^{r} Min(E_{j-1},E_j)\ .
$$ 
\end{proposition}

\noindent{\bf Proof. } If $\gamma = (V_0\subset E'_0\supset V'_1\subset\cdots\supset V'_r\subset E'_r\supset V'_{r+1})$ 
is a minimal gallery of $C_\delta$, then $E'_0$ identifies with an element of the orbit $B^- w_{D_0} Q^-_{E_0}/Q^-_{E_0}$. 
Further, to $\gamma\in C_\delta$ corresponds a sequence $(v_0,v_1,...,v_r)$ in $$
\Stab_-(\delta) 
= \Stab_-(V_0,E_0)\times \Stab_-(V_1,E_1) \times\cdots\times\Stab_-(V_r,E_r).
$$ In the proof of 
Proposition \ref{coro:non-empty}, we have seen that $h(\gamma) := (E'_0, v_1E_1,..., v_rE_r)$ 
belongs to $B^- w_{D_0} Q^-_{E_0}/Q^-_{E_0}\times \prod_{j=1}^{r} Min(E_{j-1},E_j)$, and 
we have also seen that $h$ is surjective. The fact that $h$ is injective is a consequence of 
Proposition \ref{pr:DescStab-}.
\qed 

\medskip
Let $\mathbb F_q$ be the finite field with $q$ elements and replace 
the field of complex numbers by the algebraic closure $K$ of $\mathbb F_q$. 
Assume that all groups are defined and split over $\mathbb F_q$. We replace
now $\ck$ by $\ck_q = \mathbb F_q(\!(t)\!)$, the field of Laurent series, and 
$\co$ by $\co_q = \mathbb F_q[[t]]$.  For a given
positively folded gallery $\delta = [\delta_0,\delta_1,...,\delta_r] = 
(V_0\subset E_0\supset V_1\subset\cdots\subset E_r\supset V_{r+1})$ 
we want to count the number of points (over $\mathbb F_q$) of the intersection 
$$
\{\text{minimal galleries}\}\cap C_\delta\ .
$$ 
For convenience, we first fix (and recall) some notation: $\forall j = 0, 1,..., r$, let
\begin{itemize}
\item $D_j $ be the closest chamber to $C_{V_j}^-$ containing $(E_j)_{V_j}$;
\item $\sfrak^j\supset E_{j-1}$ be a sector with vertex $V_j$ such that there exists a face $F'_j\subset -\sfrak^j$ containing $V_j$ of the same type as $E_j$;
\item $\mathbf i_j = ((i_j)_1,...,(i_j)_{r_j})$ be a reduced decomposition of $w(C_{V_j}^-, D_j)$.
\end{itemize}
We denote by $\Gamma_{\sfrak^j_{V_j}}^+(\mathbf i_j, op)$ the set of all galleries 
$\mathbf c = (C_{V_j}^-,C_1,...,C_{r_j})$ of residue chambers of type $\mathbf i_j$ 
which are positively folded with respect to $\sfrak^j_{V_j}$ and have the property that the 
face $(E_j')_{V_j}$ of the same type as $(E_j)_{V_j}$ contained in $C_{r_j}$ forms a minimal pair with 
$(E_{j-1})_{V_j}$ in $\A_{V_j}$. 

The exponents in the formula are, first, the length $\ell (w_{D_0})$ and, second, for each $\mathbf c = (c_1,...,c_{r_j})\in \Gamma_{\sfrak^j_{V_j}}^+(\mathbf i_j, op)$,
the nonnegative integers $t(\mathbf c)$ and $r(\mathbf c)$ defined in Lemma~\ref{reMinCell}:
$t(\mathbf c) = \sharp\{k\mid c_k = s_{(i_j)_k} \hbox{ and } w_{\sfrak^j_{V_j}}^{-1} \be_k < 0\}$ and 
$r(\mathbf c) = \sharp\{j\mid c_k = 1 \hbox{ and } w_{\sfrak^j_{V_j}}^{-1} \be_k < 0\}.$ Combining Lemma \ref{reMinCell}, Theorem \ref{thmLocMin} and 
Propositions~\ref{prMinLocal} -- \ref{prMinGlobal}, 
we obtain the following formula:
\begin{theorem}\label{Lpolynomialformula}
$$
L_{\lambda,\mu} (q) = \sum_{\delta\in\Gamma^+(\gamma_\lambda, \mu)} q^{\ell(w_{D_0})} 
\bigg(\prod_{j=1}^r \  \sum_{\mathbf c\in \Gamma_{\sfrak^j_{V_j}}^+(\mathbf i_j, op)} q^{t(\mathbf c)} (q-1)^{r(\mathbf c)}\bigg)\ .
$$ 
\end{theorem}
\begin{rem}\rm
According to a result of Katz (Theorem 6.1.12) in \cite{Ka}), the value $L_{\lambda,\mu}(1)$ gives the Euler-Poincar\' e 
characteristic of the variety $G(\co).\lam\cap U^-(\ck).\mu$. Now a summand above is nonzero if and only if
the gallery is minimal. It is easy to see that a gallery is minimal if and only if the gallery has as a target
an extremal weight. Thus, we recover a result of Ng\^o and Polo \cite{NP}, saying 
that this characteristic is $1$ if $\mu$ is in the orbit $W\lambda$, and it is $0$ otherwise.

\end{rem}
\begin{exam}\rm
Let us consider an example in type $A_2$. Let $\lambda = 2 \omega_1 + \omega_2$ where $\omega_i$ are the 
fundamental coweights. There are three possibilities for $\mu \leq \lambda$ in the fundamental Weyl chamber: 
$\mu = \lambda$, $\mu = 2\omega_2$ and $\mu = \omega_1$. If $\mu= \lam$, then one finds 
$L_{\lambda,\lam} (q) = q^2q^2q^2$. In the second case, $L_{\lambda,2\omega_2} (q) = q (q-1)qq^2$. 
Finally, if $\mu = \omega_1$, there are two one-skeleton galleries starting in $\Lo$ and ending in $\omega_1$. 
Let us explain the computation in the case of the gallery 
$(\Lo\subset E_1\supset V_1\subset E_2\supset V_2\subset E_3\supset \omega_1)$ plotted in the picture below.

At the vertex $\Lo$, $\ell(w_{D_0}) = 1$, therefore we get a $q$. At the vertex $V_1$, 
there is only one gallery $\mathbf c$ of residue chambers positively folded with respect 
to $\sfrak^1$, starting in $C^-_{V_1}$ and ending in a chamber containing an opposite to 
$E_1$. This gallery $\mathbf c$ has a positively (with respect to $\sfrak^1$) folding on one 
wall and crosses positively another, so we have $t(\mathbf c) = 1$ and $r(\mathbf c) = 1$, 
whence we get $(q-1)q$. At the vertex $V_2$, the gallery $(E_2\supset V_2\subset E_3)$ is 
minimal. The gallery of residue chambers has only two terms and positively (with respect to 
another sector) crosses the vertical wall, therefore, we get $q$. One computes in an analogous 
way the number of minimal one-skeleton galleries retracting on the second gallery ending in 
$\omega_1$ and one gets $q(q-1) q^2$. 
Finally, $L_{\lambda,\omega_1} (q) = q(q-1)qq + q(q-1) q^2 = 2(q-1)q^3$.

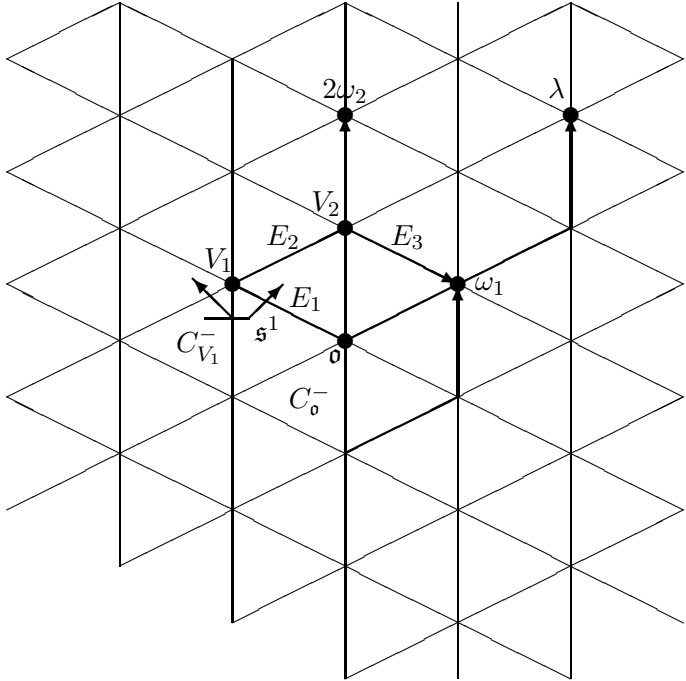
\begin{center}
 \setlength{\unitlength}{1.5cm}
\begin{picture}(8,7)
\thinlines
\put(1,2){\line(0,1){5}}
\put(2,1.5){\line(0,1){5}}
\put(3,1){\line(0,1){6}}
\put(4,1){\line(0,1){6}}
\put(5,1){\line(0,1){6}}

\put(1,2){\line(2,1){5}}
\put(0,2.5){\line(2,1){6}}
\put(0,3.5){\line(2,1){6}}
\put(0,4.5){\line(2,1){5}}
\put(0,5.5){\line(2,1){3}}
\put(0,6.5){\line(2,1){1}}
\put(2,1.5){\line(2,1){4}}
\put(3,1){\line(2,1){3}}

\put(6,1.5){\line(-2,1){6}}
\put(6,2.5){\line(-2,1){6}}
\put(6,3.5){\line(-2,1){6}}
\put(6,4.5){\line(-2,1){5}}
\put(6,5.5){\line(-2,1){3}}
\put(6,6.5){\line(-2,1){1}}
\put(5,1){\line(-2,1){5}}

\put(3,4){\circle*{0.15}}
\put(3,6){\circle*{0.15}}
\put(4,4.5){\circle*{0.15}}
\put(5,6){\circle*{0.15}}
\put(2,4.5){\circle*{0.15}}
\put(3,5){\circle*{0.15}}

\put(4.8,6.15){$\lam$}
\put(2.8,6.15){$2\omega_2$}
\put(4.15,4.45){$\omega_1$}
\put(2.85,3.8){$\Lo$}
\put(2.5,4.3){$E_1$}
\put(1.75,4.65){$V_1$}
\put(2.3,4.85){$E_2$}
\put(2.7,5.15){$V_2$}
\put(3.4,4.85){$E_3$}

\put(2.5,3.4){$C^-_{\Lo}$}
\put(1.5,3.9){$C^-_{V_1}$}
\put(2.2,4){$\sfrak^1$}

\thicklines
\put(3,4){\line(2,1){2}}
\put(5,5){\vector(0,1){1}}
\put(3,4){\line(-2,1){1}}
\put(2,4.5){\line(2,1){1}}
\put(3,5){\vector(0,1){1}}
\put(3,5){\vector(2,-1){1}}
\put(3,4){\line(0,-1){1}}
\put(3,3){\line(2,1){1}}
\put(4,3.5){\vector(0,1){1}}

\put(1.75,4.2){\line(1,0){.4}}
\put(2.15,4.2){\vector(1,1){.3}}
\put(2,4.2){\vector(-1,1){.35}}
\end{picture}
\end{center}
\vskip -30pt
\end{exam}
\section{Dimension of $r_{-\infty}^{min}(\delta)$, LS-galleries and Young tableaux}\label{LSAndYoungTableaux}

We want now to discuss some examples and the connection with the work
of Lakshmibai, Musili and Seshadri. 
Recall that the theory of a path model for a representation is a generalization of the original
idea of Lakshmibai, Musili and Seshadri (see for example \cite{LMS}, \cite{LS1}, \cite{LS2})
to index a basis of fundamental representation by sequences of Weyl group
elements satisfying certain combinatorial conditions. Monomials of these basis elements then form
a generating system for the other irreducible representations (respectively the corresponding dual 
Weyl modules in positive characteristic), 
and the aim was to show that 
special monomials, the {\it standard monomials = } monomials having a defining chain, form 
in fact a basis. This program was successfully realized in many cases, for example for all representations
of the classical groups but also in many other cases (see ibidem). The path model theory provided
a new approach and made it possible to prove the conjecture for the character in full generality
for Kac-Moody algebras \cite{L1}, the construction of an associated standard monomial theory is  
discussed in \cite{L3}.

\subsection{LS one-skeleton galleries}

Given a dominant coweight $\lam$, let $\gamma_\lam$
be a combinatorial one-skeleton gallery as in Example~\ref{exGalCoweightarbicompletelyfree}.
Let $\delta=(V_0=\Lo\subset E_0\supset\ldots \supset V_{r+1})$ 
be a positively folded combinatorial one-skeleton gallery of the same type as $\gamma_\lam$.
By Proposition~\ref{coro:non-empty} we know that the intersection of the set of minimal galleries
$G(\co).\gamma_\lam$ with the cell $C_\delta$ is a dense subset of $C_\delta$, so 
$$
\dim\,r_{-\infty}^{min}(\delta)=\dim\,(\{\text{minimal galleries}\}\cap C_{\delta})=\dim (C_{\delta}).
$$
The dimension of the cell can be computed by Proposition~\ref{pr:Cell} using combinatorial
properties of the gallery: given an affine root $(\alpha,n)$, $\alpha>0$, a vertex 
$V\in H_{\alpha,n}$ and an edge $E$ in $\A$, then we say that $(V,E)$ {\it crosses
the wall} (or hyperplane) $H_{\alpha,n}$ in the {\it positive (negative) direction} if 
$F\not\subset H_{\al,n}^-$ (respectively $F\not\subset H_{\al,n}^+$).
\begin{rem}\label{crossingroot}\rm
Using the terminology of section~\ref{suseCells}, an equivalent formulation is to say that
a {\it wall crossing is positive} if $(-\al,-n)\in \Phi_-^{\mathfrak a}(V,F)$. 
\end{rem}
For the gallery $\delta$ denote by $\sharp^+\delta$ the number of positive wall crossings, 
by $\sharp^-\delta$ the number of negative wall crossings and by $\sharp^\pm\delta$
the number of all wall crossings:
$$
\begin{array}{rcl}
\sharp^+\delta&=&\sum_{i=0}^r (\sharp\,\text{positive wall crossings of}\,(V_i,E_i))\\
\sharp^-\delta&=&\sum_{i=0}^r (\sharp\,\text{negative wall crossings of}\,(V_i,E_i))\\
\sharp^\pm\delta&=& \sharp^+\delta+ \sharp^-\delta.
\end{array}
$$
For the last number we have $\sharp^\pm\delta=\sharp^+\gamma_{\lam}=\langle\lam,2\rho\rangle$
because it depends only on the type of the gallery.
Together with Remark~\ref{crossingroot} and Proposition~\ref{pr:Cell} we get:
\begin{lemma}
$\sharp^+\delta=\dim (C_\delta)$.
\end{lemma}
An upper bound for $\sharp^+\delta$ can be determined using the target of the gallery:
\begin{proposition}\label{dimensioninequality}
Let $\mu$ be the target of $\delta$, then $\sharp^+\delta\le \langle \lam+\mu,\rho\rangle$.
\end{proposition}
\proof
Since $P_\lam\rightarrow s_\lam$ we know that $q^{-\langle\rho,\lam+\mu\rangle} L_{\lam,\mu}\in\bz[q^{-1}]$,
so the power of the leading term in $L_{\lam,\mu}$ is less or equal to $\langle\rho,\lam+\mu\rangle$.
By the formula in Theorem~\ref{Lpolynomialformula}, the maximal power of the contribution coming from
a positively folded gallery $\delta$ occurs with coefficient $+1$. The maximal power of the term coming from $\delta$
is $\dim (C_\delta)$, which proves the claim.
\qed
\medskip
\begin{dfn}\rm
We call a positively folded combinatorial one-skeleton gallery $\delta$ of the same type as $\gamma_\lam$ a {\it LS-gallery} 
if $\sharp^+\delta=\langle\lam+\mu,\rho\rangle$, where $\mu$ is the target of $\delta$.
\end{dfn}

\begin{rem}
\label{reMinLS}
All minimal combinatorial one-skeleton galleries are LS-gal\-leries. Indeed, if $\delta$ is a minimal gallery with target $\mu$, then 
$\langle\mu,2\rho\rangle=\sharp^+\delta-\sharp^-\delta$, so
$$
\sharp^+\delta=\frac{1}{2}(\sharp^+\delta-\sharp^-\delta+\sharp^\pm\delta)=\frac{1}{2}\langle\lam+\mu,2\rho\rangle
=\langle\lam+\mu,\rho\rangle.
$$
\end{rem}

\subsection{Reduction to the case of a fundamental weight}

To describe the connection of the path model with LS-galleries in the one-skeleton,
one of the first steps is the reduction to the case of
a fundamental weight. 
\begin{lemma}\label{minusculegallery}
\begin{enumerate}
  \item If $\omega$ is a minuscule coweight, then all combinatorial galleries of the same type as $\gamma_\om$
are LS-galleries.
  \item Suppose $\delta_1,\ldots,\delta_r$ are positively folded combinatorial galleries of the same type
as $\gamma_{\lam_1},\ldots,\gamma_{\lam_r}$ respectively. Suppose the concatenation
$\delta=\delta_1*\ldots*\delta_r$ is positively folded. Then $\delta$ is an LS-gallery if and only
if each of the $\delta_j$, $j=1,\ldots,r$, is a LS-gallery.
\end{enumerate}
\end{lemma}
\proof
If $\omega$ is a minuscule coweight, then all combinatorial galleries of the same type as $\gamma_\om$ have no
folds and hence are minimal, which proves the claim by Remark \ref{reMinLS}.

Let $\delta=\delta_1*\ldots*\delta_r$ be a concatenation of positively folded galleries as in $(3)$.
If $\delta$ has target $\mu$ and $\delta_i$ has target $\mu_i$, then 
$\sharp^+\delta=\sum_{j=1}^r\sharp^+\delta_j$
and  $\langle\lam+\mu,\rho\rangle=\sum_{j=1}^r \langle\lam_{j}+\mu_j,\rho\rangle$.
So by Proposition~\ref{dimensioninequality} we have equality $\sharp^+\delta=\langle\lam+\mu,\rho\rangle$ 
if and only if $\sharp^+\delta_j=\langle\lam_{j}+\mu_j,\rho\rangle$
for all $j=1,\ldots,r$.
\qed

\medskip
In the following let $\gamma_\lam$ be as in Example~\ref{exGalCoweightarbifreesequence},
we want to characterize the LS-galleries of the same type as $\gamma_\lam$.
The Lemma above reduces the consideration to the case where $\lam=\omega$ is a fundamental weight.

Let $\delta_0=(\Lo=V_0\subset E_0\supset\ldots \subset V_j\subset\ldots\supset V_r=\mu_0)$ be a positively 
folded gallery of the same type as $\gamma_\om$, and let $j$ be such that $\delta_0$ has no folds
at the vertices $V_i$ for $i\geqslant j$ (note: we do not ask $j$ to be minimal with this property).
Let $\beta$ be a positive root and suppose there exists an $m\in\bz$ such that $V_j\in H_{\beta,m}$.
Denote by $\delta$ the gallery
\begin{equation}
\label{fold}
\delta=(\Lo=V_0\subset E_0\supset\ldots \subset V_j\subset s_{\beta,m}(E_j) \supset\ldots\supset s_{\beta,m}(V_r)=\mu)
\end{equation}
Given the one-dimensional face $E_j$
let $\nu_{E_j}$ be the rational weight $V_{j+1}-V_j$. Since $E_j$ is of type $\omega$, there
exists a unique element $\tau_{E_j}\in W/W_\om$ such that the two rays $\br\tau_{E_j}(\om)$ and $\br\nu_{E_j}$
coincide.

\begin{dfn}\rm 
We say that $\delta$ is obtained from $\delta_0$ {\it by a positive fold} if $s_\beta\tau_{E_j}<\tau_{E_j}$
in the Bruhat order on $W/W_\om$. We say that $\delta$ is obtained from $\delta_0$ by an {\it LS--fold} if 
in addition $\ell(s_\beta\tau_{E_j})=\ell(\tau_{E_j})-1$ for the length function $\ell$ on $W/W_\om$
\end{dfn}

By definition, if $\delta$ is obtained from $\delta_0$ by a positive fold, then $\delta$ is also positively folded.
Obviously every positively folded gallery can be obtained from a minimal gallery by a sequence 
of such positive folds. 

To be able to characterize the LS-galleries of the same type as $\gamma_\om$, we divide
this folding algorithm into the smallest possible steps.  Since we can only fold with respect
to the roots in the local root system $\Phi_{V_j}$, we consider first the Weyl group
$W_{V_j}$ of $\Phi_{V_j}$. There exists a unique ray $\br\nu_0$ contained in the dominant Weyl
chamber with respect to $\Phi_{V_j}$ and a unique element $t\in W_{V_j}/(W_{V_j})_{\nu_0}$
such that $t(\nu_0)=\nu_{E_j}$. 

\begin{dfn}\rm
We say that the fold by $s_{\beta,m}$ is {\it minimal for the local root system} $\Phi_{V_j}$
if $\ell(s_{\beta}t)=\ell(t)-1$ for the length function $\ell$ on $W_{V_j}/(W_{V_j})_{\nu_0}$.
\end{dfn}
If the fold is not minimal, then one can find positive roots $\beta_1,\ldots,\beta_q$ in $\Phi_{V_j}$
such that $t>s_{\beta_1}t>\ldots> s_{\beta_q}\cdots s_{\beta_1} t=s_{\beta}t$ in the Bruhat
ordering on $W_{V_j}/(W_{V_j})_{\nu_0}$, and in each step the length decreases by one.
For each root $\beta_i$ let $m_i$ be such that $V_j\in H_{\beta_i,m_i}$, then the sequence
of folds by the affine reflections $s_{\beta_1,m_1},\ldots,s_{\beta_q,m_q}$ are all positive
and, by the choice, minimal. Summarizing we have:
\begin{lemma}
A positively folded gallery of the same type as $\gamma_\om$ is obtained
from a minimal gallery by a sequence of positive folds such that each fold is minimal
 for the local root system associated to the corresponding vertex.
\end{lemma}
We want to compare $\sharp^+\delta$ and $\sharp^+\delta_0$, where $\delta$ is obtained
from $\delta_0$ by a fold as in (\ref{fold}), but now assume that the positive fold is minimal.
\begin{proposition} $\sharp^+\delta \le \sharp^+\delta_0+\langle \mu-\mu_0,\rho\rangle$.
Further, $\delta$ is an LS-gallery if and only if $\delta_0$ is an LS-gallery and
the new fold is an LS-fold.
\end{proposition}
Since the condition of being folded by a sequence of LS-folds is equivalent to the condition for LS-paths,
we get as an immediate consequence:
\begin{corollary}\label{pathgalleryfundamental}
For a fundamental coweight $\om$ let $\pi_\om:[0,1]\rightarrow X^\vee_\br$ be the path $t\mapsto t\om$
and let $\pi$ be an LS-path of shape $\om$ as in \cite{L1}. As associated gallery $\gamma_\pi$ in the one-skeleton of $\A$
take the sequence of edges and vertices lying on the path. This map $\pi\mapsto \gamma_\pi$ describes a bijection 
between the LS-paths of shape $\om$ and the LS-galleries of the same type as $\gamma_\om$.
\end{corollary}
\vskip 1pt
\noindent
{\it Proof of the proposition.} For  $\delta_0=(\Lo=V_0\subset E_0\supset\ldots\supset V_r=\mu_0)$ 
let $j$ be such that $\delta=(\Lo=V_0\subset E_0\supset\ldots \subset V_j\subset s_{\beta,m}(E_j) 
\supset\ldots\supset s_{\beta,m}(V_r)=\mu)$. Denote by $\sharp_j^+\delta_0$ the number of positive 
crossings associated to the vertices $V_k$ for $k\geqslant j$.
Since the two galleries coincide till $V_j$, we have $\sharp^+\delta-\sharp^+\delta_0=\sharp_j^+\delta-\sharp_j^+\delta_0$.

Let $\nu_0$ be the rational weight $\mu_0-V_j$ and set $\nu= \mu-V_j$. There exists a rational number
$0<r\le 1$ and elements $\kappa,\tau\in W/W_\om$ such that $\nu_0=r\tau(\om)$, $\nu=r\kappa(\om)$, 
$r\langle\kappa(\om),\beta\rangle\in \bz$, and $s_\beta\tau=\kappa$. 
Note that
$$
\begin{array}{rcl}
\langle \om + \mu,\rho\rangle -\langle \om + \mu_0,\rho\rangle&=&\langle \mu-\mu_0,\rho\rangle    \\
      &=&\langle r(\kappa(\om)-\tau(\om)),\rho\rangle \\
      &=&\frac{r}{2}(\sum_{\gamma>0}\langle\kappa(\om),\gamma\rangle-\sum_{\gamma>0}\langle\tau(\om),\gamma\rangle).
\end{array}
$$
We need the following simple lemma, which we state without proof.
\begin{lemma}\label{simplelemma}
Let $\psi$ be a root system with Weyl group $W(\Psi)$
and let $\nu$ be a dominant weight. Fix $s_\beta\tau\in W(\Psi)/W(\Psi)_\nu$
and let $\beta$ be a positive root such that $s_\beta\tau<\tau$. 
We divide the set of positive roots into $\Psi^+=A\cup B$, where
$A=\{\gamma>0\mid s_\beta(\gamma)>0\}$ and $B=\{\gamma>0\mid s_\beta(\gamma)<0\}$.
Consider the following sets:
$$
A_{\tau}^+=\{\gamma\in A\mid \langle \tau(\nu), \gamma\rangle\ge 0\}\quad
A_{\tau}^0=\{\gamma\in A\mid ,\langle \tau(\nu), \gamma\rangle =0\}
$$
$$
B_{\tau}^+=\{\gamma\in B\mid\langle \tau(\nu), \gamma\rangle\ge 0\} \quad
B_{\tau}^0=\{\gamma\in B\mid \langle \tau(\nu), \gamma\rangle =0\},
$$
and similarly we define the sets $A_{\tau}^-$ and $B_{\tau}^-$.

Then $s_\beta(A_{\tau}^\pm)=A_{s_\beta\tau}^\pm$, $s_\beta(A_{\tau}^0)=A_{s_\beta\tau}^0$,
$- s_\beta(B_{\tau}^+)=B_{s_\beta\tau}^-$, $- s_\beta(B_{\tau}^-)=B_{s_\beta\tau}^+$ and
 $- s_\beta(B_{\tau}^0)=B_{s_\beta\tau}^0$.
Further, $B_{\tau}^+\cup\{\beta\}\subset B_{s_\beta \tau}^+$, and one has equality if and only if
$\ell(\tau)=\ell(s_\beta \tau)+1$ for the length function $\ell$ on $W(\Psi)/W(\Psi)_\nu$.
\end{lemma}
\vskip 5pt\noindent
Using the notation and the results of Lemma~\ref{simplelemma}, this sums reduces to 
$$
\langle \om + \mu,\rho\rangle -\langle \om + \mu_0,\rho\rangle=
r(\sum_{\gamma\in B_\kappa^+}\langle\kappa(\om),\gamma\rangle-\sum_{\gamma\in B_\tau^+}\langle\tau(\om),\gamma\rangle),
$$
since $\langle \kappa(\om),\gamma\rangle=\langle s_{\beta}\kappa(\om),s_{\beta}(\gamma)\rangle=\langle \tau(\om),s_{\beta}(\gamma)\rangle$.
Again by Lemma~\ref{simplelemma}, we can divide $B_\kappa^+$ into $B_\tau^+\cup\{\beta\}\cup \hbox{\it Rest}$ and get:
$$
\begin{array}{rcl}
\langle \om + \mu,\rho\rangle -\langle \om + \mu_0,\rho\rangle &=&
r(\sum_{\gamma\in B_\tau^+}\langle\kappa(\om)-\tau(\om),\gamma\rangle
+\langle\kappa(\om),\beta\rangle \\
&& \hskip 20pt \hfill +\sum_{\gamma\in \hbox{\it Rest}} \langle\kappa(\om),\gamma\rangle)\\
&=&\sum_{\gamma\in B_\tau^+}r\langle\kappa(\om),\beta\rangle\langle\beta,\gamma\rangle +
r\langle\kappa(\om),\beta\rangle\\
&& \hskip 20pt \hfill +\sum_{\gamma\in \hbox{\it Rest}} r\langle\kappa(\om),\gamma\rangle.
\end{array}
$$
We want to compare this sum to $\sharp_j^+\delta-\sharp_j^+\delta_0$.
If $\gamma$ is a positive root, then $(V_k,E_k)$ crosses some wall 
$H_{\gamma,p}$ positively for some $k\geqslant j$ only if $\langle\tau(\om),\gamma\rangle>0$, and if
$\gamma\in\Phi_{V_j}$, then the number of such crossings is $r\langle\tau(\om),\gamma\rangle$.
If $\gamma\not\in\Phi_{V_j}$, then the number of such crossings is $\lfloor r\langle\tau(\om),\gamma\rangle\rfloor$,
the largest integer smaller or equal to $r\langle\tau(\om),\gamma\rangle$. So again with the notation as in Lemma~\ref{simplelemma}
and the decomposition $B_\kappa^+=B_\tau^+\cup\{\beta\}\cup \hbox{\it Rest}$:
$$
\begin{array}{rcl}
\sharp_j^+\delta-\sharp_j^+\delta_0&=&\sum_{\gamma\in A_\kappa^+\cup B_\kappa^+} \lfloor r\langle\kappa(\om),\gamma\rangle\rfloor
- \sum_{\gamma\in A_\tau^+\cup B_ \tau ^+} \lfloor r\langle \tau(\om),\gamma\rangle\rfloor\\
&=&\sum_{\gamma\in B_\kappa^+} \lfloor r\langle\kappa(\om),\gamma\rangle\rfloor
- \sum_{\gamma\in B_ \tau ^+} \lfloor r\langle \tau(\om),\gamma\rangle\rfloor\\
&=&\sum_{\gamma\in B_\tau^+} (\lfloor r\langle\kappa(\om),\gamma\rangle\rfloor -\lfloor r\langle \tau(\om),\gamma\rangle\rfloor)
+ \lfloor r\langle \kappa(\om),\beta\rangle\rfloor\\
&&\hskip 20pt \hfill +\sum_{\gamma\in \hbox{\it Rest}} \lfloor r\langle \kappa(\om),\gamma\rangle\rfloor\\
&=&\sum_{\gamma\in B_\tau^+} (\lfloor r\langle\kappa(\om),\gamma\rangle\rfloor-
\lfloor r\langle \kappa(\om),\gamma\rangle-r\langle\kappa(\om),\beta\rangle\langle\beta,\gamma\rangle \rfloor)\\
&&\hskip 20pt \hfill + \lfloor r\langle \kappa(\om),\beta\rangle\rfloor+
\sum_{\gamma\in \hbox{\it Rest}} \lfloor r\langle \kappa(\om),\gamma\rangle\rfloor\\
\end{array}
$$
Since $r\langle \kappa(\om),\gamma\rangle$ is an integer by assumption, we obtain:
$$
\sharp_j^+\delta-\sharp_j^+\delta_0=\sum_{\gamma\in B_\tau^+} r\langle\kappa(\om),\beta\rangle\langle\beta,\gamma\rangle
+ r\langle \kappa(\om),\beta\rangle+\sum_{\gamma\in \hbox{\it Rest}} \lfloor r\langle \kappa(\om),\gamma\rangle\rfloor
$$
As a consequence we see:
$$
\begin{array}{rcl}
(\langle \mu-\mu_0,\rho\rangle)-(\sharp^+\delta-\sharp^+\delta_0)&=&(\langle \om+\mu,\rho\rangle-\langle \om+\mu_0,\rho\rangle)
-(\sharp^+_j\delta-\sharp_j^+\delta_0)\\
&=&\sum_{\gamma\in \hbox{\it Rest}} (r\langle \kappa(\om),\gamma\rangle-\lfloor r\langle \kappa(\om),\gamma\rangle\rfloor),
\end{array}
$$
which proves the inequality in the proposition. We have equality if and only if the right hand term above is zero.
The target $\mu$ is a special point, so $r\langle \kappa(\om),\gamma\rangle$ is an integer if an only if
$\gamma\in \Phi_{V_j}$. Since the folding is minimal by assumption, the intersection $\hbox{\it Rest}\cap\Phi_{V_j}=\emptyset$.
But this implies that we have equality if and only if $\hbox{\it Rest}=\emptyset$, i.e., the fold is an LS-fold by 
Lemma~\ref{simplelemma}.
In particular, $\delta$ is an LS-gallery if and only if $\delta_0$ is an LS-gallery and the new fold is an LS-fold.
\qed
 
\subsection{Connection with the path model}
 
Summarizing the results above, we have the following connection between the path model of a representation 
and the one-skeleton galleries:

\begin{corollary}\label{gallerpathcoro}
Write a dominant coweight $\lam=\om_{i_1}+\ldots+\om_{i_r}$ as a sum of fundamental coweights,
write $\underline{\lam}$ for this ordered decomposition. Let ${\mathcal P}_{\underline \lam}$ be the associated
path model of LS-paths of shape $\underline{\lam}$ defined in \cite{L1}. The associated one-skeleton galleries 
(same procedure as in Corollary~\ref{pathgalleryfundamental}) are precisely the LS-galleries of the same type as 
$\gamma_{\om_{i_1}}*\ldots*\gamma_{\om_{i_r}}$. 
\end{corollary}

In fact, the notion of a {\it defining chain for LS-paths}
introduced by Lakshmibai, Musili and Seshadri coincides in this case
with the notion of a defining chain for the associated gallery.
As an immediate consequence of Theorem~\ref{Lpolynomialformula} and Proposition~\ref{dimensioninequality}
we get the following character formula. In combination with Corollary~\ref{gallerpathcoro}, this provides a 
geometric proof of the 
path character formula, first conjectured by Lakshmibai 
(see for example \cite{LS2}) and proved in \cite{L1}:

\begin{corollary}\label{characterformula}
$\chara V(\lam) = \sum_{\delta} e^{target(\delta)}$, where the sum runs over all
LS-galleries of the same type as $\gamma_\lam$. 
\end{corollary}

\proof
The formula in Theorem~\ref{Lpolynomialformula} and the results above 
show that the highest power of $q$ in the Laurent polynomial $L_{\lam,\mu}$ 
is $\langle\lam+\mu,\rho\rangle$, and the coefficient of the highest power is the number
of LS-galleries having $\mu$ as a target. Since $P_\lam\rightarrow s_\lam$ for $q\rightarrow\infty$,
this proves the character formula.
\qed
\medskip

Let us now consider some of the special cases discussed in section~\ref{sec:increaseconcat},
these cases occur already in \cite{LMS}. The question why for some enumeration of the fundamental
weights the combinatorics for tableaux becomes suddenly much easier than for other enumerations
seems to have a geometric answer: because for special orderings locally minimal and globally
minimal are equivalent conditions for one-skeleton galleries.

\subsection{LS-tableaux and LS-galleries}

It remains to describe the semi-standard tableaux corresponding to LS-galleries, we call these LS-tableaux. Since the condition of being folded by a sequence of LS-folds is equivalent to the condition for LS-paths,
these tableaux can be found in \cite{L2}, we refer here to a slightly different but equivalent description by Lakshmibai.
\begin{proposition}
\begin{itemize}
\item[{\it i)}] In type ${\tt A}_n$ all semistandard tableaux are LS-tableaux.
\item[{\it ii)}] In type ${\tt B}_n$ a  semistandard tableau $\mathcal T$ is an LS-tableau if and only if 
the following holds: each pair of columns $(C_1,C_2)$ corresponding to a gallery for a non-minuscule weight satisfies the conditions for an admissible pair in Proposition B1 of \cite{Lakshmi86}.

\item[{\it iii)}] In type ${\tt C}_n$ a semistandard tableau $\mathcal T$ is an LS-tableau if and only if 
the following holds: each pair of columns $(C_1,C_2)$ corresponding to a gallery for a non-minuscule weight satisfies the conditions for an admissible pair in Proposition C1 of \cite{Lakshmi86}.

\end{itemize}
\end{proposition}

\begin{exam}
The tableau of type $\mathtt B_3$ in Example \ref{exTableaux} is semistandard but not LS.
\end{exam}


\bigskip
\emph{Acknowlegments.} We would like to thank Venkatraman Lakshmibai for helping us improving the manuscript. The first author would also like to thank Guy Rousseau for very useful comments and acknowledge financial support by the ANR as member of the project ANR-09-JCJC-0102-01. The second author acknowledge also financial support by the priority program SPP 1388 of the DFG. Both authors thank the Hausdorff research Institute for Mathematics for the hospitality during the Trimester Program ``On the interaction of representation theory with geometry and combinatorics''.

\end{document}